\documentclass[amstex,11pt,reqno,english]{amsart}
\usepackage{amsmath,amsfonts,amssymb,amsthm, bbm, color, enumitem, hyperref}

\newtheorem{lemma}{Lemma}

\newtheorem{thm}{Theorem}
\newtheorem{example}{Example}
\newtheorem{assumption}{Assumption}
\newtheorem{corollary}{Corollary}
\newtheorem{remark}{Remark}
 
\numberwithin{equation}{section}
\allowdisplaybreaks
\begin{document}
\title[It\^o-Taylor expansion and $\gamma\in\{n/2:n \in\mathbb{N}\}$ order scheme for SDE\lowercase{w}MS]{On It\^o-Taylor expansion for stochastic differential equations with Markovian switching and its application in $\gamma\in\{n/2:n \in\mathbb{N}\}$-order scheme}
\author[T. Kumar and C. Kumar]{Tejinder Kumar and Chaman Kumar}
\address{Department of Mathematics\\
Indian Institute of Technology Roorkee, India}
\email{tkumar@ma.iitr.ac.in, chaman.kumar@ma.iitr.ac.in}
\maketitle

\begin{abstract}
The coefficients of the stochastic differential equations with Markovian switching  (SDEwMS) additionally depend on a Markov chain and there is no notion of differentiating such functions with respect to the  Markov chain.
In particular, this  implies that  the It\^o-Taylor expansion for SDEwMS is not a straightforward extension of the It\^o-Taylor expansion  for stochastic differential equations (SDEs). 
Further, higher-order numerical schemes for SDEwMS are not available in the literature, perhaps because of the absence of the It\^o-Taylor expansion. 
In this article, first, we overcome these challenges and  derive the It\^o-Taylor expansion for SDEwMS, under some suitable regularity assumptions on the coefficients, by developing new techniques. 
Secondly, we demonstrate an application of our first result on the It\^o-Taylor expansion  in the numerical approximations of SDEwMS. 
We derive  an explicit scheme for SDEwMS using  the It\^o-Taylor expansion and show that the strong rate of convergence of our scheme is equal to $\gamma\in\{n/2:n\in\mathbb{N}\}$ under some  suitable  Lipschitz-type conditions on the coefficients and their derivatives.  
It is worth mentioning that designing and analysis of  the It\^o-Taylor expansion and the  $\gamma\in\{n/2:n\in\mathbb{N}\}$-order scheme for SDEwMS become much more complex and involved due to the entangling of continuous dynamics and discrete events.
Finally, our results  coincide with the corresponding results on SDEs when the state of the Markov chain is a singleton set. 
\end{abstract}
\tableofcontents
\section{Introduction}
\label{Intro}
The stochastic differential equations with Markovian switching (SDEwMS)  are widely used in many real-life situations, see, for example,  \cite{bao2016permanence}, \cite{mao2006stochastic}, \cite{mao2006hierarchical},  \cite{yin1994hybrid}, \cite{zhang1998nonlinear}, \cite{zhang2001stock} and the references therein. 
The existence and uniqueness of the solution of SDEwMS are well-known in the literature, see  \cite{dareiotis2016tamed} and \cite{mao2006stochastic}. 
However, the explicit solutions of such equations are often unknown, which makes the numerical approximations of SDEwMS  an important subject of investigation. 
The numerical schemes of order $0.5$ for SDEwMS, namely the  Euler scheme,  tamed Euler scheme and implicit Euler scheme, have received significant attention in the literature, see   \cite{mao2006stochastic}, \cite{nguyen2017pathwise}, \cite{dareiotis2016tamed} and \cite{nguyen2018tamed}. 
Recently,  order $1.0$ schemes such as (tamed) explicit Milstein-type scheme   for SDEwMS and their strong rate of convergence  have been studied in \cite{kumar2020tamedmilstein}, \cite{kumar2021milstein} and \cite{nguyen2017milstein}. 
However, no attention is given to the investigation of the higher-order numerical schemes for SDEwMS, perhaps because of the non-availability of the  It\^o-Taylor expansion for functions depending on a Markov chain. 
The derivation of the It\^o-Taylor expansion for SDEwMS is not a straightforward extension of the corresponding results on SDEs (see \cite{kloeden1992numerical}), mainly because the coefficients of SDEwMS  depend additionally on a Markov chain and there is no notion of derivative of a Markov chain dependent function with respect to the Markov chain. 

In this article, we present novel strategies to address these challenges to derive and demonstrate the It\^o-Taylor expansion for SDEwMS. 
More precisely, in the derivation of the It\^o-Taylor expansion for the traditional SDEs, as discussed in \cite{kloeden1992numerical}, authors  deal with two types of integrals, namely, Riemann integral and It\^o's integral.
However, to derive the It\^o-Taylor expansion for SDEwMS, we have to deal with one more form of integral, which is the integral with respect to the optional process $\{[M_{i_0 k_0}](t):t\in[0,T]\}$ (see Subsection \ref{sub:Martingale with MS}) associated with the Markov chain.
Thus, the techniques developed for SDEs in \cite{kloeden1992numerical} can not be extended in a  straightforward manner to SDEwMS,  which necessitates  the development of novel strategies. 
Indeed,  multi-indices, multiple integrals, hierarchical and remainder sets are defined by taking into consideration the new integrals with respect to $\{[M_{i_0 k_0}](t):t\in[0,T]\}$ which brings additional complexity in the derivation and proofs of the main results on the It\^o-Taylor expansion and the explicit  $\gamma\in\{n/2:n\in\mathbb{N}\}$-order scheme for SDEwMS. 
Here, it is worth mentioning that designing and analysis become difficult because of the random switching of the Markov chain and entangling of the continuous dynamics of the state and discrete events arising due to switching of the Markov chain. 
Further, if the state space of the Markov chain is a singleton set, then SDEwMS become  SDEs. In such a case, our It\^o-Taylor expansion reduces to the classical It\^o-Taylor expansion for SDEs given in \cite{kloeden1992numerical}.
We provide an application of the It\^o-Taylor expansion in the derivation of higher-order numerical schemes for SDEwMS. 
We  show  that the strong rate of convergence of the general numerical scheme for SDEwMS is  $\gamma\in\{n/2: n \in \mathbb{N}\}$. 
For a Markov chain with singleton state space, \textit{i.e.}, when SDEwMS reduce to SDEs,  our $\gamma\in\{n/2: n \in \mathbb{N}\}$-order scheme is the same as the $\gamma\in\{n/2: n \in \mathbb{N}\}$-order scheme  discussed in \cite{kloeden1992numerical}. 
In this case,  the regularity requirement on the coefficients is much weaker in our setting when compared with the correspodning results  for SDEs given in \cite{kloeden1992numerical}.  

The paper is organized as follows. 
The formulation and preliminaries of the article are given in Section \ref{sec:formulation}.
In Section \ref{sec:notations}, notations and definitions are given, which are used throughout the paper. 
In Section \ref{Main Result}, the main results on the It\^o-Taylor expansion and its application in $\gamma\in\{n/2:n\in\mathbb{N}\}$-order numerical schemes for SDEwMS are stated. 
The technical lemmas required to prove the  It\^o-Taylor expansion for SDEwMS, including the proof of the It\^o-Taylor expansion, are given in Section \ref{sec:ito}.
In the last section, we explain the derivation of the $\gamma\in \{n/2:n\in\mathbb{N}\}$-order numerical scheme and establish its moment stability and the strong rate of convergence.  

\section{Formulation and preliminaries}
\label{sec:formulation}
Let $(\Omega,\mathcal{F}, P)$ be a  complete probability space. 
Assume that   $W:=\{(W^1(t),W^2(t), \ldots,W^m(t)): t\geq 0\}$ is  an $\mathbb{R}^m$-valued standard Wiener process and denote the natural filtration of $W$ by $\mathbb{F}^W:=\{\mathcal{F}^W_t:t\geq 0\}$.
For a fixed $m_0\in \mathbb{N}$,  define a set $\mathcal{S}:=\{1,\ldots,m_0 \}$ and an $m_0\times m_0$ matrix  $Q:=\{q_{i_0k_0}: i_0, k_0 \in \mathcal{S}\}$ such that $q_{i_0i_0}=-\displaystyle \sum_{k_0\neq i_0 \in \mathcal{S}} q_{i_0k_0}$ for any $i_0\in \mathcal{S}$ where $q_{i_0k_0}\geq 0$  for $i_0\neq k_0$. 
Now, consider a continuous-time Markov chain $\alpha:=\{\alpha(t):t\geq 0\}$ with  state space $\mathcal{S}$ and generator  $Q$.
Clearly,  the transition probability of $\alpha$  is given by, 
\begin{align} 
P(\alpha(t+\delta)=k_0|\alpha(t)=i_0)=
\begin{cases}
q_{i_0k_0}\delta+o(\delta), &\text{if } i_0\neq k_0, \\
1+q_{i_0k_0}\delta+o(\delta), &\text{if } i_0= k_0,
\end{cases} \notag
\end{align}
for  any $t \geq 0$, $i_0, k_0 \in \mathcal{S}$ where $\delta>0$. 
Also, let $\mathbb{F}^\alpha:=\{\mathcal{F}^\alpha_t:t\geq 0\}$ be the natural filtration of $\alpha$. 
We assume that $\alpha$ is independent of $W$. 
Define $\mathcal{F}_t=\mathcal{F}_t^W \vee \mathcal{F}_t^\alpha$ for any $t\geq 0$ and equip the probability space $(\Omega,\mathcal{F}, P)$ with the filtration $\mathbb{F}:=\{\mathcal{F}_t:t\geq 0\}$.
For $T>0$,  consider a $d$-dimensional  stochastic differential equation with Markovian switching (SDEwMS)  on $(\Omega,\mathcal{F}, \mathbb{F}, P)$,  given by, 
\begin{align} \label{eq:sdems}
X(t)=X_0+\int^t_0b(X(s),\alpha(s))ds+\int^t_0 \sigma(X(s),\alpha(s))dW(s) 
\end{align} 
almost surely  for any $t\in[0,T]$ where initial value $X_0\in \mathbb{R}^d$ is an $\mathcal{F}_0$-measurable random variable and $b:\mathbb{R}^d\times \mathcal{S} \mapsto \mathbb{R}^d$ and $\sigma:\mathbb{R}^d\times \mathcal{S} \mapsto \mathbb{R}^{d \times m}$ are measurable functions. 

We make the following assumptions for existence, uniqueness and moment stability of SDEwMS \eqref{eq:sdems}. 
\begin{assumption}
\label{ass:initial data}
$E|X_0|^{2}, E|Y_0|^{2} < \infty$ and $E|X_0-Y_0|^2\leq L h^{2\gamma}$ where  $L>0$ is a constant. 
\end{assumption}   

\begin{assumption}
\label{ass: b sigma lipschitz}
There exists a  constant $L>0$ such that,
\begin{align*}
|b^k(x,i_0)-b^k(y,i_0)|+|\sigma^{(k,j)}(x,i_0)-\sigma^{(k,j)}(y,i_0)| \leq L|x-y|
\end{align*}
for all $i_0\in \mathcal{S}$, $k\in\{1,\ldots,d\}$, $j\in\{1,\ldots,m\}$ and $x,y\in \mathbb{R}^d$.
\end{assumption}

The result on the existence and uniqueness of the strong solution for SDEwMS  is stated in the following theorem and the proof can be found in \cite{mao2006stochastic}, see Theorem 3.1.3.
\begin{thm}\label{thm:true moment}
Let Assumptions \ref{ass:initial data} and \ref{ass: b sigma lipschitz} be satisfied. Then, there exists a unique continuous solution $\{X(t): \{t\in [0,T]\}\}$ of SDEwMS \eqref{eq:sdems}. Moreover, the following hold, 
\begin{align*}
E\Big(\sup_{t\in [0,T]}|X(t)|^{p}\Big|\mathcal{F}_T^{\alpha}\Big)& \leq C
\\
E\Big( \sup_{t\in [s,s+h]}|X(t)-X(s)|^p\Big|\mathcal{F}_T^{\alpha}  \Big)& \leq Ch^{p/2} 
\end{align*}
where the positive constant $C$ is independent of $h$.
\end{thm}
\subsection{Martingales associated with the Markov chain}\label{sub:Martingale with MS}
Following  \cite{nguyen2017milstein}, we define 
 \begin{align}
[M_{i_0k_0}](t):=\sum_{s\in [0, t]}\mathbbm{1}{\{\alpha(s-)=i_0\}} &\mathbbm{1}{\{\alpha(s)=k_0\}},\quad \langle M_{i_0k_0}\rangle(t):=\int_0^t q_{i_0k_0}\mathbbm{1}{\{\alpha(s-)=i_0\}} ds\nonumber
\\
M_{i_0k_0}(t)&:= [M_{i_0 k_0}](t)-\langle M_{i_0k_0}\rangle(t)  \nonumber
\end{align}
almost surely for any $t\in[0,T]$ and $i_0\neq k_0 \in \mathcal{S}$.
Clearly,  $M_{i_0k_0}(0)=0$ (a.s.) and $\{M_{i_0k_0}(t):  t\in[0,T]\}$ is a purely discontinuous square integrable martingale with respect to the filtration $\mathbb{F}^\alpha$. 
Also, $\{[M_{i_0 k_0}](t):t\in[0,T]\}$ is an optional process and  $\{\langle M_{i_0k_0}\rangle(t):t\in[0,T]\}$ is its predictable quadratic variation process. 
Moreover, the following orthogonality relations (of quadratic covariances) hold, 
\begin{align*}
[W^i, W^j]=0 \,(i\neq j), \, [M_{i_0k_0}, W^i]=0, \mbox{ and } [M_{i_0k_0}, M_{i_1k_1}]=0 \,\,((i_0k_0)\neq (i_1k_1))
\end{align*}
for any $i,j \in \{1,2,\ldots,m\}$ and $i_0, k_0, i_1, k_1 \in \mathcal{S}$. 
For convenience, we  take $M_{i_0i_0}(t)=0$ for  any $i_0\in\mathcal{S}$ and $t\in[0,T]$.  

The proof of the following lemma appears in \cite{nguyen2017milstein} (see Lemma 4.1). 
\begin{lemma}\label{lem:rateMS}
Let $q=\max \{-q_{i_0i_0}:i_0 \in \mathcal{S}\}$ and $N^{(s,t]}$ denotes the number of jumps of the Markov chain $\alpha$ in the interval $(s,t]$ for any $s<t\in[0,T]$.  Then, 
\begin{enumerate}[label=(\alph*)]
\item  $P(N^{(s,t]}\geq N)\leq q^N(s-t)^N$ whenever $N\geq 1$, and 
\item $EN^{(s,t]}\leq C (t-s)$ whenever $t-s<1/(2q)$ where $C>0$ is a constant independent of $t-s$. 
\end{enumerate}
\end{lemma}

\section{Notations and definitions}\label{Notations and definitions} 
\label{sec:notations}
For sets $A$ and $B$, define $A\setminus B:=\{x:x\in A \mbox{ and } x\notin B\}$, $A\cap B:=\{x:x\in A,x\in B\}$ and $A\cup B:=\{x:x\in A \: \text{ or }\: x\in B\}$.
$A^{(i,j)}$ and $A^{(j)}$ stand for the $(i,j)$-th element and the $j$-th column of a matrix $A$ respectively and the $k$-th element of a vector $v$ is denoted by $v^k$. 
$xy$ denotes the inner product of vectors  $x$ and $y$. 
The  indicator function of a set $A$ is denoted by $\mathbbm{1}A$. 
$\varphi$ stands for an empty set and by convention, the sum over an empty set is taken to be zero. 
$N^{(s,t]}$ is the no. of jump and $\tau_1,\ldots,\tau_{N^{(s,t]}}$ are the jump times of the Markov chain $\alpha$ in the interval $(s,t]$ for any $s<t\in [0,T]$.
The general constant is denoted by $C$ which can vary from place to place and does not depend on the step-size. 
\subsection{Multi-indices} \label{sub:Multi-indices}
We call a row vector $\beta=(j_1,\ldots,j_l)$, a \textit{multi-index} of length $l(\beta)=l \in\mathbb{N}$, when $ j_i\in\{N_1, \ldots,N_{\mu},\bar{N}_{\mu },0,1,\ldots,m\}$ for all $i\in\{1,\ldots,l\}$ and satisfies the following condition,
\begin{itemize}
\item[(a)] if $j_{i}\in\{N_1,\ldots,N_{\mu },\bar{N}_{\mu }\}$, then $j_{i-1}\notin\{N_1,\ldots,N_{\mu },\bar{N}_{\mu }\}$ for all $ i\in \{2,\ldots,l\}$
\end{itemize} 
where  $\mu\in\mathbb{N}$ is fixed. 
Here, $N_i$ signifies the number $i$ and $\bar{N}_i$, the number $i+1$ for any $i\in\mathbb{N}$.
A \textit{multi-index $\beta$ of length zero} is denoted   by $\nu$, \textit{i.e.}, $l(\nu)=0$.
Here,  each component of multi-index $\beta$ represents a type of integral in multiple integrals which becomes more clear when we define multiple integrals in  Subsection \ref{sub:multiple integral}. 
Furthermore,   the positive  integer $\mu$ is related to the  strong rate $\gamma\in \{n/2,n\in\mathbb{N}\}$ of convergence of the numerical schemes considered in this paper (see Subsection \ref{sub:scheme}).

 Also, assume that $\mathcal{M}:=\{\beta: l(\beta)\in\mathbb{N}\cup\{0\}\}$ denotes the set of all multi-indices of arbitrary length.
We divide the set $\mathcal{M}$ of all multi-indices into disjoint sets $\mathcal{M}_1$, $\mathcal{M}_2$ and $\mathcal{M}_3$ as follows, 
\begin{itemize}
\item[] $\mathcal{M}_1:=\{(j_1,\ldots,j_l)\in\mathcal{M}:j_i\in\{0,1,\ldots,m\}\forall\, i\in\{1,\ldots,l\}\}\cup \{\nu\}$,
\item[] $\mathcal{M}_2:=\{(j_1,\ldots,j_l)\in\mathcal{M}\setminus\mathcal{M}_1:j_1\in\{0,1,\ldots,m\}\}$,
\item[] $\mathcal{M}_3:=\{(j_1,\ldots,j_l)\in\mathcal{M}\setminus\mathcal{M}_1:j_1\in\{N_1,\ldots,N_{\mu},\bar{N}_{\mu}\}\}$.
\end{itemize}
 
For any $\beta=(j_1,\ldots,j_l)\in\mathcal{M}\setminus\{\nu\}$, the deletion of its first component  is represented by $-\beta=(j_2,\ldots,j_l)$ and the deletion of its last component  is represented by  $\beta-=(j_1,\ldots,j_{l-1})$.
Notice that $-\beta=\beta-=\nu$ for any $\beta\in\mathcal{M}$ of length $l(\beta)=1$.

For $\beta=(j_1,\ldots,j_k), \bar{\beta}=(\bar{j_1},\ldots,\bar{j_l}) \in \mathcal{M}$, the \textit{concatenation operation} $\star$  is defined by 
\begin{itemize}
\item[] $\beta\star\bar{\beta}:=(j_1,\ldots,j_k,\bar{j_1},\ldots,\bar{j_l})\in \mathcal{M}$ with $\nu\star\bar{\beta}=\bar{\beta}$ and $\beta\star\nu=\beta$.
\end{itemize}

Also, for $\beta=(j_1,\ldots,j_l) \in\mathcal{M}$, 
\begin{itemize}
\item[] $n(\beta)$ :  number of components of $\beta$ which are equal to $k$ where $k\in \{1,\ldots,m\}$,
\item[] $\bar{n}(\beta)$ : number of components of $\beta$ which are equal to $0$,
\item[] $[n](\beta)$ : number of components of $\beta$ which are equal to $k$ where  $k \in \{N_1,\ldots,N_{\mu},\bar{N}_{\mu}\}$,
\item[]$\eta(\beta)=n(\beta)+2\bar{n}(\beta)+\mu_{\max}(\beta)$ where $\mu_{\max}(\beta)=\displaystyle \max_{i \in \{1,\ldots, l\}}\{j_i$: $j_i\in\{N_1,\ldots,N_{\mu}$, $\bar{N}_{\mu}\}$ for any $i\in\{1,\ldots,l\}$\}. Clearly, $\eta(\nu)=0$.
\end{itemize}
Notice that $\eta(\beta)$ may be equal for two different $\beta\in\mathcal{M}$ whose length is not same. For example, $\eta((1,N_1))=\eta((N_1,1,N_1))=2$.
Further, $\{\beta\in\mathcal{M}:\eta(\beta)<0\}$ is an empty set.
\begin{example}
Let $m=4$ and consider a multi-index   $\beta=(0,N_2,2,1,N_3,0)$ of length $6$. 
Clearly, $\beta$ satisfies conditions (a) mentioned above. 
Then, $-\beta=(N_2,2,1,N_3,0)$ and $\beta-=(0,N_2,2,1,N_3)$.
Also, $n(\beta)=2$, $\bar{n}(\beta)=2$ and $\mu_{max}(\beta)=\max\{N_2,N_3\}=3$ which yields $\eta(\beta)=9$. 
Further, for  $\bar{\beta}=(4,0,0,N_3,1,\bar{N}_1)$,  $$\beta\star\bar{\beta}=(0,N_2,2,1,N_3,0,4,0,0,N_3,1,\bar{N}_1) \in \mathcal{M}.$$ 
Clearly, $\beta\star\bar{\beta}$ satisfies conditions (a) and is a multi-index of length $12$.
\end{example}
\subsection{Operators}
\label{sub:Operators}
The following operators are used throughout this article,
\begin{align*}
L^0_{i_0} & :=\sum^d_{k=1}b^k(\cdot,i_0)\frac{\partial }{\partial x^k}+\frac{1}{2}\sum^d_{k,l=1}\sum^m_{j=1}\sigma^{(k,j)}(\cdot,i_0)\sigma^{(l,j)}(\cdot,i_0)\frac{\partial^2 }{\partial x^k\partial x^l},
\\
L^j_{i_0}& :=\sum^d_{k=1}\sigma^{(k,j)}(\cdot,i_0)\frac{\partial }{\partial x^k}, \,\, j\in\{1,2,\ldots,m\}
\end{align*}
for any $i_0\in \mathcal{S}$.
Further, for $\beta\in\{(j_1,\ldots,j_l)\in\mathcal{M}:j_i\in\{0,1,\ldots,m\}\forall i\in\{1,\ldots,l\}\}$
\begin{align*}
J^{\beta}_{i_0}&:=L^{j_1}_{i_0}\cdots L^{j_l}_{i_0}
\\
J^{\beta}_{i_0k_0}&:=J^{\beta}_{k_0}-J^{\beta}_{i_0}
\end{align*}
for any $i_0,k_0\in \mathcal{S}$. 
If $\beta=\nu$, then $J_{i_0}^{\beta}$ is taken as the identity operator. 
\subsection{Multiple integrals}\label{sub:multiple integral}
Let $\beta=(j_1,\ldots,j_l)\in\mathcal{M}$ and $f:\mathbb{R}^d\times\mathcal{S}\mapsto\mathbb{R}$ be $(n(\beta)+2\bar{n}(\beta))-$times continuously differentiable function in the first component.
The multiple integral $I_{\beta}[f(X(\cdot),\alpha(\cdot))]_{s,t}$ can be defined recursively as,
\begin{align}
&I_{\beta}[f(X(\cdot),\alpha(\cdot))]_{s,t}\notag
\\
&:= 
     \begin{cases}
     f(X(t),\alpha(t))& \mbox{if } l(\beta)=0, \mbox{\textit{i.e.}, } \beta=\nu 
     \\
       \displaystyle \int^t_s I_{\beta-}[L^{j_l}_{\alpha(\cdot)}f(X(\cdot),\alpha(\cdot))]_{s,s_1}dW^{j_l}(s_1) &   \mbox{if }  j_l\in\{0,1,\ldots,m\} 
       \\ 
       \displaystyle \sum_{i_0\neq k_0}  \int^t_s\mathbbm{1}\{N^{(s,t]}=r \}  I_{\beta-}[f(X(\cdot),k_0)-f(X(\cdot),i_0)]_{s,s_1}d[M_{i_0k_0}](s_1) &\mbox{if }  j_l\in\{N_r\}, r\in\{1,\ldots,\mu\}
      \\ 
     \displaystyle \sum_{i_0\neq k_0}  \int^t_s\mathbbm{1}\{N^{(s,t]}>\mu \} I_{\beta-}[f(X(\cdot),k_0)-f(X(\cdot),i_0)]_{s,s_1}d[M_{i_0k_0}](s_1) & \mbox{if } j_l=\bar{N}_{\mu }
     \end{cases}  \notag
\end{align}
almost surely for all $s<t\in[0,T]$ where $dW^{0}(s_1):=ds_1$. 
Notice that if the state of the Markov chain $\alpha$ is fixed in $f$, say $i_0\in\mathcal{S}$, then 
\begin{align}
I_{\beta}[f(X(\cdot),i_0)]_{s,t}:=& 
     \begin{cases}
     f(X(t),i_0)& \mbox{if } l(\beta)=0 
       \\ 
       \displaystyle \int^t_s I_{\beta-}[L^{j_l}_{\alpha(\cdot)}f(X(\cdot),i_0)]_{s,s_1}dW^{j_l}(s_1) &   \mbox{if }  j_l\in\{ 0,1,\ldots,m\} 
     \end{cases}  \notag
\end{align}
almost surely for any $s<t\in[0,T]$. 
Furthermore, we define, 
\begin{align*}
&I_{\beta}[f(X(s),\alpha(\cdot))]_{s,t}\notag
\\
&:= 
     \begin{cases}
     f(X(s),\alpha(s))& \mbox{if } l(\beta)=0, \mbox{\textit{i.e.}, } \beta=\nu 
     \\
       \displaystyle \int^t_s I_{\beta-}[L^{j_l}_{\alpha(\cdot)}f(X(s),\alpha(\cdot))]_{s,s_1}dW^{j_l}(s_1) &   \mbox{if }  j_l\in\{0,1,\ldots,m\} 
       \\ 
       \displaystyle \sum_{i_0\neq k_0}  \int^t_s\mathbbm{1}\{N^{(s,t]}=r \}  I_{\beta-}[f(X(s),k_0)-f(X(s),i_0)]_{s,s_1}d[M_{i_0k_0}](s_1) &\mbox{if }  j_l\in\{N_r\}, r\in\{1,\ldots,\mu\}
      \\ 
     \displaystyle \sum_{i_0\neq k_0}  \int^t_s\mathbbm{1}\{N^{(s,t]}>\mu \} I_{\beta-}[f(X(s),k_0)-f(X(s),i_0)]_{s,s_1}d[M_{i_0k_0}](s_1) & \mbox{if } j_l=\bar{N}_{\mu }
     \end{cases}  \notag
\end{align*}
almost surely for all $s<t\in[0,T]$. 
As before, if the state of $\alpha$ is fixed in $f$, say $i_0\in\mathcal{S}$, then 
\begin{align}
I_{\beta}[f(X(s),i_0)]_{s,t}:=& 
     \begin{cases}
     f(X(s),i_0)& \mbox{if } l(\beta)=0 
       \\ 
       \displaystyle \int^t_s I_{\beta-}[L^{j_l}_{\alpha(\cdot)}f(X(s),i_0)]_{s,s_1}dW^{j_l}(s_1) &   \mbox{if }  j_l\in\{ 0,1,\ldots,m\} 
     \end{cases}  \notag
\end{align}
almost surely for any $s<t\in[0,T]$. 
Clearly, for $\beta\in\{(j_1,\ldots,j_l)\in\mathcal{M}:j_i\notin\{N_1,\ldots,N_{\mu},\bar{N}_{\mu}\}\forall\, i\in\{1,\ldots,l\}\}$,
\begin{align}
I_{\beta}[f(X(s),\alpha(s))]_{s,t}:=& 
     \begin{cases}
     f(X(s),\alpha(s))& \mbox{if } l(\beta)=0 
       \\ 
       \displaystyle \int^t_s I_{\beta-}[L^{j_l}_{\alpha(s)}f(X(s),\alpha(s))]_{s,s_1}dW^{j_l}(s_1) &   \mbox{if }  j_l\in\{ 0,1,\ldots,m\} 
     \end{cases}  \notag
\end{align} 
almost surely for any $s<t\in[0,T]$. 
We assume that  all of the above multiple integrals are well defined.
\begin{remark}
When the state space $\mathcal{S}$ of the Markov chain $\alpha$ is a singleton set, then the SDEwMS \eqref{eq:sdems} reduces to an  SDE and the integrals corresponding to the components $\{N_1,\ldots,N_{\mu},\bar{N}_{\mu}\}$ in the multiple integrals vanish. 
In such a situation, the multi-indices and multiple integrals defined above reduce to the case as discussed in Section $5.2$ of \cite{kloeden1992numerical}. 
The components $\{N_1,\ldots,N_{\mu},\bar{N}_{\mu}\}$ in  multi-indices  and the corresponding multiple integrals appear due to the switching of the Markov chain. 
\end{remark}
\begin{example}
Let $\beta=(N_1,k,N_1)\in\mathcal{M}$ for a $k\in\{1,\ldots,m\}$.
 Then, using the definitions of multiple integrals, one has
\begin{align*}
 &\hspace{-0.32cm} I_{(N_1,k,N_1)}[f(X(\cdot),\alpha(\cdot))]_{s,t}=\sum_{i_0\neq k_0}\int_s^t\mathbbm{1}\{N^{(s,t]}=1 \} I_{(N_1,k)} [f(X(\cdot),k_0)-f(X(\cdot),i_0)]_{s,s_1}  d[ M_{i_0 k_0}](s_1)
\\
&=\sum_{i_0\neq k_0}\int_s^t\mathbbm{1}\{N^{(s,t]}=1 \}\int_s^{s_1} I_{(N_1)} [L^k_{\alpha(\cdot)}f(X(\cdot),k_0)-L^k_{\alpha(\cdot)}f(X(\cdot),i_0)]_{s,s_2}  dW^k(s_2)d[ M_{i_0 k_0}](s_1)
\\
&=\sum_{i_0\neq k_0}\int_s^t\mathbbm{1}\{N^{(s,t]}=1 \}\int_s^{s_1}\sum_{i_1\neq k_1}\int_{s}^{s_2}\mathbbm{1}\{N^{(s,s_2]}=1 \} 
\\
&\qquad I_{\nu}[(L^k_{k_1}f(X(\cdot),k_0)-L^k_{i_1}f(X(\cdot),k_0))-(L^k_{k_1}f(X(\cdot),i_0)-L^k_{i_1}f(X(\cdot),i_0))]_{s,s_3}
\\
&\qquad d[ M_{i_1 k_1}](s_3) dW^k(s_2)d[ M_{i_0 k_0}](s_1)
\\
=&\sum_{i_0\neq k_0}\int_s^t\mathbbm{1}\{N^{(s,t]}=1 \}\int_s^{s_1}\sum_{i_1\neq k_1}\int_{s}^{s_2}\mathbbm{1}\{N^{(s,s_2]}=1 \} 
\\
&\qquad ((L^k_{k_1}f(X(s_3),k_0)-L^k_{i_1}f(X(s_3),k_0))-(L^k_{k_1}f(X(s_3),i_0)-L^k_{i_1}f(X(s_3),i_0)))
\\
&\qquad d[ M_{i_1 k_1}](s_3) dW^k(s_2)d[ M_{i_0 k_0}](s_1)
\end{align*}
almost surely for any $s<t\in[0,T]$.
\end{example}
\begin{example}
Let $\beta=(0,N_1)\in\mathcal{M}$.
 Then, the definition of multiple integrals yields,
 \begin{align*}
 I_{(0,N_1)}[f(&X(s),\alpha(\cdot))]_{s,t}=\sum_{i_0\neq k_0}\int_s^t\mathbbm{1}\{N^{(s,t]}=1 \} I_{(0)} [f(X(s),k_0)-f(X(s),i_0)]_{s,s_1}  d[ M_{i_0 k_0}](s_1)
 \\
 =&\sum_{i_0\neq k_0}\int_s^t\mathbbm{1}\{N^{(s,t]}=1 \}\int_s^{s_1}I_{\nu} [L^0_{\alpha(\cdot)}f(X(s),k_0)-L^0_{\alpha(\cdot)}f(X(s),i_0)]_{s,s_2}ds_2  d[ M_{i_0 k_0}](s_1)
\\
=&\sum_{i_0\neq k_0}\int_s^t\mathbbm{1}\{N^{(s,t]}=1 \}\int_s^{s_1}(L^0_{\alpha(s)}f(X(s),k_0)-L^0_{\alpha(s)}f(X(s),i_0))ds_2  d[ M_{i_0 k_0}](s_1)
 \end{align*}
 almost surely for any $s<t\in[0,T]$.
\end{example}
\subsection{Hierarchical and remainder sets}\label{sub:hierarchical and reminder}
A subset $\mathcal{A}$ of $\mathcal{M}$ is called a \textit{hierarchical set} if $\displaystyle \sup_{\beta\in\mathcal{A}}l(\beta)<\infty$   and $-\beta\in\mathcal{A}$  for every  $\beta\in\mathcal{A}\setminus\{\nu\}$. 

The \textit{remainder set} $\mathcal{B}(\mathcal{A})$ corresponding to a hierarchical set $\mathcal{A}$ is  define as,
\begin{itemize}
\item[] $\mathcal{B}(\mathcal{A}):=\{\beta\in\mathcal{M}\setminus\mathcal{A}:-\beta\in\mathcal{A}\}$. 
\end{itemize}
If $\mathcal{A}=\varphi$, then $\mathcal{B}(\mathcal{A})=\{\nu\}$. 
 \begin{example}\label{ex:heirarchical}
 Let $\mathcal{A}:=\{\beta\in\mathcal{M}: \eta(\beta)\leq 2 \}$ with $\mu=3$. 
 Then
\begin{itemize}
\item[] $\mathcal{A}=\{\nu,(N_1),(N_2),(0),(k),(k,N_1),(N_1,k),(k_1,k),(N_1,k,N_1);k,k_1\in\{1,\ldots,m\}\}$
\end{itemize} 
 is the hierarchical set and corresponding remainder set is
\begin{itemize} 
\item[]  \hspace{-1cm} $\mathcal{B}(\mathcal{A})=\{(N_3),(\bar{N}_3), (0,N_1), (0,N_2), (k,N_2), (N_1,0),(N_2,0),(N_3,0),(\bar{N}_3,0),(0,0),(k,0),(N_2,k),$\newline
$(N_3,k),(\bar{N}_3,k), (0,k),(N_2,k,N_1),(N_3,k,N_1),(\bar{N}_3,k,N_1),(0,k,N_1),(k_1,k,N_1),(0,N_1,k),$\newline
$(k_1,N_1,k), (N_1,k_1,k),(N_2,k_1,k),(N_3,k_1,k),(\bar{N}_3,k_1,k),(0,k_1,k),(k_2,k_1,k),(0,N_1,k,N_1),$\newline
 $(k_1,N_1,k,N_1);k,k_1,k_2\in\{1,\ldots,m\}\}$.
\end{itemize} 
\end{example}
\section{Main results}\label{Main Result}
In this section, we introduce two key findings of this paper, namely, the It\^o-Taylor expansion for functions additionally depending on the Markov chain $\alpha$ (see Subsection \ref{sub:ItoTaylorexpansion})  and the explicit numerical scheme of arbitrary order  $\gamma\in\{n/2:n\in\mathbb{N}\}$ for SDEwMS \eqref{eq:sdems} (see Subsection \ref{sub:scheme}). 
\subsection{It\^o-Taylor expansion}\label{sub:ItoTaylorexpansion}
In the It\^{o}-Taylor expansion, the regularity   of the drift coefficient $b$ of SDEwMS \eqref{eq:sdems} is identified  with the help of following notations. 
\begin{itemize}
\item[] $\mathcal{D}_b:=\{\beta\in\mathcal{M}:\beta\star(0)\star\bar{\beta}\in\mathcal{A}\cup\mathcal{B}(\mathcal{A}),$ the components of $\bar{\beta}\in\mathcal{M}$ are not equal to $0\}$,
\item[] $\mathcal{K}_b:=\displaystyle\max_{\beta\in\mathcal{D}_b}\{n(\beta)+2\bar{n}(\beta)\}$,
\end{itemize}
 and  for the regularity of the diffusion coefficient  $\sigma$, we define, 
\begin{itemize}
\item[] $\mathcal{D}_{\sigma}=\{\beta\in\mathcal{M}:\beta\star(j)\star\bar{\beta}\in\mathcal{A}\cup\mathcal{B}(\mathcal{A}),j\in\{0,1,\ldots,m\},$ the components of $\bar{\beta}\in\mathcal{M}$  are not equal to $k$ where $k\in\{0,1,\ldots,m\}\}$,
\item[] $\mathcal{K}_{\sigma}:=\displaystyle \max_{\beta\in\mathcal{D}_{\sigma}}\{n(\beta)+2\bar{n}(\beta)\}$.
\end{itemize}
Similarly, for the regularity of the function $f$, we define, 
\begin{itemize}
\item[] $\mathcal{K}_f:=\displaystyle \max_{\beta\in\mathcal{A}\cup\mathcal{B}(\mathcal{A})}\{n(\beta)+2\bar{n}(\beta)\}$.
\end{itemize}
 Moreover, we  write the  hierarchical set $\mathcal{A}$ as $\mathcal{A}=\tilde{\mathcal{A}}\cup(\mathcal{A}\setminus\tilde{\mathcal{A}} )$ where
\begin{itemize}
\item[] $\tilde{\mathcal{A}}:=\{(j_1,\ldots,j_l)\in\mathcal{A}: j_i\in\{N_1,\ldots,N_\mu,\bar{N}_\mu\} \mbox{ for any } i\in\{1,\ldots,l\} \}$.
\end{itemize}

We make the following assumption on the coefficients of SDEwMS \eqref{eq:sdems}.
\begin{assumption} \label{as:ito}
 For a hierarchical set $\mathcal{A}$ and for any $i_0 \in \mathcal{S}$, $k \in\{1, \ldots, d\}$ and $j\in\{1,\ldots,m\}$, let $b^k(\cdot,i_0)$ and $\sigma^{(k,j)}(\cdot,i_0)$ be $\mathcal{K}_{b}-$ and $\mathcal{K}_{\sigma} -$ times continuously differentiable functions  respectively.
\end{assumption}

The following theorem is the first main result of this article which is proved in Section \ref{sec:ito}. 
\begin{thm}[\bf{It\^o-Taylor expansion}] \label{thm:Ito-Taylor}
For a hierarchical set $\mathcal{A}$, let Assumption $\ref{as:ito}$ be satisfied  and assume that $f:\mathbb{R}^d\times\mathcal{S}\mapsto \mathbb{R}$ is a $\mathcal{K}_f-$times continuously differentiable function. 
Then, for all $ s<t\in [0, T]$,
\begin{align}
f(X(t),\alpha(t))=&\sum_{\beta\in\mathcal{A}\setminus\tilde{\mathcal{A}}}I_{\beta}[f(X(s),\alpha(s))]_{s,t}+\sum_{\beta\in\tilde{\mathcal{A}}}I_{\beta}[f(X(s),\alpha(\cdot))]_{s,t}+\sum_{\beta\in\mathcal{B}(\mathcal{A})}I_{\beta}[f(X(\cdot),\alpha(\cdot))]_{s,t}   \notag
\end{align}
 almost surely, provided all the multiple integrals appearing above exist.
\end{thm}
The above expansion is characterized by the specific selection of the hierarchical set $\mathcal{A}$.
Later in Subsection \ref{sec:Derivation of scheme},  it is used to construct the numerical scheme of arbitrary order $\gamma\in\{n/2:n\in\mathbb{N}\}$, see equation \eqref{eq:gen.scheme}.
\begin{remark}
 The aforementioned expansion can be considered as a generalization of the It\^o-Taylor expansion for SDE discussed in Theorem 5.5.1 of \cite{kloeden1992numerical} in the sense that the coefficients of SDE additionally depend on a Markov chain.
 Indeed, if the state space $\mathcal{S}$ of the Markov chain $\alpha$ is a singleton set, then  SDEwMS \eqref{eq:sdems} reduces to an  SDE and $\tilde{\mathcal{A}}=\varphi$ which leads to   the It\^o-Taylor expansion for SDE. 
The terms corresponding to $\tilde{\mathcal{A}}$ in the above expansion appear due to  the presence of the Markov chain in the coefficients of the SDE. 
\end{remark}
The following two examples illustrate the It\^{o}-Taylor expansion (Theorem \ref{thm:Ito-Taylor}) for  some specific hierarchical sets $\mathcal{A}$. 
\begin{example}
Let $\mathcal{A}=\{\nu\}$. 
Then, $\tilde{\mathcal{A}}=\varphi$, $\mathcal{A}\setminus\tilde{\mathcal{A}}=\{\nu\}$ and  $ \mathcal{B}(\mathcal{A})=\{(N_1),\ldots, (N_\mu), (\bar{N}_\mu),(0),$ $(1), \ldots, (m)\}$.
Further, $\mathcal{D}_b=\{\nu\}$ which gives $\mathcal{K}_b=0$. 
Similarly, $\mathcal{K}_\sigma=0$ and $\mathcal{K}_f=2$.
Thus, by Theorem~$\ref{thm:Ito-Taylor}$,
\begin{align*}
f(X(t),\alpha(t))=&I_{\nu}[f(X(s),\alpha(s))]_{s,t}+\sum_{r=1}^{\mu}I_{(N_r)}[f(X(\cdot),\alpha(\cdot))]_{s,t}+I_{(\bar{N}_\mu)}[f(X(\cdot),\alpha(\cdot))]_{s,t}
\\
&+I_{(0)}[f(X(\cdot),\alpha(\cdot))]_{s,t}+\sum_{j=1}^mI_{(j)}[f(X(\cdot),\alpha(\cdot))]_{s,t}
\end{align*}
which on using  the definition of multiple integrals from Subsection \ref{sub:multiple integral} yields, 
\begin{align*}
f(X(t),&\alpha(t))=f(X(s),\alpha(s))+\sum_{r=1}^\mu \sum_{i_0\neq k_0}\int_s^t \mathbbm{1}\{N^{(s,t]}=r \} (f(X(s_1),k_0)-f(X(s_1),i_0))d[M_{i_0k_0}](s_1)
\\
&+ \sum_{i_0\neq k_0}\int_s^t\mathbbm{1}\{N^{(s,t]}>\mu \}  (f(X(s_1),k_0)-f(X(s_1),i_0))d[M_{i_0k_0}](s_1)
\\
&+\int_s^t L^{0}_{\alpha(s_1)} f(X(s_1),\alpha(s_1))ds_1+\sum_{j=1}^m\int_s^t L^{j}_{\alpha(s_1)} f(X(s_1),\alpha(s_1))dW^j(s_1)
\end{align*}
almost surely for all $ s<t\in [0, T]$. 
\end{example}

\begin{example}
Let $\mathcal{A}:=\{(j_1,\ldots,j_l)\in\mathcal{M}: \eta((j_1,\ldots,j_l))\leq 2$,  if  $j_i\notin\{ N_1,\ldots,N_{\mu},1,\ldots,m\}, \forall  \, i\in\{1,\ldots,l\}\}\cup\{(j_1,\ldots,j_l)\in\mathcal{M}: \eta((j_1,\ldots,j_l))\leq 1$  if  $j_i\in\{ N_1,\ldots,N_{\mu},1,\ldots,m\}$ for any $ i\in\{1,\ldots,l\} \}$ with $\mu=3$. 
Clearly, $\mathcal{A}=\{\nu,(0),(N_1),(k); k \in \{1, \ldots, m\}\}$, $\tilde{\mathcal{A}}=\{(N_1)\}$, $\mathcal{A}\setminus\tilde{\mathcal{A}}=\{\nu,(0),(k); k \in \{1, \ldots, m\}\}$ and $\mathcal{B}(\mathcal{A})$ $=$ $\{(N_2),$ $(N_3)$, $(\bar{N}_3),$ $(N_1,0),$ $ (N_2,0),$  $(N_3,0),$ $(\bar{N}_3,0),$ $(0,0),$ $(k,0),$ $(0,N_1),$ $(k,N_1),$ $(N_1,k),$ $(N_2,k),$ $(N_3,k),$ $(\bar{N}_3,k),$ $(0,k),$ $(k_1,k);k,k_1\in\{1,\ldots,m\}\}$. 
Further, $\mathcal{D}_b=\mathcal{D}_\sigma=\{\nu,(N_1),$ $(N_2),(N_3),(\bar{N}_3),(0),(k);k\in\{1,\ldots,m\}\}$ which gives $\mathcal{K}_b=2$, $\mathcal{K}_\sigma=2$ and $\mathcal{K}_f=4$.
Moreover, by Theorem \ref{thm:Ito-Taylor}, we obtain
\begin{align*}
f(&X(t),\alpha(t))=I_{\nu}[f(X(s),\alpha(s))]_{s,t}+I_{(0)}[f(X(s),\alpha(s))]_{s,t}+\sum_{k=1}^mI_{(k)}[f(X(s),\alpha(s))]_{s,t}
\\
&+I_{(N_1)}[f(X(s),\alpha(\cdot))]_{s,t}+I_{(N_2)}[f(X(\cdot),\alpha(\cdot))]_{s,t}+I_{(N_3)}[f(X(\cdot),\alpha(\cdot))]_{s,t}+I_{(\bar{N}_3)}[f(X(\cdot),\alpha(\cdot))]_{s,t}
\\
&+\sum_{i=1}^3I_{(N_i,0)}[f(X(\cdot),\alpha(\cdot))]_{s,t}+I_{(\bar{N}_3,0)}[f(X(\cdot),\alpha(\cdot))]_{s,t}+I_{(0,0)}[f(X(\cdot),\alpha(\cdot))]_{s,t}
\\
&+\sum_{k=1}^mI_{(k,0)}[f(X(\cdot),\alpha(\cdot))]_{s,t}+I_{(0,N_1)}[f(X(\cdot),\alpha(\cdot))]_{s,t}+\sum_{k=1}^mI_{(k,N_1)}[f(X(\cdot),\alpha(\cdot))]_{s,t}
\\
&+\sum_{i=1}^3\sum_{k=1}^mI_{(N_i,k)}[f(X(\cdot),\alpha(\cdot))]_{s,t}+\sum_{k=1}^mI_{(\bar{N}_3,k)}[f(X(\cdot),\alpha(\cdot))]_{s,t}+\sum_{k=1}^mI_{(0,k)}[f(X(\cdot),\alpha(\cdot))]_{s,t}
\\
&+\sum_{k,k_1=1}^mI_{(k_1,k)}[f(X(\cdot),\alpha(\cdot))]_{s,t}
\end{align*}
which on using  the definition of multiple integrals from Subsection \ref{sub:multiple integral} yields,
\begin{align*}
f(&X(t),\alpha(t))=f(X(s),\alpha(s))+\int_s^t L^0_{\alpha(s)}f(X(s),\alpha(s))ds_1+\sum_{k=1}^m\int_s^t L^k_{\alpha(s)}f(X(s),\alpha(s))dW^k(s_1)
\\
&+\sum_{i_0\neq k_0}\int_s^t \mathbbm{1}\{N^{(s,t]}=1 \} (f(X(s),k_0)-f(X(s),i_0))d[M_{i_0k_0}](s_1)
\\
&+\sum_{i_0\neq k_0}\int_s^t \mathbbm{1}\{N^{(s,t]}=2 \} (f(X(s_1),k_0)-f(X(s_1),i_0))d[M_{i_0k_0}](s_1)
\\
&+\sum_{i_0\neq k_0}\int_s^t \mathbbm{1}\{N^{(s,t]}=3 \} (f(X(s_1),k_0)-f(X(s_1),i_0))d[M_{i_0k_0}](s_1)
\\
&+\sum_{i_0\neq k_0}\int_s^t \mathbbm{1}\{N^{(s,t]}>3 \} (f(X(s_1),k_0)-f(X(s_1),i_0))d[M_{i_0k_0}](s_1)
\\
&+\sum_{i=1}^3\int_s^t \sum_{i_0\neq k_0}\int_s^{s_1}\mathbbm{1}\{N^{(s,s_1]}=i \}  (L^{0}_{k_0} f(X(s_2),k_0)-L^{0}_{i_0} f(X(s_2),i_0))d[ M_{i_0k_0}](s_2)ds_1
\\
&+\int_s^t \sum_{i_0\neq k_0}\int_s^{s_1}\mathbbm{1}\{N^{(s,s_1]}>3 \}  (L^{0}_{k_0} f(X(s_2),k_0)-L^{0}_{i_0} f(X(s_2),i_0))d[ M_{i_0k_0}](s_2)ds_1
\\
&+\int_s^t\int_s^{s_1}L^{0}_{\alpha(s_2)} L^{0}_{\alpha(s_2)} f(X(s_2),\alpha(s_2))ds_2ds_1
\\
&+\sum_{k=1}^m\int_s^t\int_s^{s_1}L^{k}_{\alpha(s_2)} L^{0}_{\alpha(s_2)} f(X(s_2),\alpha(s_2))dW^k(s_2)ds_1
\\
&+\sum_{i_0\neq k_0}\int_s^{t}\mathbbm{1}\{N^{(s,t]}=1 \}\int_s^{s_1} (L^0_{\alpha(s_2)}f(X(s_2),k_0)-L^0_{\alpha(s_2)}f(X(s_2),i_0))ds_2d[M_{i_0k_0}](s_1)
\\
&+\sum_{k=1}^m\sum_{i_0\neq k_0}\int_s^{t}\mathbbm{1}\{N^{(s,t]}=1 \}\int_s^{s_1} (L^k_{\alpha(s_2)}f(X(s_2),k_0)-L^k_{\alpha(s_2)}f(X(s_2),i_0))dW^k(s_2)d[M_{i_0k_0}](s_1)
\\
&+\sum_{i=1}^3\sum_{k=1}^m\int_s^t \sum_{i_0\neq k_0}\int_s^{s_1}\mathbbm{1}\{N^{(s,s_1]}=i \}  (L^{k}_{k_0} f(X(s_2),k_0)-L^{k}_{i_0} f(X(s_2),i_0))d[ M_{i_0k_0}](s_2)dW^k(s_1)
\\
&+\sum_{k=1}^m\int_s^t \sum_{i_0\neq k_0}\int_s^{s_1}\mathbbm{1}\{N^{(s,s_1]}>3 \}  (L^{k}_{k_0} f(X(s_2),k_0)-L^{k}_{i_0} f(X(s_2),i_0))d[ M_{i_0k_0}](s_2)dW^k(s_1)
\\
&+\sum_{k=1}^m\int_s^t\int_s^{s_1}L^{0}_{\alpha(s_2)} L^{k}_{\alpha(s_2)} f(X(s_2),\alpha(s_2))ds_2dW^k(s_1)
\\
&+\sum_{k,k_1=1}^m\int_s^t\int_s^{s_1}L^{k_1}_{\alpha(s_2)} L^{k}_{\alpha(s_2)} f(X(s_2),\alpha(s_2))dW^{k_1}(s_2)dW^k(s_1)
\end{align*}
almost surely for all $ s<t\in [0, T]$. 
\end{example}
\subsection{Explicit $\mathbf{\gamma}$-order numerical scheme}\label{sub:scheme}
For explicit numerical scheme of $\gamma\in\{n/2:n\in\mathbb{N}\}$-order  (see equation \eqref{eq:gen.scheme}),  we fix $\mu=2\gamma$ and define 
\begin{itemize}
\item[] $\mathcal{A}_{\gamma}^b:=\{\nu\} \cup \{(j_1,\ldots,j_l)\in\mathcal{M}: \eta((j_1,\ldots,j_l))\leq 2\gamma-1, \mbox{ if } j_i= 0\, \forall\: i\in\{1,\ldots,l\}\}\cup\{(j_1,\ldots,j_l)\in\mathcal{M}: \eta((j_1,\ldots,j_l))\leq 2\gamma-2$  if  $j_i\in\{ N_1,\ldots,N_{\mu},1,\ldots,m\}$ for any $ i\in\{1,\ldots,l\} \}$.
\item[] $\mathcal{A}_{\gamma}^\sigma:=\{\beta\in\mathcal{M}: \eta(\beta)\leq 2\gamma-1 \}$.
\end{itemize}
The sets $\mathcal{A}_{\gamma}^b$ and $\mathcal{A}_{\gamma}^\sigma$ are  hierarchical sets and their remainder sets are denoted by $\mathcal{B}(\mathcal{A}_\gamma^b):=\{\beta\in\mathcal{M}\setminus\mathcal{A}_{\gamma}^b:-\beta\in\mathcal{A}_{\gamma}^b\}$ and $\mathcal{B}(\mathcal{A}_\gamma^\sigma):=\{\beta\in\mathcal{M}\setminus\mathcal{A}_{\gamma}^\sigma:-\beta\in\mathcal{A}_{\gamma}^\sigma\}$, respectively. 
Further, we write $\mathcal{A}_{\gamma}^b=\tilde{\mathcal{A}}_{\gamma}^b\cup(\mathcal{A}_{\gamma}^b\setminus\tilde{\mathcal{A}}_{\gamma}^b)$ and $\mathcal{A}_{\gamma}^\sigma=\tilde{\mathcal{A}}_{\gamma}^\sigma\cup(\mathcal{A}_{\gamma}^\sigma\setminus\tilde{\mathcal{A}}_{\gamma}^\sigma)$ where 
\begin{itemize}
\item[] $\tilde{\mathcal{A}}_{\gamma}^b:=\{(j_1,\ldots,j_l)\in\mathcal{A}_{\gamma}^b: j_i\in\{N_1,\ldots,N_\mu\} \mbox{ for any } i\in\{1,\ldots,l \}\}$,
\item[] $\tilde{\mathcal{A}}_{\gamma}^\sigma:=\{(j_1,\ldots,j_l)\in\mathcal{A}_{\gamma}^\sigma: j_i\in\{N_1,\ldots,N_\mu\} \mbox{ for any } i\in\{1,\ldots,l \}\}.
$
\end{itemize}
Also, we assume that $b^k(\cdot,i_0)$ and $\sigma^{(k,j)}(\cdot,i_0)$ are sufficiently smooth functions for all $i_0 \in \mathcal{S}$, $k \in\{1, \ldots, d\}$ and $j\in\{1,\ldots,m\}$. 
We partition the interval $[0,T]$ into $n_T\in\mathbb{N}$ sub-intervals of equal length $h=T/n_T>0$, \textit{i.e.}, $t_n=nh$ for any $n\in\{0,1,\ldots, n_T\}$. 
 The explicit $\gamma\in\{n/2:n\in\mathbb{N}\}$-order scheme  for SDEwMS \eqref{eq:sdems} at grid point $t_{n+1}$ is given by,
\begin{align} 
Y^k&(t_{n+1})=Y^k(t_{n})+\int_{t_{n}}^{t_{n+1}}\Big(\sum_{\beta\in\mathcal{A}_{\gamma}^b\setminus\tilde{\mathcal{A}}_{\gamma}^b}I_{\beta}[b^k(Y(t_n),\alpha(t_n)) ]_{t_n,s}+\sum_{\beta\in\tilde{\mathcal{A}}_{\gamma}^b}I_{\beta}[b^k(Y(t_n),\alpha(\cdot)) ]_{t_n,s}\Big)ds\notag
\\
&+\sum_{j=1}^m\int_{t_{n}}^{t_{n+1}}\Big(\sum_{\beta\in\mathcal{A}_{\gamma}^\sigma\setminus\tilde{\mathcal{A}}_{\gamma}^\sigma}I_{\beta}[\sigma^{(k,j)}(Y(t_n),\alpha(t_n)) ]_{t_n,s}+\sum_{\beta\in\tilde{\mathcal{A}}_{\gamma}^\sigma}I_{\beta}[\sigma^{(k,j)}(Y(t_n),\alpha(\cdot)) ]_{t_n,s}\Big)dW^j(s)\label{eq:gen.scheme}
\end{align}
 almost surely for all $n\in\{0,1,\ldots,n_T-1\}$ and $k\in\{1,\ldots,d\}$ where the initial value $Y_0:=Y(0)$ is an $\mathcal{F}_0$-measurable random variable in $\mathbb{R}^d$.
 \begin{remark}
If the state space $\mathcal{S}$ of the Markov chain $\alpha$ is a singleton set, then the SDEwMS \eqref{eq:sdems} reduces to an  SDE and the integrals appearing in the terms corresponding to the sets $\tilde{\mathcal{A}}_{\gamma}^b$ and $\tilde{\mathcal{A}}_{\gamma}^\sigma$ in equation \eqref{eq:gen.scheme} vanish.
 In this case, \eqref{eq:gen.scheme} is the the $\gamma$-order explicit numerical scheme for SDE, see Section $10.6$ in \cite{kloeden1992numerical}.
 Clearly,  terms corresponding to the sets $\tilde{\mathcal{A}}_{\gamma}^b$ and $\tilde{\mathcal{A}}_{\gamma}^\sigma$ in \eqref{eq:gen.scheme} appear due to  the presence of the Markov chain.   
\end{remark}

We now give some examples of explicit numerical scheme \eqref{eq:gen.scheme} for $\gamma\in\{0.5,1.0,1.5\}$.

\begin{example}[\textbf{Euler scheme for SDEwMS}]
If  $\gamma=0.5$, then $\mathcal{A}_{0.5}^b=\mathcal{A}_{0.5}^\sigma=\{\nu\}$ and $\tilde{\mathcal{A}}_{0.5}^b=\tilde{\mathcal{A}}_{0.5}^\sigma=\varphi$.
By using \eqref{eq:gen.scheme} for $\gamma=0.5$ and the definition of multiple integrals from Subsection $\ref{sub:multiple integral}$, the Euler scheme for SDEwMS \eqref{eq:sdems} is given by 
\begin{align}
Y^k(t_{n+1})=&Y^k(t_{n})+\int_{t_{n}}^{t_{n+1}}I_{\nu}[b^k(Y(t_n),\alpha(t_n))]_{t_n,s}ds+\sum_{j=1}^m \int_{t_{n}}^{t_{n+1}}I_{\nu}[\sigma^{(k,j)}(Y(t_n),\alpha(t_n))]_{t_n,s}dW^j(s)\notag
\\
=&Y^k(t_n)
+\int_{t_n}^{t_{n+1}}b^k(Y(t_n),\alpha(t_n))ds+\sum_{j=1}^m \int_{t_n}^{t_{n+1}}\sigma^{(k,j)}(Y(t_n),\alpha(t_n))dW^j(s)\notag
\end{align}
almost surely for all $n\in\{0,1,\ldots,n_T-1\}$ and $k\in\{1,\ldots,d\}$, see also \cite{yuan2004convergence}. 
 If the state space $\mathcal{S}$ of the Markov chain $\alpha$ is a singleton set, then the SDEwMS \eqref{eq:sdems} and the above scheme  reduce to  the classical SDE and  its Euler scheme, respectively, as detailed in section 10.2 of \cite{kloeden1992numerical}.
\end{example}
\begin{example}[\textbf{Milstein-type scheme for SDEwMS}]
Let  $\gamma=1.0$. Then,  $\mathcal{A}_{1.0}^b=\{\nu\}$, $\mathcal{A}_{1.0}^\sigma=\{\nu,(N_1),(1),\ldots,(m)\}$, $\tilde{\mathcal{A}}_{1.0}^b=\varphi$ and $\tilde{\mathcal{A}}_{1.0}^\sigma=\{(N_1)\}$, which on using \eqref{eq:gen.scheme}  and the definition of multiple integrals from Subsection $\ref{sub:multiple integral}$ yields
\begin{align}
Y^k(&t_{n+1})=Y^k(t_{n})+\int_{t_{n}}^{t_{n+1}}I_{\nu}[b^k(Y(t_n),\alpha(t_n))]_{t_n,s}ds+\sum_{j=1}^m \int_{t_{n}}^{t_{n+1}}\Big(I_{\nu}[\sigma^{(k,j)}(Y(t_n),\alpha(t_n))]_{t_n,s}\notag
\\
&+\sum_{j_1=1}^m I_{(j_1)}[\sigma^{(k,j)}(Y(t_n),\alpha(t_n))]_{t_n,s}+I_{(N_1)}[\sigma^{(k,j)}(Y(t_n),\alpha(\cdot))]_{t_n,s}\Big)dW^j(s)\notag
\\
=&Y^k(t_n)
+\int_{t_n}^{t_{n+1}}b^k(Y(t_n),\alpha(t_n))ds+\sum_{j=1}^m \int_{t_n}^{t_{n+1}}\sigma^{(k,j)}(Y(t_n),\alpha(t_n))dW^j(s)\notag
\\
&+\sum_{j,j_1=1}^m \int_{t_n}^{t_{n+1}}\int_{t_n}^{s}L^{j_1}_{\alpha(t_n)}\sigma^{(k,j)}(Y(t_n),\alpha(t_n))dW^{j_1}(s_1)dW^j(s)\notag
\\
&+\sum_{j=1}^m \int_{t_n}^{t_{n+1}}\sum_{i_0\neq k_0}\int_{t_n}^{s}\mathbbm{1}\{N^{(t_n,s]}=1\}(\sigma^{(k,j)}(Y(t_n),k_0)-\sigma^{(k,j)}(Y(t_n),i_0))d[M_{i_0k_0}](s_1)dW^j(s)\notag
\end{align}
almost surely for all $n\in\{0,1,\ldots,n_T-1\}$ and $k\in\{1,\ldots,d\}$ which is the Milstein-type scheme for SDEwMS \eqref{eq:sdems}, see also \cite{kumar2021milstein} and \cite{nguyen2017milstein}.
If the state space $\mathcal{S}$ of the Markov chain $\alpha$ is a singleton set, then SDEwMS \eqref{eq:sdems} is an SDE, and the last term of the above equation vanishes which leads to the Milstein scheme for SDE, as detailed in Section 10.3 of \cite{kloeden1992numerical}.
Further, if the following commutative condition,
\begin{align*}
L^{j_1}_{i_0}\sigma^{(k,j)}(x,i_0)=L^{j}_{i_0}\sigma^{(k,j_1)}(x,i_0)
\end{align*} 
hold for all $x\in\mathbb{R}^d$, $i_0\in\mathcal{S}$, $j,j_1\in\{1,\ldots,m\}$ and $k\in\{1,\ldots,d\}$, then we write 
\begin{align*}
Y^k&(t_{n+1})=Y^k(t_{n})+b^k(Y(t_n),\alpha(t_n))h+\sum_{j=1}^m \sigma^{(k,j)}(Y(t_n),\alpha(t_n))(W^j(t_{n+1})-W^j(t_n))
\\
+&\frac{1}{2}\sum_{j,j_1=1}^mL^{j_1}_{\alpha(t_n)}\sigma^{(k,j)}(Y(t_n),\alpha(t_n))((W^{j_1}(t_{n+1})-W^{j_1}(t_n))(W^j(t_{n+1})-W^j(t_n))-\mathbbm{1}\{j=j_1\}h)
\\
+&\sum_{j=1}^m \mathbbm{1}\{N^{(t_n,t_{n+1})}=1\}(\sigma^{(k,j)}(Y(t_n),\alpha(\tau_1))-\sigma^{(k,j)}(Y(t_n),\alpha(t_n)))(W^j(t_{n+1})-W^j(\tau_1))
\end{align*}
almost surely for all $n\in\{0,1,\ldots,n_T-1\}$ and $k\in\{1,\ldots,d\}$ where $\tau_1$ is the first jump time of the Markov chain $\alpha$ in the interval $(t_n,t_{n+1})$.
\end{example}

\begin{example}[\textbf{$\mathbf{1.5}$-order scheme for SDEwMS}]
If $\gamma=1.5$, then
\begin{itemize}
\item[] $\mathcal{A}_{1.5}^b=\{\nu,(N_1),(0),(j_1);j_1\in\{1,\ldots,m\}\}$,
\item[] $\mathcal{A}_{1.5}^\sigma=\{\nu,(N_1),(N_2),(0),(j_1),(j_1,N_1),(N_1,j_1),(j_1,j_2),(N_1,j_1,N_1);j_1,j_2\in\{1,\ldots,m\}\}$,
\item[] $\tilde{\mathcal{A}}_{1.5}^b=\{(N_1)\}$,
\item[] $\tilde{\mathcal{A}}_{1.5}^\sigma=\{(N_1),(N_2),(j_1,N_1),(N_1,j_1),(N_1,j_1,N_1);j_1\in\{1,\ldots,m\}\}$.
\end{itemize}
 Further, by employing \eqref{eq:gen.scheme} for $\gamma=1.5$, we have
\begin{align*}
&Y^k(t_{n+1})=Y^k(t_{n})+\int_{t_{n}}^{t_{n+1}}\Big(I_{\nu}[b^k(Y(t_n),\alpha(t_n))]_{t_n,s}+I_{(0)}[b^k(Y(t_n),\alpha(t_n))]_{t_n,s}
\\
&+\sum_{j_1=1}^mI_{(j_1)}[b^k(Y(t_n),\alpha(t_n))]_{t_n,s}+I_{(N_1)}[b^k(Y(t_n),\alpha(\cdot))]_{t_n,s}\Big)ds
\\
&+\sum_{j=1}^m \int_{t_{n}}^{t_{n+1}}\Big(I_{\nu}[\sigma^{(k,j)}(Y(t_n),\alpha(t_n))]_{t_n,s}+I_{(0)}[\sigma^{(k,j)}(Y(t_n),\alpha(t_n))]_{t_n,s}
\\
&+\sum_{j_1=1}^m I_{(j_1)}[\sigma^{(k,j)}(Y(t_n),\alpha(t_n))]_{t_n,s}+\sum_{j_1,j_2=1}^mI_{(j_1,j_2)}[\sigma^{(k,j)}(Y(t_n),\alpha(t_n))]_{t_n,s}
\\
&+I_{(N_1)}[\sigma^{(k,j)}(Y(t_n),\alpha(\cdot))]_{t_n,s}+I_{(N_2)}[\sigma^{(k,j)}(Y(t_n),\alpha(\cdot))]_{t_n,s}+\sum_{j_1=1}^m I_{(j_1,N_1)}[\sigma^{(k,j)}(Y(t_n),\alpha(\cdot))]_{t_n,s}
\\
&+\sum_{j_1=1}^mI_{(N_1,j_1)}[\sigma^{(k,j)}(Y(t_n),\alpha(\cdot))]_{t_n,s}+\sum_{j_1=1}^m I_{(N_1,j_1,N_1)}[\sigma^{(k,j)}(Y(t_n),\alpha(\cdot))]_{t_n,s}\Big)dW^j(s)
\\
\end{align*}
which by the definition of multiple integrals from subsection $\ref{sub:multiple integral}$ yields
\begin{align}
Y^k(&t_{n+1})=Y^k(t_n)
+\int_{t_n}^{t_{n+1}}b^k(Y(t_n),\alpha(t_n))ds+\int_{t_n}^{t_{n+1}}\int_{t_n}^sL^{0}_{\alpha(t_n)}b^k(Y(t_n),\alpha(t_n))ds_1ds\notag
\\
&+\sum_{j_1=1}^m\int_{t_n}^{t_{n+1}}\int_{t_n}^sL^{j_1}_{\alpha(t_n)}b^k(Y(t_n),\alpha(t_n))dW^{j_1}(s_1)ds\notag
\\
&+\int_{t_n}^{t_{n+1}}\sum_{i_0\neq k_0}\int_{t_n}^s\mathbbm{1}\{N^{(t_n,s]}=1\}(b^k(Y(t_n),k_0)-b^k(Y(t_n),i_0))d[M_{i_0k_0}](s_1)ds\notag
\\
&+\sum_{j=1}^m \int_{t_n}^{t_{n+1}}\sigma^{(k,j)}(Y(t_n),\alpha(t_n))dW^j(s)\notag
\\
&+\sum_{j=1}^m \int_{t_n}^{t_{n+1}}\int_{t_n}^{s}L^{0}_{\alpha(t_n)}\sigma^{(k,j)}(Y(t_n),\alpha(t_n))ds_1dW^j(s)\notag
\\
&+\sum_{j,j_1=1}^m \int_{t_n}^{t_{n+1}}\int_{t_n}^{s}L^{j_1}_{\alpha(t_n)}\sigma^{(k,j)}(Y(t_n),\alpha(t_n))dW^{(j_1)}(s_1)dW^j(s)\notag
\\
&+\sum_{j,j_1,j_2=1}^m \int_{t_n}^{t_{n+1}}\int_{t_n}^s\int_{t_n}^{s_1}L^{j_2}_{\alpha(t_n)}L^{j_1}_{\alpha(t_n)}\sigma^{(k,j)}(Y(t_n),\alpha(t_n))dW^{j_2}(s_2)dW^{j_1}(s_1)dW^j(s)\notag
\\
&+\sum_{j=1}^m \int_{t_n}^{t_{n+1}}\sum_{i_0\neq k_0}\int_{t_n}^{s}\mathbbm{1}\{N^{(t_n,s]}=1\}(\sigma^{(k,j)}(Y(t_n),k_0)-\sigma^{(k,j)}(Y(t_n),i_0))d[M_{i_0k_0}](s_1)dW^j(s)\notag
\\
&+\sum_{j=1}^m \int_{t_n}^{t_{n+1}}\sum_{i_0\neq k_0}\int_{t_n}^{s}\mathbbm{1}\{N^{(t_n,s]}=2\}(\sigma^{(k,j)}(Y(t_n),k_0)-\sigma^{(k,j)}(Y(t_n),i_0))d[M_{i_0k_0}](s_1)dW^j(s)\notag
\\
&+\sum_{j,j_1=1}^m \int_{t_n}^{t_{n+1}}\sum_{i_0\neq k_0}\int_{t_n}^s\mathbbm{1}\{N^{(t_n,s]}=1\}\int_{t_n}^{s_1}(L^{j_1}_{\alpha(t_n)}\sigma^{(k,j)}(Y(t_n),k_0)\notag
\\
&\qquad-L^{j_1}_{\alpha(t_n)}\sigma^{(k,j)}(Y(t_n),i_0))dW^{j_1}(s_2)d[M_{i_0k_0}](s_1)dW^j(s)\notag
\\
&+\sum_{j,j_1=1}^m \int_{t_n}^{t_{n+1}}\int_{t_n}^s\sum_{i_0\neq k_0}\int_{t_n}^{s_1}\mathbbm{1}\{N^{(t_n,s_1]}=1\}(L^{j_1}_{k_0}\sigma^{(k,j)}(Y(t_n),k_0)\notag
\\
&\qquad-L^{j_1}_{i_0}\sigma^{(k,j)}(Y(t_n),i_0))d[M_{i_0k_0}](s_2)dW^{j_1}(s_1)dW^j(s)\notag
\\
&+\sum_{j,j_1=1}^m \int_{t_n}^{t_{n+1}}\sum_{i_0\neq k_0}\int_{t_n}^s\mathbbm{1}\{N^{(t_n,s]}=1\}\int_{t_n}^{s_1}\sum_{i_1\neq k_1}\int_{t_n}^{s_2}\mathbbm{1}\{N^{(t_n,s_2]}=1\}\notag
\\
&\qquad (L^{j_1}_{k_1}\sigma^{(k,j)}(Y(t_n),k_0)-L^{j_1}_{i_1}\sigma^{(k,j)}(Y(t_n),k_0))-(L^{j_1}_{k_1}\sigma^{(k,j)}(Y(t_n),i_0)-L^{j_1}_{i_1}\sigma^{(k,j)}(Y(t_n),i_0)) \notag
\\
& \qquad \qquad d[M_{i_1k_1}](s_3)dW^{j_1}(s_2)d[M_{i_0k_0}](s_1)dW^j(s)\notag
\end{align}
almost surely for all $n\in\{0,1,\ldots,n_T-1\}$ and $k\in\{1,\ldots,d\}$.
 The above equation is an order 1.5 strong Taylor scheme for SDEwMS \eqref{eq:sdems}.
 Notice that the last term on the right side of the above equation is zero.
 If the state space $\mathcal{S}$ of the Markov chain $\alpha$ is a singleton set, then the SDEwMS \eqref{eq:sdems} reduces to an  SDE and the $5th$, $10th$ to $14th$ terms on the right side of the above equation, which appear due to the switching of the Markov chain $\alpha$, vanish. 
 This gives an  order 1.5 strong Taylor scheme for SDE, see Section 10.4 in \cite{kloeden1992numerical}. 
 Now, if following commutative conditions,
 \begin{align*}
L^{j_1}_{i_0}\sigma^{(k,j)}(x,i_0)=L^{j}_{i_0}\sigma^{(k,j_1)}(x,i_0),\qquad L^{j_2}_{i_0}L^{j_1}_{i_0}\sigma^{(k,j)}(x,i_0)=L^{j_1}_{i_0}L^{j_2}_{i_0}\sigma^{(k,j)}(x,i_0)
\end{align*} 
hold for all $x\in\mathbb{R}^d$, $i_0\in\mathcal{S}$, $j,j_1,j_2\in\{1,\ldots,m\}$ and $k\in\{1,\ldots,d\}$, then we have
\begin{align*}
Y^k(&t_{n+1})=Y^k(t_n)+b^k(Y(t_n),\alpha(t_n))h+\frac{1}{2}L^{0}_{\alpha(t_n)}b^k(Y(t_n),\alpha(t_n))h^2
\\
&+\sum_{j_1=1}^mL^{j_1}_{\alpha(t_n)}b^k(Y(t_n),\alpha(t_n))\Delta_n Z^{j_1}
\\
&+\mathbbm{1}\{N^{(t_n,t_{n+1})}=1\}(b^k(Y(t_n),\alpha(\tau_1))-b^k(Y(t_n),\alpha(t_n)))(t_{n+1}-\tau_1)
\\
&+\sum_{j=1}^m \sigma^{(k,j)}(Y(t_n),\alpha(t_n))\Delta_nW^j+\sum_{j=1}^mL^{0}_{\alpha(t_n)}\sigma^{(k,j)}(Y(t_n),\alpha(t_n))(h\Delta_nW^j-\Delta_n Z^{j})
\\
&+\frac{1}{2}\sum_{j,j_1=1}^mL^{j_1}_{\alpha(t_n)}\sigma^{(k,j)}(Y(t_n),\alpha(t_n))(\Delta_nW^{j_1}\Delta_nW^j-\mathbbm{1}\{j=j_1\}h)
\\
&+\frac{1}{6}\sum_{j,j_1,j_2=1}^m L^{j_2}_{\alpha(t_n)}L^{j_1}_{\alpha(t_n)}\sigma^{(k,j)}(Y(t_n),\alpha(t_n))\big(\Delta_n W^{j}\Delta_n W^{j_1}\Delta_n W^{j_2}
\\
&\qquad-\mathbbm{1}\{j\neq j_1,j\neq j_2,j_1=j_2\}h\Delta_n W^{j}-3\mathbbm{1}\{j=j_1=j_2\}h\Delta_n W^{j}\big)
\\
&+\sum_{j=1}^m \mathbbm{1}\{N^{(t_n,t_{n+1})}=1\}(\sigma^{(k,j)}(Y(t_n),\alpha(\tau_1))-\sigma^{(k,j)}(Y(t_n),\alpha(t_n)))(W^j(t_{n+1})-W^j(\tau_1))
\\
&+\sum_{j=1}^m \mathbbm{1}\{N^{(t_n,t_{n+1})}=2\}(\sigma^{(k,j)}(Y(t_n),\alpha(\tau_2))-\sigma^{(k,j)}(Y(t_n),\alpha(t_n)))(W^j(t_{n+1})-W^j(\tau_2))
\\
&+\sum_{j,j_1=1}^m \mathbbm{1}\{N^{(t_n,t_{n+1})}=1\}(L^{j_1}_{\alpha(t_n)}\sigma^{(k,j)}(Y(t_n),\alpha(\tau_1))-L^{j_1}_{\alpha(t_n)}\sigma^{(k,j)}(Y(t_n),\alpha(t_n)))
\\
&\qquad\times(W^{j_1}(\tau_1)-W^{j_1}(t_n))(W^j(t_{n+1})-W^j(\tau_1))
\\
&+\frac{1}{2}\sum_{j,j_1=1}^m\mathbbm{1}\{N^{(t_n,t_{n+1})}=1\}(L^{j_1}_{\alpha(\tau_1)}\sigma^{(j)}(Y(t_n),\alpha(\tau_1))-L^{j_1}_{\alpha(t_n)}\sigma^{(j)}(Y(t_n),\alpha(t_n)))
\\
&\qquad\times\big((W^{j_1}(t_{n+1})-W^{j_1}(\tau_1))(W^{j}(t_{n+1})-W^{j}(\tau_1))-\mathbbm{1}\{j=j_1\}(t_{n+1}-\tau_1)\big)
\end{align*}
almost surely for all $n\in\{0,1,\ldots,n_T-1\}$ and $k\in\{1,\ldots,d\}$ where $\Delta_nZ^j:=\int_{t_n}^{t_{n+1}}\int_{t_n}^sdW^{j_1}(s_1)ds$ is normally distributed with mean zero and variance $E(\Delta_nZ^j)^2=h^3/3$, $\Delta_nW^j:=W^j(t_{n+1})-W^j(t_n)$ for all $j\in\{1,\ldots,m\}$ and $\tau_1,\tau_2$ are the first and second jump times of the Markov chain $\alpha$ in the interval $(t_n,t_{n+1})$, respectively. 
\end{example}

\begin{example}[\textbf{$\mathbf{2.0}$-order scheme for SDEwMS}]
Let $\gamma=2.0$.
Then,
\begin{itemize}
\item[]  $\mathcal{A}_{2.0}^b=\mathcal{A}_{1.5}^b\cup\{(N_2),(N_1,j_1),(j_1,N_1),(j_1,j_2),(N_1,j_1,N_1);j_1,j_2\in\{1,\ldots,m\}\}$
\item[] $\tilde{\mathcal{A}}_{2.0}^b=\tilde{\mathcal{A}}_{1.5}^b\cup\{(N_2,),(N_1,j_1),(j_1,N_1),(N_1,j_1,N_1);j_1\in\{1,\ldots,m\}\}$,
\item[] $\mathcal{A}_{2.0}^\sigma=\mathcal{A}_{1.5}^\sigma\cup\{(N_3),(0,N_1),(j_1,N_2),(N_1,0),(j_1,0),(N_2,j_1),(0,j_1),(N_1,0,N_1),$ 
 $(N_2,j_1,N_1),$ $(j_1,j_2,N_1),$ $(N_1,j_1,N_2),$ $(N_2,j_1,N_2),$ $(j_1,N_1,j_2),$ $(N_1,j_1,j_2),(j_1,j_2,j_3),$ 
 $(j_1,N_1,j_2,N_1),$ $(N_1,j_1,j_2,N_1),$ $(N_1,j_1,N_1,j_2),$ $(N_1,j_1,N_1,j_2,N_1)$ $;j_1,j_2,j_3\in\{1,\ldots,m\}$\}
 \item[] $\tilde{\mathcal{A}}_{2.0}^\sigma=\tilde{\mathcal{A}}_{1.5}^\sigma\cup\{(N_3),(0,N_1),(j_1,N_2),(N_1,0),(N_2,j_1),(N_1,0,N_1),(N_2,j_1,N_1),(j_1,j_2,N_1),$ $(N_1,j_1,N_2),(N_2,j_1,N_2),(j_1,N_1,j_2),(N_1,j_1,j_2),(j_1,N_1,j_2,N_1),(N_1,j_1,j_2,N_1), 
  (N_1,j_1,N_1,j_2),$ $(N_1,j_1,N_1,j_2,N_1) ;j_1,j_2\in\{1,\ldots,m\}\}. $
\end{itemize}
By the application of \eqref{eq:gen.scheme}, the 2.0-order scheme for SDEwMS \eqref{eq:sdems} is written in the form of multiple integrals as
\begin{align*}
Y^k&(t_{n+1})=Y^k(t_{n})+\int_{t_{n}}^{t_{n+1}}\Big(\sum_{\beta\in\mathcal{A}_{1.5}^b\setminus\tilde{\mathcal{A}}_{1.5}^b}I_{\beta}[b^k(Y(t_n),\alpha(t_n)) ]_{t_n,s}+\sum_{j_1,j_2=1}^m I_{(j_1,j_2)}[b^k(Y(t_n),\alpha(t_n)) ]_{t_n,s}
\\
&+\sum_{\beta\in\tilde{\mathcal{A}}_{1.5}^b}I_{\beta}[b^k(Y(t_n),\alpha(\cdot)) ]_{t_n,s}+I_{(N_2)}[b^k(Y(t_n),\alpha(\cdot)) ]_{t_n,s}+\sum_{j_1=1}^mI_{(N_1,j_1)}[b^k(Y(t_n),\alpha(\cdot)) ]_{t_n,s}
\\
&+\sum_{j_1=1}^mI_{(j_1,N_1)}[b^k(Y(t_n),\alpha(\cdot)) ]_{t_n,s}+\sum_{j_1=1}^mI_{(N_1,j_1,N_1)}[b^k(Y(t_n),\alpha(\cdot)) ]_{t_n,s}\Big)ds\notag
\\
&+\sum_{j=1}^m\int_{t_{n}}^{t_{n+1}}\Big(\sum_{\beta\in\mathcal{A}_{1.5}^\sigma\setminus\tilde{\mathcal{A}}_{1.5}^\sigma}I_{\beta}[\sigma^{(k,j)}(Y(t_n),\alpha(t_n)) ]_{t_n,s}+\sum_{j_1=1}^mI_{(j_1,0)}[\sigma^{(k,j)}(Y(t_n),\alpha(t_n)) ]_{t_n,s}
\\
&+\sum_{j_1=1}^mI_{(0,j_1)}[\sigma^{(k,j)}(Y(t_n),\alpha(t_n)) ]_{t_n,s}+\sum_{j_1,j_2,j_3=1}^mI_{(j_1,j_2,j_3)}[\sigma^{(k,j)}(Y(t_n),\alpha(t_n)) ]_{t_n,s}
\\
&+\sum_{\beta\in\tilde{\mathcal{A}}_{1.5}^\sigma}I_{\beta}[\sigma^{(k,j)}(Y(t_n),\alpha(\cdot)) ]_{t_n,s}+I_{(N_3)}[\sigma^{(k,j)}(Y(t_n),\alpha(\cdot)) ]_{t_n,s}+I_{(0,N_1)}[\sigma^{(k,j)}(Y(t_n),\alpha(\cdot)) ]_{t_n,s}
\\
&+\sum_{j_1=1}^mI_{(j_1,N_2)}[\sigma^{(k,j)}(Y(t_n),\alpha(\cdot)) ]_{t_n,s}+I_{(N_1,0)}[\sigma^{(k,j)}(Y(t_n),\alpha(\cdot)) ]_{t_n,s}
\\
&+\sum_{j_1=1}^mI_{(N_2,j_1)}[\sigma^{(k,j)}(Y(t_n),\alpha(\cdot)) ]_{t_n,s}+I_{(N_1,0,N_1)}[\sigma^{(k,j)}(Y(t_n),\alpha(\cdot)) ]_{t_n,s}
\\
&+\sum_{j_1=1}^mI_{(N_2,j_1,N_1)}[\sigma^{(k,j)}(Y(t_n),\alpha(\cdot)) ]_{t_n,s}+\sum_{j_1,j_2=1}^mI_{(j_1,j_2,N_1)}[\sigma^{(k,j)}(Y(t_n),\alpha(\cdot)) ]_{t_n,s}
\\
&+\sum_{j_1=1}^mI_{(N_1,j_1,N_2)}[\sigma^{(k,j)}(Y(t_n),\alpha(\cdot)) ]_{t_n,s}+\sum_{j_1=1}^mI_{(N_2,j_1,N_2)}[\sigma^{(k,j)}(Y(t_n),\alpha(\cdot)) ]_{t_n,s}
\\
&+\sum_{j_1,j_2=1}^mI_{(j_1,N_1,j_2)}[\sigma^{(k,j)}(Y(t_n),\alpha(\cdot)) ]_{t_n,s}+\sum_{j_1,j_2=1}^mI_{(N_1,j_1,j_2)}[\sigma^{(k,j)}(Y(t_n),\alpha(\cdot)) ]_{t_n,s}
\\
&+\sum_{j_1,j_2=1}^mI_{(j_1,N_1,j_2,N_1)}[\sigma^{(k,j)}(Y(t_n),\alpha(\cdot)) ]_{t_n,s}+\sum_{j_1,j_2=1}^mI_{(N_1,j_1,j_2,N_1)}[\sigma^{(k,j)}(Y(t_n),\alpha(\cdot)) ]_{t_n,s}
\\
&+\sum_{j_1,j_2=1}^mI_{(N_1,j_1,N_1,j_2)}[\sigma^{(k,j)}(Y(t_n),\alpha(\cdot)) ]_{t_n,s}+\sum_{j_1,j_2=1}^mI_{(N_1,j_1,N_1,j_2,N_1)}[\sigma^{(k,j)}(Y(t_n),\alpha(\cdot)) ]_{t_n,s}\Big)dW^j(s)
\end{align*}
almost surely for all $n\in\{0,1,\ldots,n_T-1\}$ and $k\in\{1,\ldots,d\}$. 
If the Markov chain $\alpha$ has singletin state space $\mathcal{S}$, then SDEwMS \eqref{eq:sdems} reduces to an SDE and fourth to eighth and thirteen to twentynine terms on the right side of the  above equation become zero and this leads to 2.0-order strong Taylor scheme for SDE, one can refer to Section 10.5 in \cite{kloeden1992numerical}.
\end{example}

In order to investigate the strong rate of convergence of  the explicit numerical scheme \eqref{eq:gen.scheme} for SDEwMS \eqref{eq:sdems}, we make the following assumptions. 

\begin{assumption}
\label{ass:A b sigma ds dw lipschitz}
There exists a  constant $L>0$ such that
\begin{align*}
|J^\beta_{i_0}b^k(x,i_0)-J^\beta_{i_0}b^k(y,i_0)| + |J^{\bar{\beta}}_{i_0}\sigma^{(k,j)}(x,i_0)-J^{\bar{\beta}}_{i_0}\sigma^{(k,j)}(y,i_0)|\leq L|x-y| 
\end{align*} 
for all $\beta\in (\mathcal{A}^b_\gamma\setminus\tilde{\mathcal{A}}_\gamma^b)\setminus\{\nu\}$ and $\bar{\beta}\in \mathcal{A}^\sigma_\gamma\setminus\tilde{\mathcal{A}}_\gamma^\sigma\setminus\{\nu\}$, $i_0\in \mathcal{S}$, $k\in\{1,\ldots,d\}$, $j\in\{1,\ldots,m\}$ and $x,y\in \mathbb{R}^d$. 
\end{assumption}

\begin{assumption}
\label{ass:A b N lipschitz}
There exists a  constant $L>0$ such that, for all $\beta\in\tilde{\mathcal{A}}_\gamma^b$, $i_1,\ldots,i_{r+1}\in\mathcal{S}$, $k\in\{1,\ldots,d\}$ and $x,y\in \mathbb{R}^d$,
\begin{align*}
|J^{\beta_{r+1}}_{i_{r+1}}J^{\beta_r}_{i_r}\cdots J^{\beta_2}_{i_2}J^{\beta_1}_{i_1}b^k(x,i_1)-J^{\beta_{r+1}}_{i_{r+1}}J^{\beta_r}_{i_r}\cdots J^{\beta_2}_{i_2}J^{\beta_1}_{i_1}b^k(y,i_1)|\leq L |x-y|
\end{align*}
where $\beta=\beta_{r+1}\star(N_{\mu_{r}})\star \beta_{r}\star \cdots\star(N_{\mu_2})\star \beta_2\star(N_{\mu_1})\star \beta_1$, $ \mu_1,\ldots,\mu_r\in\{1,\ldots,2\gamma-2\}$, $\beta_1,\ldots,\beta_{r+1}\in\mathcal{M}_1$ and $r\in\{1,\ldots,2\gamma-2\}$.
\end{assumption}

\begin{assumption} \label{ass:A sigma N lipschitz}
There exists a  constant $L>0$ such that, for all $\beta\in\tilde{\mathcal{A}}_\gamma^\sigma$, $i_1,\ldots,i_{r+1}\in\mathcal{S}$, $k\in\{1,\ldots,d\}$, $j\in\{1,\ldots,m\}$ and $x,y\in \mathbb{R}^d$,
\begin{align*}
|J^{\beta_{r+1}}_{i_{r+1}}J^{\beta_r}_{i_r}\cdots J^{\beta_2}_{i_2}J^{\beta_1}_{i_1}\sigma^{(k,j)}(x,i_1)-J^{\beta_{r+1}}_{i_{r+1}}J^{\beta_r}_{i_r}\cdots J^{\beta_2}_{i_2}J^{\beta_1}_{i_1}\sigma^{(k,j)}(y,i_1)|\leq L |x-y|
\end{align*}
where $\beta=\beta_{r+1}\star(N_{\mu_{r}})\star \beta_{r}\star \cdots\star(N_{\mu_2})\star \beta_2\star(N_{\mu_1})\star \beta_1$, $ \mu_1,\ldots,\mu_r\in\{1,\ldots,2\gamma-1\}$, $\beta_1,\ldots,\beta_{r+1}\in\mathcal{M}_1$ and $r\in\{1,\ldots,2\gamma-2\}$.
\end{assumption}
To  state the forthcoming assumptions, we define the following sets,
 \begin{itemize}
 \item[] $\tilde{\mathcal{B}}_1^b:=\{(j)\star\beta:j\in\{N_1,\ldots,N_\mu,\bar{N}_\mu,0,1,\ldots,m\},\beta\in\{(j_1,\ldots,j_l)\in\mathcal{B}(\mathcal{A}_{\gamma-0.5}^b)\cap\mathcal{A}_{\gamma}^b:j_i\notin\{N_1,\ldots,N_{\mu}\}\forall\, i\in\{1,\ldots,l\}\}\cup \{\nu\} \}$ and
 \item[] $\tilde{\mathcal{B}}_1^\sigma:=\{(j)\star\beta:j\in\{N_1,\ldots,N_\mu,\bar{N}_\mu,0,1,\ldots,m\},\beta\in\{(j_1,\ldots,j_l)\in\mathcal{B}(\mathcal{A}_{\gamma-0.5}^\sigma)\cap\mathcal{A}_{\gamma}^\sigma:j_i\notin\{N_1,\ldots,N_{\mu}\}\forall\, i\in\{1,\ldots,l\}\}\cup \{\nu\}\}.$
\end{itemize}

\begin{assumption}
\label{ass:reminder b sigma ds dw growth}
For all $i_0\in \mathcal{S}$, $k\in\{1,\ldots,d\}$, $j\in\{1,\ldots,m\}$, $x\in \mathbb{R}^d$, $\beta\in(\mathcal{B}(\mathcal{A}_\gamma^b)\setminus\tilde{\mathcal{B}}_1^b) \cap \mathcal{M}_1$ and $\bar{\beta}\in(\mathcal{B}(\mathcal{A}_\gamma^\sigma)\setminus\tilde{\mathcal{B}}_1^\sigma) \cap \mathcal{M}_1$, there exists a constant $L>0$ such that 
\begin{align*}
|J^\beta_{i_0}b^k(x,i_0)|+|J^{\bar{\beta}}_{i_0}\sigma^{(k,j)}(x,i_0)|\leq L(1+|x|).
\end{align*}
\end{assumption}

 \begin{assumption}\label{ass:reminder b N growth}
 For all $\beta\in(\mathcal{B}(\mathcal{A}_\gamma^b)\setminus\tilde{\mathcal{B}}_1^b) \cap (\mathcal{M}_2\cup\mathcal{M}_3)$, $i_1,\ldots,i_{r+1}\in\mathcal{S}$, $r\in\{1,\ldots,2\gamma-2\}$, $k\in\{1,\ldots,d\}$ and $x\in \mathbb{R}^d$,  there exists a  constant $L>0$ such that
\begin{align*}
|J^{\beta_{r+1}}_{i_{r+1}}J^{\beta_r}_{i_r}\cdots J^{\beta_2}_{i_2}J^{\beta_1}_{i_1}b^k(x,i_1)|\leq L (1+|x|)
\end{align*}
where $\beta=\beta_{r+1}\star(N_{\mu_{r}})\star \beta_{r}\star \cdots\star(N_{\mu_2})\star \beta_2\star(N_{\mu_1})\star \beta_1$, $\mu_1,\ldots,\mu_r\in\{1,\ldots,2\gamma\}$ and $\beta_1,\ldots,\beta_{r+1}\in\mathcal{M}_1$.
\end{assumption}
 
\begin{assumption}\label{ass:reminder sigma N growth}
 For all $\beta\in(\mathcal{B}(\mathcal{A}_\gamma^\sigma)\setminus\tilde{\mathcal{B}}_1^\sigma) \cap (\mathcal{M}_2 \cup \mathcal{M}_3)$, $i_1,\ldots,i_{r+1}\in\mathcal{S}$, $r\in\{1,\ldots,2\gamma-1\}$, $k\in\{1,\ldots,d\}$, $j\in\{1,\ldots,m\}$ and $x\in \mathbb{R}^d$, there exists a  constant $L>0$ such that
\begin{align*}
|J^{\beta_{r+1}}_{i_{r+1}}J^{\beta_r}_{i_r}\cdots J^{\beta_2}_{i_2}J^{\beta_1}_{i_1}\sigma^{(k,j)}(x,i_1)|\leq L (1+|x|),
\end{align*}  
where $\beta=\beta_{r+1}\star(N_{\mu_{r}})\star \beta_{r}\star \cdots\star(N_{\mu_2})\star \beta_2\star(N_{\mu_1})\star \beta_1$, $\mu_1,\ldots,\mu_r\in\{1,\ldots,2\gamma\}$ and  $\beta_1,\ldots,\beta_{r+1}\in\mathcal{M}_1$.
\end{assumption} 

The following theorem is  the second main result of this article.
\begin{thm} \label{thm:main}
Let Assumptions $\ref{ass:initial data}$, \ref{ass: b sigma lipschitz} and \ref{ass:A b sigma ds dw lipschitz} to $\ref{ass:reminder sigma N growth}$ be satisfied. Then,  the numerical scheme \eqref{eq:gen.scheme} converges in $\mathcal{L}^2$-sense to the true solution of SDEwMS \eqref{eq:sdems} with a rate of convergence equal to $\gamma\in\{n/2:n\in\mathbb{N}\}$, \textit{i.e.}, there exists a constant $C>0$, independent of $h$, such that
$$
E\Big(\sup_{n\in\{0,1,\ldots,n_T\}}|X(t_n)-Y(t_n)|^2 \Big)\leq C h^{2\gamma}
$$
where $0<h<1/(2q)$ with $q:=\max\{-q_{i_0i_0}:i_0\in\mathcal{S}\}$. 
\end{thm}
We conclude this section by following remarks.
\begin{remark}
If the Markov chain $\alpha$ has singleton state space $\mathcal{S}$, then SDEwMS \eqref{eq:sdems} reduces to an SDE and Assumptions \ref{ass:A b N lipschitz}, \ref{ass:A sigma N lipschitz}, \ref{ass:reminder b N growth} and \ref{ass:reminder sigma N growth} does not appear.
\end{remark}

\begin{remark}
\label{rem:A b sigma ds dw growth}
Due to the Assumptions \ref{ass: b sigma lipschitz} and \ref{ass:A b sigma ds dw lipschitz}, 
\begin{align*}
|J^{\beta}_{i_0}b^k(x,i_0)|+ |J^{\bar{\beta}}_{i_0}\sigma^{(k,j)}(x,i_0)| \leq C(1+|x|)
\end{align*}
for all $\beta\in \mathcal{A}^b_\gamma\setminus\tilde{\mathcal{A}}_\gamma^b$, $\bar{\beta}\in \mathcal{A}^\sigma_\gamma\setminus\tilde{\mathcal{A}}_\gamma^\sigma$, $i_0\in\mathcal{S}$, $k\in\{1,\ldots,d\}$, $j\in\{1,\ldots,m\}$ and $x\in \mathbb{R}^d$.
\end{remark}

We obtain the following remarks due to Assumptions \ref{ass:A b N lipschitz} and \ref{ass:A sigma N lipschitz}. 
\begin{remark}\label{rem:A b sigma N lipschitz}
Notice that, for all $\beta=\beta_{r+1}\star(N_{\mu_{r}})\star \beta_{r}\star \cdots\star(N_{\mu_2})\star \beta_2\star(N_{\mu_1})\star \beta_1 \in\tilde{\mathcal{A}}_\gamma^b$, the expansion of $J^{\beta_{r+1}}_{i_{r+1}}J^{\beta_r}_{i_rk_r}\cdots J^{\beta_2}_{i_2k_2}(J^{\beta_1}_{k_1}b^k(x,k_1)-J^{\beta_1}_{i_1}b^k(x,i_1))-J^{\beta_{r+1}}_{i_{r+1}}J^{\beta_r}_{i_rk_r}\cdots J^{\beta_2}_{i_2k_2}(J^{\beta_1}_{k_1}b^k(y,k_1)-J^{\beta_1}_{i_1}b^k(y,i_1))$ consists of $2^{r}$ terms of the type $J^{\beta_{r+1}}_{i_{r+1}}$ $J^{\beta_r}_{\rho_r}\cdots J^{\beta_2}_{\rho_2}J^{\beta_1}_{\rho_1}b^k(x,\rho_1)-J^{\beta_{r+1}}_{i_{r+1}}J^{\beta_r}_{\rho_r}\cdots J^{\beta_2}_{\rho_2}J^{\beta_1}_{\rho_1}b^k(y,\rho_1)$ where $\rho_j\in\{i_j,k_j\}$,  for all $j\in\{1,\ldots,r\}$, $k\in\{1,\ldots,d\}$,  $i_1, \ldots, i_{r+1}, k_1,\ldots,k_r\in\mathcal{S}$, $ \mu_1,\ldots,\mu_r\in\{1,\ldots,2\gamma-2\}$ and $\beta_1,\ldots,\beta_{r+1}\in\mathcal{M}_1$.
Hence, Assumption \ref{ass:A b N lipschitz} gives
\begin{align*}
|J^{\beta_{r+1}}_{i_{r+1}}J^{\beta_r}_{i_rk_r}\cdots J^{\beta_2}_{i_2k_2}(J^{\beta_1}_{k_1}b^k(x,k_1)&-J^{\beta_1}_{i_1}b^k(x,i_1))-J^{\beta_{r+1}}_{i_{r+1}}J^{\beta_r}_{i_rk_r}\cdots J^{\beta_2}_{i_2k_2}(J^{\beta_1}_{k_1}b^k(y,k_1)-J^{\beta_1}_{i_1}b^k(y,i_1))|
\\
\leq&C2^{r}|x-y|\leq C2^{2\gamma-2}|x-y| \leq C|x-y|
\end{align*}
for all  $x,y\in \mathbb{R}^d$ and $r\in\{1,\ldots,2\gamma-2\}$.
Similarly, 
 for all $\bar{\beta}=\bar{\beta}_{\bar{r}+1}\star(N_{\bar{\mu}_{\bar{r}}})\star \bar{\beta}_{\bar{r}}\star \cdots\star(N_{\bar{\mu}_2})\star \bar{\beta}_2\star(N_{\bar{\mu}_1})\star \bar{\beta}_1\in\tilde{\mathcal{A}}_\gamma^\sigma$, Assumption \ref{ass:A sigma N lipschitz} yields,
\begin{align*}
|J^{\bar{\beta}_{\bar{r}+1}}_{\bar{i}_{\bar{r}+1}}&J^{\bar{\beta}_{\bar{r}}}_{\bar{i}_{\bar{r}}\bar{k}_{\bar{r}}}\cdots J^{\bar{\beta}_2}_{\bar{i}_2\bar{k}_2}(J^{\bar{\beta}_1}_{\bar{k}_1}\sigma^{(k,j)}(x,\bar{k}_1)-J^{\bar{\beta}_1}_{\bar{i}_1}\sigma^{(k,j)}(x,\bar{i}_1))
\\
-&J^{\bar{\beta}_{\bar{r}+1}}_{\bar{i}_{\bar{r}+1}}J^{\bar{\beta}_{\bar{r}}}_{\bar{i}_{\bar{r}}\bar{k}_{\bar{r}}}\cdots J^{\bar{\beta}_2}_{\bar{i}_2\bar{k}_2}(J^{\bar{\beta}_1}_{\bar{k}_1}\sigma^{(k,j)}(y,\bar{k}_1)-J^{\bar{\beta}_1}_{\bar{i}_1}\sigma^{(k,j)}(y,\bar{i}_1))|\leq C|x-y|
\end{align*}
where $\bar{i}_1,\ldots,\bar{i}_{\bar{r}+1}, \bar{k}_1, \ldots, \bar{k}_{\bar{r}} \in\mathcal{S}$, $ \bar{\mu}_1,\ldots,\bar{\mu}_{\bar{r}}\in\{1,\ldots,2\gamma-1\}$, $\bar{\beta}_1,\ldots,\bar{\beta}_{\bar{r}+1}\in\mathcal{M}_1$, $\bar{r}\leq 2\gamma-1$, $k\in\{1,\ldots,d\}$, $j\in\{1,\ldots,m\}$ and $x,y\in \mathbb{R}^d$.
\end{remark}

\begin{remark}\label{rem:A b N growth}
Due to Remark \ref{rem:A b sigma N lipschitz}, for all $\beta=\beta_{r+1}\star(N_{\mu_{r}})\star \beta_{r}\star \cdots\star(N_{\mu_2})\star \beta_2\star(N_{\mu_1})\star \beta_1\in\tilde{\mathcal{A}}_\gamma^b$, $i_1,\ldots,i_{r+1}, k_1, \ldots, k_r \in\mathcal{S}$, $k\in\{1,\ldots,d\}$ and $x\in \mathbb{R}^d$,
\begin{align*}
|J^{\beta_{r+1}}_{i_{r+1}}J^{\beta_r}_{i_rk_r}\cdots J^{\beta_2}_{i_2k_2}(J^{\beta_1}_{k_1}b^k(x,k_1)-J^{\beta_1}_{i_1}b^k(x,i_1))|\leq C(1+|x|)
\end{align*}
where  $ \mu_1,\ldots,\mu_r\in\{1,\ldots,2\gamma-2\}$, $\beta_1,\ldots,\beta_{r+1}\in\mathcal{M}_1$ and $r\in\{1,\ldots,2\gamma-2\}$.
As before, for all $\bar{\beta}=\bar{\beta}_{\bar{r}+1}\star(N_{\bar{\mu}_{\bar{r}}})\star \bar{\beta}_{\bar{r}}\star \cdots\star(N_{\bar{\mu}_2})\star \bar{\beta}_2\star(N_{\bar{\mu}_1})\star \bar{\beta}_1\in\tilde{\mathcal{A}}_\gamma^\sigma$,
\begin{align*}
|J^{\bar{\beta}_{\bar{r}+1}}_{\bar{i}_{\bar{r}+1}}&J^{\bar{\beta}_{\bar{r}}}_{\bar{i}_{\bar{r}}\bar{k}_{\bar{r}}}\cdots J^{\bar{\beta}_2}_{\bar{i}_2\bar{k}_2}(J^{\bar{\beta}_1}_{\bar{k}_1}\sigma^{(k,j)}(x,\bar{k}_1)-J^{\bar{\beta}_1}_{\bar{i}_1}\sigma^{(k,j)}(x,\bar{i}_1))|\leq C(1+|x|)
\end{align*}
where $\bar{i}_1,\ldots,\bar{i}_{\bar{r}+1}, \bar{k}_1, \ldots, \bar{k}_{\bar{r}} \in\mathcal{S}$, $ \bar{\mu}_1,\ldots,\bar{\mu}_{\bar{r}}\in\{1,\ldots,2\gamma-1\}$, $\bar{\beta}_1,\ldots,\bar{\beta}_{\bar{r}+1}\in\mathcal{M}_1$, $\bar{r}\leq 2\gamma-1$, $k\in\{1,\ldots,d\}$, $j\in\{1,\ldots,m\}$ and $x\in \mathbb{R}^d$.
\end{remark}

As a result of the Assumptions \ref{ass:reminder b N growth} and \ref{ass:reminder sigma N growth}, we get the following remarks.

\begin{remark}\label{rem:reminder b N growth}
Notice that, for all  $\beta=\beta_{r+1}\star(N_{\mu_{r}})\star \beta_{r}\star \cdots\star(N_{\mu_2})\star \beta_2\star(N_{\mu_1})\star \beta_1\in(\mathcal{B}(\mathcal{A}_\gamma^b)\setminus\tilde{\mathcal{B}}_1^b)\cap (\mathcal{M}_2 \cup\mathcal{M}_3)) $, the expression of $J^{\beta_{r+1}}_{i_{r+1}}J^{\beta_r}_{i_rk_r}\cdots J^{\beta_2}_{i_2k_2}(J^{\beta_1}_{k_1}b^k(x,k_1)-J^{\beta_1}_{i_1}b^k(x,i_1))$ consists of $2^r$ terms of the type $J^{\beta_{r+1}}_{i_{r+1}}$ $J^{\beta_r}_{\rho_r}\cdots J^{\beta_2}_{\rho_2}J^{\beta_1}_{\rho_1}b^k(x,\rho_1)$ where $\rho_j\in\{i_j,k_j\}$  for all $j\in\{1,\ldots,r\}$, $k\in\{1,\ldots,d\}$,  $i_1, \ldots, i_{r+1}, k_1,\ldots,k_r\in\mathcal{S}$, $ \mu_1,\ldots,\mu_r\in\{1,\ldots,2\gamma\}$, $\beta_1,\ldots,\beta_{r+1}\in\mathcal{M}_1$ and $r\in\{1,\ldots,2\gamma-2\}$.
Thus, Assumption \ref{ass:reminder b N growth} yields,
\begin{align*}
|J^{\beta_{r+1}}_{i_{r+1}}&J^{\beta_r}_{i_rk_r}\cdots J^{\beta_2}_{i_2k_2}(J^{\beta_1}_{k_1}b^k(x,k_1)-J^{\beta_1}_{i_1}b^k(x,i_1))|\leq C2^r(1+|x|)\leq C2^{2\gamma-2}(1+|x|)\leq C(1+|x|)
\end{align*}
for all $x\in \mathbb{R}^d$.
\end{remark}

\begin{remark}\label{rem:reminder sigma N growth}
By using Assumption \ref{ass:reminder sigma N growth} and adopting the similar arguments as used in Remark \ref{rem:reminder b N growth}, for all $\beta\in(\mathcal{B}(\mathcal{A}_\gamma^\sigma)\setminus\tilde{\mathcal{B}}_1^\sigma)\cap (\mathcal{M}_2 \cup\mathcal{M}_3))$, $i_1,\ldots,i_{r+1}\in\mathcal{S}$, $k\in\{1,\ldots,d\}$, $j\in\{1,\ldots,m\}$ and $x\in \mathbb{R}^d$, we have
\begin{align*}
|J^{\beta_{r+1}}_{i_{r+1}}&J^{\beta_r}_{i_rk_r}\cdots J^{\beta_2}_{i_2k_2}(J^{\beta_1}_{k_1}\sigma^{(k,j)}(x,k_1)-J^{\beta_1}_{i_1}\sigma^{(k,j)}(x,i_1))|\leq C(1+|x|)
\end{align*}
where $\beta=\beta_{r+1}\star(N_{\mu_{r}})\star \beta_{r}\star \cdots\star(N_{\mu_2})\star \beta_2\star(N_{\mu_1})\star \beta_1$, $\mu_1,\ldots,\mu_r\in\{1,\ldots,2\gamma\}$, $\beta_1,\ldots,\beta_{r+1}\in\mathcal{M}_1$ and $r\in\{1,\ldots,2\gamma-1\}$.
\end{remark}

\section{It\^o-Taylor expansion}
\label{sec:ito}
In this section, we prove our first main result on the It\^o-Taylor expansion, \textit{i.e.}, Theorem \ref{thm:Ito-Taylor}, but first we establish several lemmas and corollaries required in the proof. 

The proof of the following lemma can be found  in \cite{nguyen2017milstein} (see Lemma 2.2). 
\begin{lemma}\label{lem:ItowMS}
Let $f:\mathbb{R}^d\times\mathcal{S}\mapsto\mathbb{R}$ be a twice  continuously differentiable function in the first argument.
Then,
\begin{align*}
f(X(t),\alpha(t))=&f(X(s),\alpha(s))+\int^t_sL^{0}_{\alpha(u)}f(X(u),\alpha(u))du+\sum_{j=1}^m\int^t_sL^j_{\alpha(u)}f(X(u),\alpha(u))dW^{j}(u)
\\
&\quad+\sum_{i_0\neq k_0}\int^t_s(f(X(u),k_0)-f(X(u),i_0))d[M_{i_0k_0}](u)
\\
=& I_{\nu}[f(X(s),\alpha(s))]_{s,t}+I_{(0)}[f(X(\cdot),\alpha(\cdot))]_{s,t}+\sum_{j=1}^mI_{(j)}[f(X(\cdot),\alpha(\cdot))]_{s,t}
\\
&+\sum_{r=1}^{\mu}I_{(N_r)}[f(X(\cdot),\alpha(\cdot))]_{s,t}+I_{(\bar{N}_{\mu})}[f(X(\cdot),\alpha(\cdot))]_{s,t}
\end{align*}
 almost surely for all $s<t\in[0,T]$.
\end{lemma}
When the state of the Markov chain $\alpha$ is fixed, then we obtain the following corollary. 
\begin{corollary}\label{cor:Ito}
For every $i_0 \in \mathcal{S}$,  let $f(\cdot,i_0):\mathbb{R}^d\mapsto\mathbb{R}$ be a twice continuously differential function in the first argument. Then,
\begin{align*}
f(X(t),i_0)=&f(X(s),i_0)+\int^t_sL^0_{\alpha(u)}f(X(u),i_0)du+\sum^m_{j=1}\int^t_sL^j_{\alpha(u)}f(X(u),i_0)dW^{j}(u)
\\
=&I_{\nu}[f(X(s),i_0)]_{s,t}+I_{(0)}[f(X(\cdot),i_0)]_{s,t}+\sum_{j=1}^mI_{(j)}[f(X(\cdot),i_0)]_{s,t} 
\end{align*}
almost surely for all $s<t\in[0,T]$.
\end{corollary}
For the following lemma, we define 
\begin{align*}
\mathbbm{I}^{j}_{(s,t]}=\begin{cases} 
\mathbbm{1}\{N^{(s,t]}=r\} & \mbox{ if } j= N_r, r\in\{1,\ldots,\mu\}
\\
\mathbbm{1}\{N^{(s,t]}>\mu\} & \mbox{ if }  j= \bar{N}_\mu
\end{cases}
\end{align*}
 for any  $s<t\in[0,T]$ where $j\in\{N_1, \ldots, N_\mu, \bar{N}_\mu\}$.
\begin{lemma}\label{lem:Ito on beta}
Let $\beta\in\mathcal{M}$ such that $l(\beta) \in \{2, 3,\ldots\}$. Then, for a sufficiently smooth function $~f:\mathbb{R}^d\times\mathcal{S}\mapsto\mathbb{R}$, we have
\begin{align*}
I_{\beta}[f(X(\cdot),\alpha(\cdot))&]_{s,t}
=\begin{cases}
\displaystyle I_{\beta}[f(X(s),\alpha(s))]_{s,t}+\sum_{j\in\{N_1,\ldots,N_{\mu },\bar{N}_{\mu },0,1,\ldots,m\}} I_{(j)\star\beta}[f(X(\cdot),\alpha(\cdot))]_{s,t} &\mbox{if }\beta\in\mathcal{M}_1
\\
\displaystyle I_{\beta}[f(X(s),\alpha(\cdot))]_{s,t}+\sum_{j\in\{N_1,\ldots,N_{\mu },\bar{N}_{\mu },0,1,\ldots,m\}} I_{(j)\star\beta}[f(X(\cdot),\alpha(\cdot))]_{s,t} &\mbox{if }\beta\in\mathcal{M}_2
\\
\displaystyle  I_{\beta}[f(X(s),\alpha(\cdot))]_{s,t}+\sum_{j=0}^m I_{(j)\star\beta}[f(X(\cdot),\alpha(\cdot))]_{s,t} &\mbox{if }\beta\in\mathcal{M}_3
\end{cases}
\end{align*}
almost surely  for any $s<t\in[0,T]$, provided all the multiple integrals appearing above exist. 
If  $l(\beta)=0$, then  the cases $\beta \in \mathcal{M}_2$ and $\beta \in \mathcal{M}_3$ on the right side of the above expression is dropped and when $l(\beta)=1$,  the case $\beta\in \mathcal{M}_2$  does not appear.  
\end{lemma}
\begin{proof}
First, we establish the result for $l(\beta)\in \{0,1\}$. 
 Let $l(\beta)=0$, \textit{i.e.}, $\beta=\nu\in\mathcal{M}_1$. 
  By using the definition of multiple integrals from Subsection \ref{sub:multiple integral} and Lemma \ref{lem:ItowMS}, we have,
\begin{align*}
I_{\beta}[f(X(\cdot),\alpha(\cdot))]_{s,t}=&f(X(t),\alpha(t))
\\
=&I_{\nu}[f(X(s),\alpha(s))]_{s,t}+I_{(0)}[f(X(\cdot),\alpha(\cdot))]_{s,t}+\sum_{j=1}^mI_{(j)}[f(X(\cdot),\alpha(\cdot))]_{s,t}
\\
&+\sum_{r=1}^{\mu}I_{(N_r)}[f(X(\cdot),\alpha(\cdot))]_{s,t}+I_{(\bar{N}_{\mu})}[f(X(\cdot),\alpha(\cdot))]_{s,t}
\\
=&I_{\beta}[f(X(s),\alpha(s))]_{s,t}+\sum_{j\in\{N_1,\ldots,N_{\mu },\bar{N}_{\mu },0,1,\ldots,m\}} I_{(j)\star\beta}[f(X(\cdot),\alpha(\cdot))]_{s,t} 
\end{align*}
almost surely for any $s<t\in[0,T]$.
Now, let $l(\beta)=1$ and $\beta=(j_1)\in \mathcal{M}_1$ for $j_1 \in \{0,1,\ldots,m\}$. 
 As before, we use the definition of multiple integrals from Subsection \ref{sub:multiple integral} and Lemma \ref{lem:ItowMS} to obtain, 
\begin{align*}
I_{\beta}[f(X(\cdot),&\alpha(\cdot))]_{s,t}=\int_s^tL^{j_1}_{\alpha(s_1)} f(X(s_1),\alpha(s_1))dW^{j_1}(s_1)
\\
=&\int_s^t\Big(L^{j_1}_{\alpha(s)} f(X(s),\alpha(s))+\sum_{j=0}^m\int_{s}^{s_1}L^j_{\alpha(s_2)}L^{j_1}_{\alpha(s_2)} f(X(s_2),\alpha(s_2))dW^j(s_2)
\\
&+\sum_{r=1}^\mu\sum_{i_0\neq k_0}\int_s^{s_1}\mathbbm{1}\{N^{(s,s_1]}=r\}(L^{j_1}_{k_0} f(X(s_2),k_0)-L^{j_1}_{i_0} f(X(s_2),i_0))d[M_{i_0k_0}](s_2)
\\
&+\sum_{i_0\neq k_0}\int_s^{s_1}\mathbbm{1}\{N^{(s,s_1]}>\mu\}(L^{j_1}_{k_0} f(X(s_2),k_0)-L^{j_1}_{i_0} f(X(s_2),i_0))d[M_{i_0k_0}](s_2)\Big)dW^{j_1}(s_1)
\\
=&I_{\beta}[f(X(s),\alpha(s))]_{s,t}+\sum_{j\in\{N_1,\ldots,N_{\mu },\bar{N}_{\mu },0,1,\ldots,m\}} I_{(j)\star\beta}[f(X(\cdot),\alpha(\cdot))]_{s,t} 
\end{align*}
almost surely for any $s<t\in[0,T]$.

Similarly, let $l(\beta)=1$ and $\beta=(j_1)\in\mathcal{M}_3$ for $j_1\in\{N_1, \ldots, N_\mu, \bar{N}_\mu\}$. Then, on using the definition of multiple integrals from Subsection \ref{sub:multiple integral} and Corollary \ref{cor:Ito}, one has 
\begin{align*}
I_{\beta}[f(X(\cdot),\alpha(\cdot))]_{s,t}=&\sum_{i_0\neq k_0}\int_s^t\mathbbm{I}^{j_1}_{(s,t]}(f(X(s_1),k_0)-f(X(s_1),i_0))d[M_{i_0k_0}](s_1)
\\
=&\sum_{i_0\neq k_0}\int_s^t\mathbbm{I}^{j_1}_{(s,t]}\Big((f(X(s),k_0)-f(X(s),i_0))
\\
&+\sum_{j=0}^m\int_s^{s_1}(L^j_{\alpha(s_2)}f(X(s_2),k_0)-L^j_{\alpha(s_2)}f(X(s_2),i_0))dW^j(s_2)\Big)d[M_{i_0k_0}](s_1)
\\
=&I_{\beta}[f(X(s),\alpha(\cdot))]_{s,t}+\sum_{j=0}^m I_{(j)\star\beta}[f(X(\cdot),\alpha(\cdot))]_{s,t}
\end{align*}
almost surely for any $s<t\in[0,T]$.

We use mathematical induction to prove the case $l(\beta) \in \{2,3,\ldots\}$. 
First, we establish the base case, \textit{i.e.}, $l(\beta)=2$.
Let $\beta=(j_1,j_2)\in \mathcal{M}_1$.
 Due to the definition of multiple integrals from Subsection \ref{sub:multiple integral} and Lemma \ref{lem:ItowMS}, we obtain
\begin{align*}
I_{\beta}[f(X(\cdot),&\alpha(\cdot))]_{s,t}=\int_s^t\int_s^{s_1} L^{j_1}_{\alpha(s_2)}L^{j_2}_{\alpha(s_2)} f(X(s_2),\alpha(s_2))dW^{j_1}(s_2)dW^{j_2}(s_1)
\\
=&\int_s^t\int_s^{s_1}\Big( L^{j_1}_{\alpha(s)}L^{j_2}_{\alpha(s)} f(X(s),\alpha(s))+\sum_{j=0}^m\int_{s}^{s_2}L^j_{\alpha(s_3)}L^{j_1}_{\alpha(s_3)}L^{j_2}_{\alpha(s_3)} f(X(s_3),\alpha(s_3))dW^j(s_3)
\\
&+\sum_{r=1}^\mu\sum_{i_0\neq k_0}\int_s^{s_2}\mathbbm{1}\{N^{(s,s_2]}=r\}(L^{j_1}_{k_0}L^{j_2}_{k_0} f(X(s_3),k_0)-L^{j_1}_{i_0}L^{j_2}_{i_0}f(X(s_3),i_0))d[M_{i_0k_0}](s_3)
\\
&+\sum_{i_0\neq k_0}\int_s^{s_2}\mathbbm{1}\{N^{(s,s_2]}>\mu\}(L^{j_1}_{k_0}L^{j_2}_{k_0} f(X(s_3),k_0)
\\
&\qquad-L^{j_1}_{i_0}L^{j_2}_{i_0}f(X(s_3),i_0))d[M_{i_0k_0}](s_3)\Big)dW^{j_1}(s_2)dW^{j_2}(s_1)
\\
=&I_{\beta}[f(X(s),\alpha(s))]_{s,t}+\sum_{j\in\{N_1,\ldots,N_{\mu },\bar{N}_{\mu },0,1,\ldots,m\}} I_{(j)\star\beta}[f(X(\cdot),\alpha(\cdot))]_{s,t} 
\end{align*}
almost surely for any $s<t\in[0,T]$.
Now, let $\beta=(j_1,j_2)\in \mathcal{M}_2$.
 As before, the definition of multiple integrals from Subsection \ref{sub:multiple integral} and Lemma \ref{lem:ItowMS} yields
\begin{align*}
I_{\beta}[f(X(\cdot),&\alpha(\cdot))]_{s,t}
\\
=&\sum_{i_0\neq k_0}\int_s^t\mathbbm{I}^{j_2}_{(s,t]}\int_s^{s_1}(L^{j_1}_{\alpha(s_2)}f(X(s_2),k_0)-L^{j_1}_{\alpha(s_2)}f(X(s_2),i_0))dW^{j_1}(s_2)d[M_{i_0k_0}](s_1)
\\
=&\sum_{i_0\neq k_0}\int_s^t\mathbbm{I}^{j_2}_{(s,t]}\int_s^{s_1}\Big((L^{j_1}_{\alpha(s)}f(X(s),k_0)-L^{j_1}_{\alpha(s)}f(X(s),i_0))
\\
&+\sum_{j=0}^m\int_{s}^{s_2}L^j_{\alpha(s_3)}(L^{j_1}_{\alpha(s_3)}f(X(s_3),k_0)-L^{j_1}_{\alpha(s_3)}f(X(s_3),i_0))dW^j(s_3)
\\
&+\sum_{r=1}^\mu\sum_{i_1\neq k_1}\int_s^{s_2}\mathbbm{1}\{N^{(s,s_2]}=r\}((L^{j_1}_{k_1}f(X(s_3),k_0)-L^{j_1}_{i_1}f(X(s_3),k_0)))
\\
&-(L^{j_1}_{k_1}f(X(s_3),i_0)-L^{j_1}_{i_1}f(X(s_3),i_0)))d[M_{i_0k_0}](s_3)
\\
&+\sum_{i_1\neq k_1}\int_s^{s_2}\mathbbm{1}\{N^{(s,s_2]}>\mu\}((L^{j_1}_{k_1}f(X(s_3),k_0)-L^{j_1}_{i_1}f(X(s_3),k_0)))
\\
&-(L^{j_1}_{k_1}f(X(s_3),i_0)-L^{j_1}_{i_1}f(X(s_3),i_0)))d[M_{i_0k_0}](s_3)\Big)dW^{j_1}(s_2)d[M_{i_0k_0}](s_1)
\\
=&I_{\beta}[f(X(s),\alpha(\cdot))]_{s,t}+\sum_{j\in\{N_1,\ldots,N_{\mu },\bar{N}_{\mu },0,1,\ldots,m\}} I_{(j)\star\beta}[f(X(\cdot),\alpha(\cdot))]_{s,t}
\end{align*}
almost surely for any $s<t\in[0,T]$.
Further, let $\beta=(j_1,j_2)\in \mathcal{M}_3$.
By the definition of multiple integrals from Subsection \ref{sub:multiple integral} and Corollary \ref{cor:Ito}, we have
\begin{align*}
I_{\beta}&[f(X(\cdot),\alpha(\cdot))]_{s,t}=\int_s^t\sum_{i_0\neq k_0}\int_s^{s_1}\mathbbm{I}_{(s,s_1]}^{j_1} (L^{j_2}_{k_0} f(X(s_2),k_0)-L^{j_2}_{i_0} f(X(s_2),i_0))d[M_{i_0k_0}](s_2)dW^{j_2}(s_1)
\\
=&\int_s^t\sum_{i_0\neq k_0}\int_s^{s_1}\mathbbm{I}_{(s,s_1]}^{j_1}\Big( (L^{j_2}_{k_0} f(X(s),k_0)-L^{j_2}_{i_0} f(X(s),i_0))
\\
&+\sum_{j=0}^m\int_s^{s_2}(L^j_{\alpha(s_3)}L^{j_2}_{k_0} f(X(s_3),k_0)-L^j_{\alpha(s_3)}L^{j_2}_{i_0} f(X(s_3),i_0))dW^j(s_3)\Big)d[M_{i_0k_0}](s_2)dW^{j_2}(s_1)
\\
=&I_{\beta}[f(X(s),\alpha(\cdot))]_{s,t}+\sum_{j=0}^m I_{(j)\star\beta}[f(X(\cdot),\alpha(\cdot))]_{s,t}
\end{align*}
almost surely for any $s<t\in[0,T]$.

For the inductive arguments, we assume that the result holds for $l(\beta)=k-1\geq 2$. 
Now, consider  $\beta=(j_1,\ldots,j_k)\in\mathcal{M}$ and $k\geq 3$.
 \newline
 \textbf{Case 1.} Let $\beta\in\mathcal{M}_1 $.
Then, the definition of multiple integrals from Subsection \ref{sub:multiple integral} gives
 \begin{align*}
  I_{\beta}[f(X(\cdot),\alpha(\cdot))]_{s,t}=&\int_s^tI_{\beta-}[L^{j_k}_{\alpha(\cdot)}f(X(\cdot),\alpha(\cdot))]_{s,u}dW^{j_k}(u)
 \end{align*}
 almost surely for any $s<t\in[0,T]$.
Notice that $\beta-\in\mathcal{M}_1$ and the  inductive hypothesis is 
 \begin{align}
 I_{\beta-}[L^{j_k}_{\alpha(\cdot)}f(X(\cdot),\alpha(\cdot))]_{s,u}=&I_{\beta-}[L^{j_k}_{\alpha(s)}f(X(s),\alpha(s))]_{s,u}\notag
 \\
 &+\sum_{j\in\{N_1,\ldots,N_{\mu },\bar{N}_{\mu },0,1,\ldots,m\}} I_{(j)\star\beta-}[L^{j_k}_{\alpha(\cdot)}f(X(\cdot),\alpha(\cdot))]_{s,u}\notag
 \end{align}
 which by the definition of multiple integrals from Subsection \ref{sub:multiple integral} yields,
  \begin{align*}
I_{\beta}[f(X(\cdot),\alpha(\cdot))]_{s,t} =&\int_s^t\Big(I_{\beta-}[L^{j_k}_{\alpha(s)}f(X(s),\alpha(s))]_{s,u}
 \\
 &+\sum_{j\in\{N_1,\ldots,N_{\mu },\bar{N}_{\mu },0,1,\ldots,m\}} I_{(j)\star\beta-}[L^{j_k}_{\alpha(\cdot)}f(X(\cdot),\alpha(\cdot))]_{s,u}\Big)dW^{j_k}(u)
 \\
 =& I_{\beta}[f(X(s),\alpha(s))]_{s,t}+\sum_{j\in\{N_1,\ldots,N_{\mu },\bar{N}_{\mu },0,1,\ldots,m\}} I_{(j)\star\beta}[f(X(\cdot),\alpha(\cdot))]_{s,t}
 \end{align*}
almost surely for any $s<t\in[0,T]$.
\newline 
\textbf{Case 2.}
Let $\beta\in\mathcal{M}_2$ and $j_k\notin\{N_1,\ldots,N_{\mu },\bar{N}_{\mu }\}$.
Due to the definition of multiple integrals from Subsection \ref{sub:multiple integral}, we obtain
\begin{align*}
 I_{\beta}[f(X(\cdot),\alpha(\cdot))]_{s,t}=&\int_s^tI_{\beta-}[L^{j_k}_{\alpha(\cdot)}f(X(\cdot),\alpha(\cdot))]_{s,u}dW^{j_k}(u)
\end{align*}
almost surely for any $s<t\in[0,T]$.
As before, notice that $\beta-\in\mathcal{M}_2$ and the inductive hypothesis is
\begin{align}
I_{\beta-}[L^{j_k}_{\alpha(\cdot)}f(X(\cdot),\alpha(\cdot))]_{s,u}=&I_{\beta-}[L^{j_k}_{\alpha(\cdot)}f(X(s),\alpha(\cdot))]_{s,u}\notag
 \\
 &+\sum_{j\in\{N_1,\ldots,N_{\mu },\bar{N}_{\mu },0,1,\ldots,m\}} I_{(j)\star\beta-}[L^{j_k}_{\alpha(\cdot)}f(X(\cdot),\alpha(\cdot))]_{s,u} \notag
\end{align}
almost surely for any $s<u\in[0,T]$. 
Thus, by applying the above equation and the  definition of multiple integrals from Subsection \ref{sub:multiple integral}, one has
 \begin{align*}
 I_{\beta}[f(X(\cdot),\alpha(\cdot))]_{s,t}=&\int_s^t\Big(I_{\beta-}[L^{j_k}_{\alpha(\cdot)}f(X(s),\alpha(\cdot))]_{s,u}
 \\
 &+\sum_{j\in\{N_1,\ldots,N_{\mu },\bar{N}_{\mu },0,1,\ldots,m\}} I_{(j)\star\beta-}[L^{j_k}_{\alpha(\cdot)}f(X(\cdot),\alpha(\cdot))]_{s,u}\Big)dW^{j_k}(u)
 \\
 =&I_{\beta}[f(X(s),\alpha(\cdot))]_{s,t}+\sum_{j\in\{N_1,\ldots,N_{\mu },\bar{N}_{\mu },0,1,\ldots,m\}} I_{(j)\star\beta}[f(X(\cdot),\alpha(\cdot))]_{s,t}
 \end{align*}
  almost surely for any $s<t\in[0,T]$.
 \newline
 \textbf{Case 3.} Let $\beta\in\mathcal{M}_2 $ and $j_k\in\{N_1,\ldots,N_{\mu },\bar{N}_{\mu }\}$.
By using the definition of multiple integrals from Subsection \ref{sub:multiple integral}, we can write
\begin{align*}
I_{\beta}[f(X(\cdot),\alpha(\cdot))]_{s,t}=&\sum_{i_0\neq k_0}\int_s^t\mathbbm{I}_{(s,t]}^{j_k}I_{\beta-}[f(X(\cdot),k_0)-f(X(\cdot),i_0)]_{s,u}d[M_{i_0k_0}](u)
\end{align*}
almost surely for any $s<t\in[0,T]$.
Notice that $\beta-\in\mathcal{M}_1\cup\mathcal{M}_2$ and the inductive hypothesis is 
\begin{align}
I_{\beta-}[f(X(\cdot),k_0)&-f(X(\cdot),i_0)]_{s,u}=I_{\beta-}[f(X(s),k_0)-f(X(s),i_0)]_{s,u}\notag
 \\
 &+\sum_{j\in\{N_1,\ldots,N_{\mu },\bar{N}_{\mu },0,1,\ldots,m\}} I_{(j)\star\beta-}[f(X(\cdot),k_0)-f(X(\cdot),i_0)]_{s,u}\notag
\end{align}
almost surely for any $s<u\in[0,T]$.
Further, above equation and  the definition of multiple integrals from Subsection \ref{sub:multiple integral} yields, 
 \begin{align*}
 I_{\beta}[f(X(\cdot),\alpha(\cdot))]_{s,t}=&\sum_{i_0\neq k_0}\int_s^t\mathbbm{I}_{(s,t]}^{j_k}\Big(I_{\beta-}[f(X(s),k_0)-f(X(s),i_0)]_{s,u}
 \\
 &+\sum_{j\in\{N_1,\ldots,N_{\mu },\bar{N}_{\mu },0,1,\ldots,m\}} I_{(j)\star\beta-}[f(X(\cdot),k_0)-f(X(\cdot),i_0)]_{s,u}\Big)d[M_{i_0k_0}](u)
 \\
  =&I_{\beta}[f(X(s),\alpha(\cdot))]_{s,t}+\sum_{j\in\{N_1,\ldots,N_{\mu },\bar{N}_{\mu },0,1,\ldots,m\}} I_{(j)\star\beta}[f(X(\cdot),\alpha(\cdot))]_{s,t}
 \end{align*}
  almost surely for any $s<t\in[0,T]$.
 \newline
 \textbf{Case 4.} Let $\beta\in\mathcal{M}_3 $ and $j_k\notin\{N_1,\ldots,N_{\mu },\bar{N}_{\mu }\}$.
 By the  definition of multiple integrals from Subsection \ref{sub:multiple integral}, we have 
 \begin{align*}
  I_{\beta}[f(X(\cdot)&,\alpha(\cdot))]_{s,t}=\int_s^tI_{\beta-}[L^{j_k}_{\alpha(\cdot)}f(X(\cdot),\alpha(\cdot))]_{s,u}dW^{j_k}(u)
 \end{align*}
 almost surely for any $s<t\in[0,T]$.
 In this case, $\beta-\in\mathcal{M}_3$ and the inductive hypothesis is
 \begin{align*}
 I_{\beta-}[L^{j_k}_{\alpha(\cdot)}f(X(\cdot),\alpha(\cdot))]_{s,u}=&I_{\beta-}[L^{j_k}_{\alpha(\cdot)}f(X(s),\alpha(\cdot))]_{s,u}+\sum_{j=0}^{m} I_{(j)\star\beta-}[L^{j_k}_{\alpha(\cdot)}f(X(\cdot),\alpha(\cdot))]_{s,u}
 \end{align*}
 almost surely for any $s<u\in[0,T]$.
By the definition of multiple integrals from Subsection \ref{sub:multiple integral} and inductive hypothesis, we have 
  \begin{align*}
 I_{\beta}[f(X(\cdot),\alpha(\cdot))]_{s,t} =&\int_s^t\Big(I_{\beta-}[L^{j_k}_{\alpha(\cdot)}f(X(s),\alpha(\cdot))]_{s,u}
 \\
 &+\sum_{j=0}^{m} I_{(j)\star\beta-}[L^{j_k}_{\alpha(\cdot)}f(X(\cdot),\alpha(\cdot))]_{s,u}\Big)dW^{j_k}(u)
 \\
 =&I_{\beta}[f(X(s),\alpha(\cdot))]_{s,t}+\sum_{j=0}^{m} I_{(j)\star\beta}[f(X(\cdot),\alpha(\cdot))]_{s,t}
 \end{align*}
  almost surely for any $s<t\in[0,T]$.
  \newline
 \textbf{Case 5.} Let $\beta\in\mathcal{M}_3 $ and $j_k\in\{N_1,\ldots,N_{\mu },\bar{N}_{\mu }\}$.
The definition of multiple integrals from Subsection \ref{sub:multiple integral} gives 
\begin{align*}
 I_{\beta}[f(X(\cdot),\alpha(\cdot))]_{s,t}=&\sum_{i_0\neq k_0}\int_s^t\mathbbm{I}_{(s,t]}^{j_k}I_{\beta-}[f(X(\cdot),k_0)-f(X(\cdot),i_0)]_{s,u}d[M_{i_0k_0}](u)
\end{align*}
almost surely for any $s<t\in[0,T]$.
Notice that $\beta-\in\mathcal{M}_3$, then the inductive hypothesis is written as 
\begin{align}
I_{\beta-}[f(X(\cdot),k_0)-f(X(\cdot),i_0)]_{s,u}=&I_{\beta-}[f(X(s),k_0)-f(X(s),i_0)]_{s,u}\notag
 \\
 &+\sum_{j=0}^{m} I_{(j)\star\beta-}[f(X(\cdot),k_0)-f(X(\cdot),i_0)]_{s,u}\notag
\end{align}
almost surely for any $s<u\in[0,T]$.
Hence, the above equation and the definition of multiple integrals from Subsection \ref{sub:multiple integral} gives
 \begin{align*}
 I_{\beta}[f(X(\cdot),\alpha(\cdot))]_{s,t}=&\sum_{i_0\neq k_0}\int_s^t\mathbbm{I}_{(s,t]}^{j_k}\Big(I_{\beta-}[f(X(s),k_0)-f(X(s),i_0)]_{s,u}
 \\
 &+\sum_{j=0}^{m} I_{(j)\star\beta-}[f(X(\cdot),k_0)-f(X(\cdot),i_0)]_{s,u}\Big)d[M_{i_0k_0}](u)
 \\
  =&I_{\beta}[f(X(\cdot),\alpha(\cdot))]_{s,t}+\sum_{j=0}^{m} I_{(j)\star\beta}[f(X(\cdot),\alpha(\cdot))]_{s,t}
 \end{align*}
  almost surely for any $s<t\in[0,T]$. This completes the proof.
\end{proof}

Now we will prove one of the main result of this article, \textit{i.e.}, theorem \ref{thm:Ito-Taylor}.
\begin{proof}[\textbf{Proof of Theorem \ref{thm:Ito-Taylor}}]
If $\mathcal{A}=\varphi$, then the proof  is trivial.
We shall prove Theorem \ref{thm:Ito-Taylor} for $\mathcal{A}\neq \varphi$ by induction on
\begin{equation}
\eta(\mathcal{A})=\sup_{\beta\in\mathcal{A}}\eta(\beta) \nonumber
\end{equation}
where $\eta(\beta)$ is defined in Subsection \ref{sub:Multi-indices}. 
Let $\eta(\mathcal{A})=0$.
Then, $\mathcal{A}=\{\nu\}$ with the remainder set
\begin{itemize}
\item[] $\mathcal{B}(\mathcal{A})=\big\{(N_1),\ldots,(N_{\mu}),(\bar{N}_{\mu}),(0),(1),\ldots,(m)\big\}$ and  $\tilde{\mathcal{A}}=\varphi$.
\end{itemize}
Thus, on using Lemma \ref{lem:ItowMS} and the definition of multiple integrals from Subsection \ref{sub:multiple integral}, one has
\begin{align*}
f(X(t),\alpha(t))&=f(X(s),\alpha(s))+\int^t_sL^{0}_{\alpha(u)}f(X(u),\alpha(u))du+\sum_{j=1}^m\int^t_sL^j_{\alpha(u)}f(X(u),\alpha(u))dW^{j}(u)
\\
&+\sum_{r=1}^{\mu}\sum_{i_0\neq k_0}\int^t_s\mathbbm{1}\{N^{(s,t]}=r\} (f(X(u),k_0)-f(X(u),i_0))d[M_{i_0k_0}](u)
\\
&+\sum_{i_0\neq k_0}\int^t_s\mathbbm{1}\{N^{(s,t]}>\mu\} (f(X(u),k_0)-f(X(u),i_0))d[M_{i_0k_0}](u)
\\
=&\sum_{\beta\in\mathcal{A}\setminus\tilde{\mathcal{A}}}I_{\beta}[f(X(s),\alpha(s))]_{s,t}+\sum_{\beta\in\tilde{\mathcal{A}}}I_{\beta}[f(X(s),\alpha(\cdot))]_{s,t}+\sum_{\beta\in\mathcal{B}(\mathcal{A})}I_{\beta}[f(X(\cdot),\alpha(\cdot))]_{s,t} \nonumber 
\end{align*}
almost surely for all $s<t\in[0,T]$.
Thus, the result holds for $\eta(\mathcal{A})=0$. 

Now, let $\eta(\mathcal{A})\geq 1$ and  $\zeta:=\{\beta\in\mathcal{A}:\:\eta(\beta)\leq \eta(\mathcal{A})-1\}$.
Notice that $\zeta$ is a hierarchical set.
Indeed, $\nu \in \zeta$ because $\nu\in \mathcal{A}$ and  $\eta(\nu)=0\leq \eta(\mathcal{A})-1$. 
Also, for any $\beta \in \zeta\setminus \{\nu\}$, $\beta\in \mathcal{A}$ and $\eta(\beta)\leq \eta(\mathcal{A})-1$ which implies $-\beta\in \mathcal{A}$ and $\eta(-\beta)\leq \eta(\beta)\leq \eta(\mathcal{A})-1$. 
Further, the inductive hypothesis is
\begin{align}
f(X(t),\alpha(t))=&\sum_{\beta\in \zeta\setminus\tilde{\zeta}}I_{\beta}\big[f(X(s),\alpha(s))\big]_{s,t}+\sum_{\beta\in\tilde{\zeta}}I_{\beta}\big[f(X(s),\alpha(\cdot))\big]_{s,t}\nonumber
\\
&+\sum_{\beta\in\mathcal{B}(\zeta)}I_{\beta}\big[f(X(\cdot),\alpha(\cdot))\big]_{s,t} \label{eq:ind.Ito Taylor Proof}
\end{align}
almost surely for all $s<t\in[0,T]$ where  $\tilde{\zeta}:=\{(j_1,\ldots,j_l)\in\zeta: j_i\in\{ N_1,\ldots,N_{\mu},\bar{N}_{\mu}\}\text{ for any }i\in\{1,\ldots,l\} \}$.

We write $\mathcal{B}(\zeta)$ as $\mathcal{B}(\zeta)=\bar{\mathcal{B}_1}\cup\bar{\mathcal{B}}_2\cup\bar{\mathcal{B}}_3\cup(\mathcal{B}(\zeta)\setminus(\bar{\mathcal{B}_1}\cup\bar{\mathcal{B}}_2\cup\bar{\mathcal{B}}_3))$ where
\begin{itemize}
\item[] $\bar{\mathcal{B}}_1:=\{(j_1,\ldots,j_l)\in\mathcal{B}(\zeta) \cap \mathcal{A}: j_i\in\{0,1,\ldots,m\}\forall\, i\in\{1,\ldots,l\}\},$
\item[] $\bar{\mathcal{B}}_2:=\{(j_1,\ldots,j_l)\in \mathcal{B}(\zeta) \cap \mathcal{A}:  j_1\in\{N_1,\ldots,N_{\mu },\bar{N}_{\mu}\}$,
\item[] $\bar{\mathcal{B}_3}:=\{(j_1,\ldots,j_l)\in(\mathcal{B}(\zeta)\cap \mathcal{A})\setminus\bar{\mathcal{B}}_1:  j_1\in\{0,1,\ldots,m\}\}$.
\end{itemize}

Notice that $\bar{\mathcal{B}}_1\subset\mathcal{M}_1$, $\bar{\mathcal{B}}_2\subset\mathcal{M}_3$ and $\bar{\mathcal{B}}_3\subset\mathcal{M}_2$.
 By using Lemma \ref{lem:Ito on beta} for $\beta\in\bar{\mathcal{B}}_1$, we have
\begin{align}
\sum_{\beta\in\bar{\mathcal{B}}_1}I_{\beta}[f(X(\cdot),\alpha(\cdot))]_{s,t}=&\sum_{\beta\in\bar{\mathcal{B}}_1}I_{\beta}[f(X(s),\alpha(s))]_{s,t}\nonumber
\\
&+\sum_{\beta\in\bar{\mathcal{B}}_1}\sum_{j\in\{N_1,\ldots,N_{\mu },\bar{N}_{\mu },0,1,\ldots,m\}} I_{(j)\star \beta}[f(X(\cdot),\alpha(\cdot))]_{s,t}\nonumber
\\
=&\sum_{\beta\in\bar{\mathcal{B}}_1}I_{\beta}[f(X(s),\alpha(s))]_{s,t}+\sum_{\beta\in\tilde{\bar{\mathcal{B}}}_1}I_{\beta}[f(X(\cdot),\alpha(\cdot))]_{s,t} \label{eq:beta1 in ito}
\end{align}
almost surely for all $s<t\in[0,T]$ where 
\begin{itemize}
\item[] $\tilde{\bar{\mathcal{B}}}_1:=\{(j)\star \beta:\beta\in\bar{\mathcal{B}}_1,j\in\{N_1,\ldots,N_{\mu},\bar{N}_{\mu },0,1,\ldots,m\}\}$
\end{itemize}
and $\eta(\beta)\geq\eta(\mathcal{A})+1$ for all $\beta\in\tilde{\bar{\mathcal{B}}}_1$.
Further, for $\beta\in\bar{\mathcal{B}}_2$, Lemma \ref{lem:Ito on beta}  yields
\begin{align}
\sum_{\beta\in\bar{\mathcal{B}}_2}I_{\beta}[f(X(\cdot),\alpha(\cdot))]_{s,t}=&\sum_{\beta\in\bar{\mathcal{B}}_2}I_{\beta}[f(X(s),\alpha(\cdot))]_{s,t}+\sum_{\beta\in\bar{\mathcal{B}}_2}\sum_{j=0}^m I_{(j)\star \beta}[f(X(\cdot),\alpha(\cdot))]_{s,t} \nonumber
\\
=&\sum_{\beta\in\bar{\mathcal{B}}_2}I_{\beta}[f(X(s),\alpha(\cdot))]_{s,t}+\sum_{\beta\in\tilde{\bar{\mathcal{B}}}_2}I_{\beta}[f(X(\cdot),\alpha(\cdot))]_{s,t}\label{eq:beta2 in ito}
\end{align}
almost surely for all $s<t\in[0,T]$ where 
\begin{itemize}
\item[] $\tilde{\bar{\mathcal{B}}}_2:=\{(j)\star \beta:\beta\in \bar{\mathcal{B}}_2,j\in\{0,1,2,\ldots,m\}\}$.
\end{itemize}
and $\eta(\beta)\geq\eta(\mathcal{A})+1$ for all $\beta\in\tilde{\bar{\mathcal{B}}}_2$.
Similarly, on using Lemma \ref{lem:Ito on beta} for $\beta\in\bar{\mathcal{B}}_3$, we obtain
\begin{align*}
\sum_{\beta\in\bar{\mathcal{B}_3}}I_{\beta}[f(X(\cdot),\alpha(\cdot))]_{s,t}=&\sum_{\beta\in\bar{\mathcal{B}_3}}I_{\beta}[f(X(s),\alpha(\cdot))]_{s,t}\nonumber
\\
&+\sum_{\beta\in\bar{\mathcal{B}_3}}\sum_{j\in\{N_1,\ldots,N_{\mu },\bar{N}_{\mu },0,1,\ldots,m\}} I_{(j)\star \beta}[f(X(\cdot),\alpha(\cdot))]_{s,t} 
\\
=&\sum_{\beta\in\bar{\mathcal{B}_3}}I_{\beta}[f(X(s),\alpha(\cdot))]_{s,t}+\sum_{\beta\in\tilde{\bar{\mathcal{B}}}_3}I_{\beta}[f(X(\cdot),\alpha(\cdot))]_{s,t}
\end{align*} 
almost surely for all $s<t\in[0,T]$ where
\begin{itemize}
\item[] $\tilde{\bar{\mathcal{B}}}_3:=\{(j)\star \beta:\beta\in\bar{\mathcal{B}_3},j\in\{N_1,\ldots,N_{\mu},\bar{N}_{\mu},0,1,\ldots,m\}\}$.
\end{itemize}
Now notice that there may be some $\beta\in\tilde{\bar{\mathcal{B}}}_3$ such that $\eta(\beta)=\eta(\mathcal{A})$ and $\beta\in\mathcal{A}$. Thus, we define
\begin{itemize}
\item[] $\bar{\mathcal{B}}_4:=\tilde{\bar{\mathcal{B}}}_3\cap \mathcal{A}$. 
\end{itemize}
Clearly, $\bar{\mathcal{B}}_4\subset\mathcal{M}_3$.
By the application of Lemma \ref{lem:Ito on beta} for $\beta\in \bar{\mathcal{B}}_4$, the above equation can be written as
\begin{align}
\sum_{\beta\in\bar{\mathcal{B}_3}}I_{\beta}[f(X(\cdot),\alpha(\cdot))]_{s,t}=&\sum_{\beta\in\bar{\mathcal{B}_3}}I_{\beta}[f(X(s),\alpha(\cdot))]_{s,t}+\sum_{\beta\in\bar{\mathcal{B}}_4}I_{\beta}[f(X(s),\alpha(\cdot))]_{s,t} \nonumber
\\
& +\sum_{\beta\in\bar{\mathcal{B}}_4}\sum_{j=0}^m I_{(j)\star \beta}[f(X(\cdot),\alpha(\cdot))]_{s,t}+\sum_{\beta\in\tilde{\bar{\mathcal{B}}}_3\setminus\bar{\mathcal{B}}_4}I_{\beta}[f(X(\cdot),\alpha(\cdot))]_{s,t}\nonumber
\\
=&\sum_{\beta\in\bar{\mathcal{B}_3}}I_{\beta}[f(X(s),\alpha(\cdot))]_{s,t}+\sum_{\beta\in\bar{\mathcal{B}}_4}I_{\beta}[f(X(s),\alpha(\cdot))]_{s,t}\nonumber
\\
&+\sum_{\beta\in\tilde{\bar{\mathcal{B}}}_4}I_{\beta}[f(X(\cdot),\alpha(\cdot))]_{s,t}+\sum_{\beta\in\tilde{\bar{\mathcal{B}}}_3\setminus\bar{\mathcal{B}}_4}I_{\beta}[f(X(\cdot),\alpha(\cdot))]_{s,t}\label{eq:beta3 in ito}
\end{align}
almost surely for all $s<t\in[0,T]$ where $\tilde{\bar{\mathcal{B}}}_4:=\big\{(j)\star \beta: \beta\in \bar{\mathcal{B}}_4,\:j\in\{0,1,2,\ldots,m\}\big\}$ with $\eta(\beta)\geq\eta(\mathcal{A})+1$.

Notice that, $\zeta\cup\bar{\mathcal{B}}_1\cup\bar{\mathcal{B}}_2\cup\bar{\mathcal{B}}_3\cup\bar{\mathcal{B}}_4\subset\mathcal{A}$. 
Clearly $\nu\in\zeta$.
Let us consider $\beta\in\mathcal{A}\setminus\{\nu\}$ to prove $\mathcal{A}\subset\zeta\cup\bar{\mathcal{B}}_1\cup\bar{\mathcal{B}}_2\cup\bar{\mathcal{B}}_3\cup\bar{\mathcal{B}}_4$.
If $\eta(\beta)\leq \eta(\mathcal{A})-1$, then $\beta\in\zeta$.
Now, let $\eta(\beta)=\eta(\mathcal{A})$ which implies either $\eta(-\beta)<\eta(\beta)$ or $\eta(-\beta)=\eta(\beta)$.
Let $\eta(-\beta)<\eta(\beta)$, then the definition of hierarchical and remainder sets from Subsection \ref{sub:hierarchical and reminder} yields $-\beta\in\zeta$ and $\beta\in\mathcal{B}(\zeta)$ which gives $\beta\in\bar{\mathcal{B}}_1\cup\bar{\mathcal{B}}_2\cup\bar{\mathcal{B}}_3$.
Further, if $\eta(-\beta)=\eta(\beta)$, then  $-\beta\in\bar{\mathcal{B}}_3$ and $\beta\in\bar{\mathcal{B}}_4$.
Hence 
\begin{itemize}
\item[] $\mathcal{A}=\zeta\cup\bar{\mathcal{B}}_1\cup\bar{\mathcal{B}}_2\cup\bar{\mathcal{B}}_3\cup\bar{\mathcal{B}}_4$ and $\tilde{\mathcal{A}}=\zeta\cup\bar{\mathcal{B}}_2\cup\bar{\mathcal{B}}_3\cup\bar{\mathcal{B}}_4$.
\end{itemize}
Moreover, 
\begin{itemize}
\item[] $(\mathcal{B}(\zeta)\setminus (\bar{\mathcal{B}}_1\cup\bar{\mathcal{B}}_2\cup\bar{\mathcal{B}}_3))\cup\tilde{\bar{\mathcal{B}}}_1\cup\tilde{\bar{\mathcal{B}}}_2\cup(\tilde{\bar{\mathcal{B}}}_3\setminus\bar{\mathcal{B}}_4)\cup\tilde{\bar{\mathcal{B}}}_4=\{\beta\in\mathcal{M}\setminus\mathcal{A} :-\beta\in\zeta\}\cup\{\beta\in\mathcal{M}\setminus\mathcal{A}:-\beta\in\bar{\mathcal{B}}_{1}\}\cup\{\beta\in\mathcal{M}\setminus\mathcal{A}:-\beta\in\bar{\mathcal{B}}_{2}\}\cup\{\beta\in\mathcal{M}\setminus\mathcal{A}:-\beta\in\bar{\mathcal{B}}_{3}\}\cup\{\beta\in\mathcal{M}\setminus\mathcal{A}:-\beta\in\bar{\mathcal{B}}_4\}=\mathcal{B}(\mathcal{A})$.
\end{itemize}

 Thus, by using \eqref{eq:beta1 in ito}, \eqref{eq:beta2 in ito} and \eqref{eq:beta3 in ito} in \eqref{eq:ind.Ito Taylor Proof}, we obtain 
 \begin{align*}
 f(X(t),\alpha(t))=&\sum_{\beta\in\mathcal{A}\setminus\tilde{\mathcal{A}}}I_{\beta}[f(X(s),\alpha(s))]_{s,t}+\sum_{\beta\in\tilde{\mathcal{A}}}I_{\beta}[f(X(s),\alpha(\cdot))]_{s,t}+\sum_{\beta\in\mathcal{B}(\mathcal{A})}I_{\beta}[f(X(\cdot),\alpha(\cdot))]_{s,t}
 \end{align*}
almost surely for all $s<t\in[0,T]$. This completes the proof.
\end{proof} 
\begin{remark}
If there is no jump then $\tilde{\zeta},\bar{\mathcal{B}}_2,\bar{\mathcal{B}}_3,\tilde{\bar{\mathcal{B}}}_2,\tilde{\bar{\mathcal{B}}}_3,\bar{\mathcal{B}}_4,\tilde{\bar{\mathcal{B}}}_4$ in preceding proof are empty sets and SDEwMS \eqref{eq:sdems} reduces to an  SDE.
As a result, by repeating the preceding steps, we can prove Theorem 5.5.1 of \cite{kloeden1992numerical}, \textit{i.e.}, It\^o-Taylor expansion for SDE.
\end{remark}
\section{Explicit $\gamma$-order  scheme for SDEwMS}
\label{sec:scheme der., moment, convergence}
This section is devoted to the derivation of the $\gamma\in \{n/2:n\in\mathbb{N}\}$-order numerical scheme for SDEwMS \eqref{eq:sdems} and  prove its   moment stability  and  strong rate of convergence.

\subsection{Derivation}
\label{sec:Derivation of scheme}
In this subsection, we present the derivation of the explicit $\gamma\in \{n/2:n\in\mathbb{N}\}$-order scheme for SDEwMS \eqref{eq:sdems}.
For this, we recall the definitions of  hierarchical sets $\mathcal{A}_{\gamma-0.5}^b$ and $\mathcal{A}_{\gamma-0.5}^\sigma$ from Subsection \ref{sub:scheme} and  $\mathcal{A}_{0}^b=\mathcal{A}_{0}^\sigma=\varphi$. 
Then, we use Theorem \ref{thm:Ito-Taylor} to obtain, 
\begin{align*}
b^k(X(s),\alpha(s))=&\sum_{\beta\in\mathcal{A}_{\gamma-0.5}^b\setminus\tilde{\mathcal{A}}_{\gamma-0.5}^b}I_{\beta}[b^k(X(t_n),\alpha(t_n))]_{t_n,s}
+\sum_{\beta\in\tilde{\mathcal{A}}_{\gamma-0.5}^b}I_{\beta}[b^k(X(t_n),\alpha(\cdot))]_{t_n,s}
\\
&+\sum_{\beta\in\mathcal{B}(\mathcal{A}_{\gamma-0.5}^b)}I_{\beta}[b^k(X(\cdot),\alpha(\cdot))]_{t_n,s}
\\
\sigma^{(k,j)}(X(s),\alpha(s))=&\sum_{\beta\in\mathcal{A}_{\gamma-0.5}^\sigma\setminus\tilde{\mathcal{A}}_{\gamma-0.5}^\sigma}I_{\beta}[\sigma^{(k,j)}(X(t_n),\alpha(t_n))]_{t_n,s}+\sum_{\beta\in\tilde{\mathcal{A}}_{\gamma-0.5}^\sigma}I_{\beta}[\sigma^{(k,j)}(X(t_n),\alpha(\cdot))]_{t_n,s}
\\
&+\sum_{\beta\in\mathcal{B}(\mathcal{A}_{\gamma-0.5}^\sigma)}I_{\beta}[\sigma^{(k,j)}(X(\cdot),\alpha(\cdot))]_{t_n,s}
\end{align*}
almost surely for all $n\in\{0,1,\ldots,n_T-1\}$, $j\in\{1,\ldots,m\}$, $k\in\{1,\ldots,d\}$ and $s\in [t_n, t_{n+1}]$.
Thus, from \eqref{eq:sdems}, we have, 
\begin{align*}
X^k(t_{n+1}&)=X^k(t_{n})+\int_{t_n}^{t_{n+1}}\Big(\sum_{\beta\in\mathcal{A}_{\gamma-0.5}^b\setminus\tilde{\mathcal{A}}_{\gamma-0.5}^b}I_{\beta}[b^k(X(t_n),\alpha(t_n))]_{t_n,s}
\\
&+\sum_{\beta\in\tilde{\mathcal{A}}_{\gamma-0.5}^b}I_{\beta}[b^k(X(t_n),\alpha(\cdot))]_{t_n,s}
+\sum_{\beta\in\mathcal{B}(\mathcal{A}_{\gamma-0.5}^b)}I_{\beta}[b^k(X(\cdot),\alpha(\cdot))]_{t_n,s}\Big)ds\notag
\\
&+\sum_{j=1}^m\int_{t_n}^{t_{n+1}}\Big(\sum_{\beta\in\mathcal{A}_{\gamma-0.5}^\sigma\setminus\tilde{\mathcal{A}}_{\gamma-0.5}^\sigma}I_{\beta}[\sigma^{(k,j)}(X(t_n),\alpha(t_n))]_{t_n,s}
\\
&+\sum_{\beta\in\tilde{\mathcal{A}}_{\gamma-0.5}^\sigma}I_{\beta}[\sigma^{(k,j)}(X(t_n),\alpha(\cdot))]_{t_n,s}
+\sum_{\beta\in\mathcal{B}(\mathcal{A}_{\gamma-0.5}^\sigma)}I_{\beta}[\sigma^{(k,j)}(X(\cdot),\alpha(\cdot))]_{t_n,s}\Big)dW^j(s)
\end{align*}
almost surely for any $k \in \{1,\ldots,d\}$ and $n \in \{0, \ldots, n_T-1\}$. 
Let $\mu=2\gamma$ be fixed and notice that  the remainder sets $\mathcal{B}(\mathcal{A}_{\gamma-0.5}^b)$ and $\mathcal{B}(\mathcal{A}_{\gamma-0.5}^\sigma)$ contain elements of $\mathcal{A}_\gamma^b$ and  $\mathcal{A}_\gamma^\sigma$ which are identified with the following sets, 
\begin{itemize}
\item[] $\mathcal{B}_1^b:=\{\nu\}\cup\{(j_1,\ldots,j_l)\in\mathcal{B}(\mathcal{A}_{\gamma-0.5}^b)\cap \mathcal{A}_{\gamma}^b:j_i\notin\{N_1,\ldots,N_{\mu}\}\forall i\in\{1,\ldots,l\}\},$
\item[] $\mathcal{B}_1^\sigma:=\{\nu\}\cup\{(j_1,\ldots,j_l)\in\mathcal{B}(\mathcal{A}_{\gamma-0.5}^\sigma)\cap\mathcal{A}_{\gamma}^\sigma:j_i\notin\{N_1,\ldots,N_{\mu}\}\forall i\in\{1,\ldots,l\}\},$
\item[] $\mathcal{B}_2^b:=\{(j_1,\ldots,j_l)\in\mathcal{B}(\mathcal{A}_{\gamma-0.5}^b)\cap \mathcal{A}_{\gamma}^b:j_1\in\{N_1,\ldots,N_{\mu}\}\},$
\item[] $\mathcal{B}_2^\sigma:=\{(j_1,\ldots,j_l)\in\mathcal{B}(\mathcal{A}_{\gamma-0.5}^\sigma) \cap \mathcal{A}_{\gamma}^\sigma: j_1\in\{N_1,\ldots,N_{\mu}\}\},$
\item[] $\mathcal{B}_3^b:=\{(j_1,\ldots,j_l)\in(\mathcal{B}(\mathcal{A}_{\gamma-0.5}^b)\cap \mathcal{A}_{\gamma}^b ) \setminus\mathcal{B}_1^b:j_1\in\{0,1,\ldots,m\}\},$
\item[] $\mathcal{B}_3^\sigma:=\{(j_1,\ldots,j_l)\in( \mathcal{B}(\mathcal{A}_{\gamma-0.5}^\sigma)\cap \mathcal{A}_{\gamma}^\sigma ) \setminus\mathcal{B}_1^\sigma:j_1\in\{0,1,\ldots,m\}\}$. 
\end{itemize}
Clearly,  $\mathcal{B}_2^b, \mathcal{B}_2^\sigma \subset \mathcal{M}_3$ and $\mathcal{B}_3^b, \mathcal{B}_3^\sigma \subset \mathcal{M}_2$. 
Similarly, some elements of $\mathcal{B}(\mathcal{A}_\gamma^b)$ and $\mathcal{B}(\mathcal{A}_\gamma^\sigma)$ are identified by the below mentioned sets, 
\begin{itemize}
\item[] $\mathcal{B}_4^b:= \mathcal{B}(\mathcal{A}_{\gamma-0.5}^b)\setminus(\mathcal{B}_1^b\cup\mathcal{B}_2^b\cup\mathcal{B}_3^b),$
\item[] $\mathcal{B}_4^\sigma:= \mathcal{B}(\mathcal{A}_{\gamma-0.5}^\sigma)\setminus(\mathcal{B}_1^\sigma\cup\mathcal{B}_2^\sigma\cup\mathcal{B}_3^\sigma)$. 
\end{itemize}
Thus,  by using the sets defined above, we write, 
\begin{align*}
X^k(t_{n+1})&=X^k(t_{n})+\int_{t_n}^{t_{n+1}}\Big(\sum_{\beta\in\mathcal{B}_1^b\cup(\mathcal{A}_{\gamma-0.5}^b\setminus\tilde{\mathcal{A}}_{\gamma-0.5}^b)}I_{\beta}[b^k(X(t_n),\alpha(t_n))]_{t_n,s}
\\
&+\sum_{\beta\in\tilde{\mathcal{A}}_{\gamma-0.5}^b}I_{\beta}[b^k(X(t_n),\alpha(\cdot))]_{t_n,s}
+\sum_{\beta\in\mathcal{B}_2^b\cup\mathcal{B}_3^b\cup\mathcal{B}_4^b}I_{\beta}[b^k(X(\cdot),\alpha(\cdot))]_{t_n,s}
\\
&+\sum_{\beta\in\mathcal{B}_1^b}I_{\beta}[b^k(X(\cdot),\alpha(\cdot))-b^k(X(t_n),\alpha(t_n))]_{t_n,s}\Big)ds\notag
\\
&+\sum_{j=1}^m\int_{t_n}^{t_{n+1}}\Big(\sum_{\beta\in\mathcal{B}_1^\sigma\cup(\mathcal{A}_{\gamma-0.5}^\sigma\setminus\tilde{\mathcal{A}}_{\gamma-0.5}^\sigma)}I_{\beta}[\sigma^{(k,j)}(X(t_n),\alpha(t_n))]_{t_n,s}
\\
&+\sum_{\beta\in\tilde{\mathcal{A}}_{\gamma-0.5}^\sigma}I_{\beta}[\sigma^{(k,j)}(X(t_n),\alpha(\cdot))]_{t_n,s}
+\sum_{\beta\in\mathcal{B}_2^\sigma\cup\mathcal{B}_3^\sigma\cup\mathcal{B}_4^\sigma}I_{\beta}[\sigma^{(k,j)}(X(\cdot),\alpha(\cdot))]_{t_n,s}
\\
&+\sum_{\beta\in\mathcal{B}_1^\sigma}I_{\beta}[\sigma^{(k,j)}(X(\cdot),\alpha(\cdot))-\sigma^{(k,j)}(X(t_n),\alpha(t_n))]_{t_n,s}\Big)dW^j(s)
\end{align*}
which on the application of  Lemma \ref{lem:Ito on beta} for $\beta\in\{\mathcal{B}_2^b,\mathcal{B}_3^b,\mathcal{B}_2^\sigma,\mathcal{B}_3^\sigma\}$ yields
\begin{align*}
X^k(t_{n+1})&=X^k(t_{n})+\int_{t_n}^{t_{n+1}}\Big(\sum_{\beta\in\mathcal{B}_1^b\cup(\mathcal{A}_{\gamma-0.5}^b\setminus\tilde{\mathcal{A}}_{\gamma-0.5}^b)}I_{\beta}[b^k(X(t_n),\alpha(t_n))]_{t_n,s}
\\
&+\sum_{\beta\in\tilde{\mathcal{A}}_{\gamma-0.5}^b\cup\mathcal{B}_2^b\cup\mathcal{B}_3^b}I_{\beta}[b^k(X(t_n),\alpha(\cdot))]_{t_n,s}+\sum_{\beta\in\mathcal{B}_2^b}\sum_{k_1=0}^mI_{(k_1)\star\beta}[b^k(X(\cdot),\alpha(\cdot))]_{t_n,s}
\\
&+\sum_{\beta\in\mathcal{B}_3^b}\sum_{k_1\in\{N_1,\ldots,N_{\mu},\bar{N}_{\mu},0,1,\ldots,m\}}I_{(k_1)\star\beta}[b^k(X(\cdot),\alpha(\cdot))]_{t_n,s}+\sum_{\beta\in\mathcal{B}_4^b}I_{\beta}[b^k(X(\cdot),\alpha(\cdot))]_{t_n,s}
\\
&+\sum_{\beta\in\mathcal{B}_1^b}I_{\beta}[b^k(X(\cdot),\alpha(\cdot))-b^k(X(t_n),\alpha(t_n))]_{t_n,s}\Big)ds\notag
\\
&+\sum_{j=1}^m\int_{t_n}^{t_{n+1}}\Big(\sum_{\beta\in\mathcal{B}_1^\sigma\cup(\mathcal{A}_{\gamma-0.5}^\sigma\setminus\tilde{\mathcal{A}}_{\gamma-0.5}^\sigma)}I_{\beta}[\sigma^{(k,j)}(X(t_n),\alpha(t_n))]_{t_n,s}
\\
&+\sum_{\beta\in\tilde{\mathcal{A}}_{\gamma-0.5}^\sigma\cup\mathcal{B}_2^\sigma\cup\mathcal{B}_3^\sigma}I_{\beta}[\sigma^{(k,j)}(X(t_n),\alpha(\cdot))]_{t_n,s}
+\sum_{\beta\in\mathcal{B}_2^\sigma}\sum_{k_1=0}^mI_{(k_1)\star\beta}[\sigma^{(k,j)}(X(\cdot),\alpha(\cdot))]_{t_n,s}
\\
&+\sum_{\beta\in\mathcal{B}_3^\sigma}\sum_{k_1\in\{N_1,\ldots,N_{\mu},\bar{N}_{\mu},0,1,\ldots,m\}}I_{(k_1)\star\beta}[\sigma^{(k,j)}(X(\cdot),\alpha(\cdot))]_{t_n,s}+\sum_{\beta\in\mathcal{B}_4^\sigma}I_{\beta}[\sigma^{(k,j)}(X(\cdot),\alpha(\cdot))]_{t_n,s}
\\
&+\sum_{\beta\in\mathcal{B}_1^\sigma}I_{\beta}[\sigma^{(k,j)}(X(\cdot),\alpha(\cdot))-\sigma^{(k,j)}(X(t_n),\alpha(t_n))]_{t_n,s}\Big)dW^j(s)
\end{align*}
almost surely for all $n\in\{0,1,\ldots,n_T-1\}$ and $k\in\{1,\ldots,d\}$.
Now, to summarize the notations of the fourth, fifth, tenth and eleventh terms on the right side of the above equation, we denote by
\begin{itemize}
\item[] $\tilde{\mathcal{B}}_2^b:=\{(k_1)\star \beta:\beta\in\mathcal{B}_2^b,k_1\in\{0,1,\ldots,m\}\},$
\item[] $\tilde{\mathcal{B}}_2^\sigma:=\{(k_1)\star \beta:\beta\in\mathcal{B}_2^\sigma,k_1\in\{0,1,\ldots,m\}\},$
\item[] $\tilde{\mathcal{B}}_3^b:=\{(k_1)\star \beta:\beta\in\mathcal{B}_3^b,k_1\in\{N_1,\ldots,N_{\mu},\bar{N}_{\mu},0,1,\ldots,m\}\},$
\item[] $\tilde{\mathcal{B}}_3^\sigma:=\{(k_1)\star \beta:\beta\in\mathcal{B}_3^\sigma,k_1\in\{N_1,\ldots,N_{\mu},\bar{N}_{\mu},0,1,\ldots,m\}\}$.
\end{itemize}
Also notice that there are some elements of $\mathcal{A}_{\gamma}^b$ and $\mathcal{A}_{\gamma}^\sigma$ which are present in $\tilde{\mathcal{B}}_3^b$ and $\tilde{\mathcal{B}}_3^\sigma$, respectively and we identify them with the help of following notations,
\begin{itemize}
\item[] $\mathcal{B}_5^b:=\tilde{\mathcal{B}}_3^b\cap \mathcal{A}_{\gamma}^b,$
\item[] $\mathcal{B}_5^\sigma:=\tilde{\mathcal{B}}_3^\sigma\cap \mathcal{A}_{\gamma}^\sigma$.
\end{itemize}
Clearly, $\mathcal{B}_5^b$, $\mathcal{B}_5^\sigma\subset\mathcal{M}_3$.
Thus, by using the above notations, we write, 
\begin{align*}
X^k(t_{n+1})&=X^k(t_{n})+\int_{t_n}^{t_{n+1}}\Big(\sum_{\beta\in\mathcal{B}_1^b\cup(\mathcal{A}_{\gamma-0.5}^b\setminus\tilde{\mathcal{A}}_{\gamma-0.5}^b)}I_{\beta}[b^k(X(t_n),\alpha(t_n))]_{t_n,s}
\\
&+\sum_{\beta\in\tilde{\mathcal{A}}_{\gamma-0.5}^b\cup\mathcal{B}_2^b\cup\mathcal{B}_3^b}I_{\beta}[b^k(X(t_n),\alpha(\cdot))]_{t_n,s}
+\sum_{\beta\in\mathcal{B}_5^b}I_{\beta}[b^k(X(\cdot),\alpha(\cdot))]_{t_n,s}
\\
&+\sum_{\beta\in\mathcal{B}_4^b\cup\tilde{\mathcal{B}}_2^b\cup(\tilde{\mathcal{B}}_3^b\setminus\mathcal{B}_5^b)}I_{\beta}[b^k(X(\cdot),\alpha(\cdot))]_{t_n,s}
\\
&+\sum_{\beta\in\mathcal{B}_1^b}I_{\beta}[b^k(X(\cdot),\alpha(\cdot))-b^k(X(t_n),\alpha(t_n))]_{t_n,s}\Big)ds\notag
\\
&+\sum_{j=1}^m\int_{t_n}^{t_{n+1}}\Big(\sum_{\beta\in\mathcal{B}_1^\sigma\cup(\mathcal{A}_{\gamma-0.5}^\sigma\setminus\tilde{\mathcal{A}}_{\gamma-0.5}^\sigma)}I_{\beta}[\sigma^{(k,j)}(X(t_n),\alpha(t_n))]_{t_n,s}
\\
&+\sum_{\beta\in\tilde{\mathcal{A}}_{\gamma-0.5}^\sigma\cup\mathcal{B}_2^\sigma\cup\mathcal{B}_3^\sigma}I_{\beta}[\sigma^{(k,j)}(X(t_n),\alpha(\cdot))]_{t_n,s}
+\sum_{\beta\in\mathcal{B}_5^\sigma}I_{\beta}[\sigma^{(k,j)}(X(\cdot),\alpha(\cdot))]_{t_n,s}
\\
&+\sum_{\beta\in\mathcal{B}_4^\sigma\cup\tilde{\mathcal{B}}_2^\sigma\cup(\tilde{\mathcal{B}}_3^\sigma\setminus\mathcal{B}_5^\sigma)}I_{\beta}[\sigma^{(k,j)}(X(\cdot),\alpha(\cdot))]_{t_n,s}
\\
&+\sum_{\beta\in\mathcal{B}_1^\sigma}I_{\beta}[\sigma^{(k,j)}(X(\cdot),\alpha(\cdot))-\sigma^{(k,j)}(X(t_n),\alpha(t_n))]_{t_n,s}\Big)dW^j(s)
\end{align*}
almost surely for all $n\in\{0,1,\ldots,n_T-1\}$ and $k\in\{1,\ldots,d\}$.
Further, on using Lemma \ref{lem:Ito on beta} for $\beta\in\{\mathcal{B}_5^b,\mathcal{B}_5^\sigma\}$, we have
\begin{align*}
X^k(t_{n+1})&=X^k(t_{n})+\int_{t_n}^{t_{n+1}}\Big(\sum_{\beta\in(\mathcal{B}_1^b\cup(\mathcal{A}_{\gamma-0.5}^b\setminus\tilde{\mathcal{A}}_{\gamma-0.5}^b))}I_{\beta}[b^k(X(t_n),\alpha(t_n))]_{t_n,s}\notag
\\
&+\sum_{\beta\in\tilde{\mathcal{A}}_{\gamma-0.5}^b\cup\mathcal{B}_2^b\cup\mathcal{B}_3^b\cup\mathcal{B}_5^b}I_{\beta}[b^k(X(t_n),\alpha(\cdot))]_{t_n,s}
+\sum_{\beta\in\mathcal{B}_4^b\cup\tilde{\mathcal{B}}_2^b\cup(\tilde{\mathcal{B}}_3^b\setminus\mathcal{B}_5^b)\cup\tilde{\mathcal{B}}_5^b}I_{\beta}[b^k(X(\cdot),\alpha(\cdot))]_{t_n,s}\notag
\\
&+\sum_{\beta\in\mathcal{B}_1^b}I_{\beta}[b^k(X(\cdot),\alpha(\cdot))-b^k(X(t_n),\alpha(t_n))]_{t_n,s}\Big)ds\notag
\\
&+\sum_{j=1}^m\int_{t_n}^{t_{n+1}}\Big(\sum_{\beta\in(\mathcal{B}_1^\sigma\cup(\mathcal{A}_{\gamma-0.5}^\sigma\setminus\tilde{\mathcal{A}}_{\gamma-0.5}^\sigma))}I_{\beta}[\sigma^{(k,j)}(X(t_n),\alpha(t_n))]_{t_n,s}\notag
\\
&+\sum_{\beta\in\tilde{\mathcal{A}}_{\gamma-0.5}^\sigma\cup\mathcal{B}_2^\sigma\cup\mathcal{B}_3^\sigma\cup\mathcal{B}_5^\sigma}I_{\beta}[\sigma^{(k,j)}(X(t_n),\alpha(\cdot))]_{t_n,s}\notag
\\
&+\sum_{\beta\in\mathcal{B}_4^\sigma\cup\tilde{\mathcal{B}}_2^\sigma\cup(\tilde{\mathcal{B}}_3^\sigma\setminus\mathcal{B}_5^\sigma)\cup\tilde{\mathcal{B}}_5^\sigma}I_{\beta}[\sigma^{(k,j)}(X(\cdot),\alpha(\cdot))]_{t_n,s}\notag
\\
&+\sum_{\beta\in\mathcal{B}_1^\sigma}I_{\beta}[\sigma^{(k,j)}(X(\cdot),\alpha(\cdot))-\sigma^{(k,j)}(X(t_n),\alpha(t_n))]_{t_n,s}\Big)dW^j(s) 
\end{align*}
almost surely for all $n\in\{0,1,\ldots,n_T-1\}$ and $k\in\{1,\ldots,d\}$ where 
\begin{itemize}
\item[] $\tilde{\mathcal{B}}_5^b:=\{(k_1)\star \beta:\beta\in\mathcal{B}_5^b,k_1\in\{0,1,\ldots,m\}\},$
\item[] $\tilde{\mathcal{B}}_5^\sigma:=\{(k_1)\star \beta:\beta\in\mathcal{B}_5^\sigma,k_1\in\{0,1,\ldots,m\}\}$.
\end{itemize}

Clearly, $\mathcal{A}_{\gamma-0.5}^\sigma\cup\mathcal{B}_1^\sigma\cup\mathcal{B}_2^\sigma\cup\mathcal{B}_3^\sigma\cup\mathcal{B}_5^\sigma\subset\mathcal{A}_\gamma^\sigma$. 
Notice that $\nu\in\mathcal{A}_{\gamma-0.5}^\sigma\cup\mathcal{B}_1^\sigma\cup\mathcal{B}_2^\sigma\cup\mathcal{B}_3^\sigma\cup\mathcal{B}_5^\sigma$.
Now, we prove $\mathcal{A}_\gamma^\sigma\subset\mathcal{A}_{\gamma-0.5}^\sigma\cup\mathcal{B}_1^\sigma\cup\mathcal{B}_2^\sigma\cup\mathcal{B}_3^\sigma\cup\mathcal{B}_5^\sigma$. 
For this, let $\beta\in\mathcal{A}_\gamma^\sigma\setminus\{\nu\}$ which implies $\eta(\beta)\leq 2\gamma-1$.
If $\eta(\beta)\leq 2\gamma-2$, then $\beta\in\mathcal{A}_{\gamma-0.5}^\sigma$.
Further, consider $\eta(\beta)=2\gamma-1$ which gives rise to two cases, either $\eta(-\beta)<\eta(\beta)$ or $\eta(-\beta)=\eta(\beta)$. 
If $\eta(-\beta)<\eta(\beta)$,  then by the definition of hierarchical and remainder sets from Subsection \ref{sub:hierarchical and reminder}, we have  $-\beta\in\mathcal{A}_{\gamma-0.5}^\sigma$ and $\beta\in\mathcal{B}(\mathcal{A}_{\gamma-0.5}^\sigma)$ which gives $\beta\in\mathcal{B}_1^\sigma\cup\mathcal{B}_2^\sigma\cup\mathcal{B}_3^\sigma$.
Moreover, if $\eta(-\beta)=\eta(\beta)$, then  $-\beta\in\mathcal{B}_3^\sigma$ and $\beta\in\mathcal{B}_5^\sigma$.
Hence 
\begin{itemize}
\item[] $\mathcal{A}_\gamma^\sigma=\mathcal{A}_{\gamma-0.5}^\sigma\cup\mathcal{B}_1^\sigma\cup\mathcal{B}_2^\sigma\cup\mathcal{B}_3^\sigma\cup\mathcal{B}_5^\sigma$ and 
\item[] $\tilde{\mathcal{A}}_{\gamma}^\sigma=\tilde{\mathcal{A}}_{\gamma-0.5}^\sigma\cup\mathcal{B}_2^\sigma\cup\mathcal{B}_3^\sigma\cup\mathcal{B}_5^\sigma$.
\end{itemize}
Moreover, 
\begin{itemize}
\item[] $\mathcal{B}_4^\sigma\cup\tilde{\mathcal{B}}_1^\sigma\cup\tilde{\mathcal{B}}_2^\sigma\cup(\tilde{\mathcal{B}}_3^\sigma\setminus\mathcal{B}_5^\sigma)\cup\tilde{\mathcal{B}}_5^\sigma=\{\beta\in\mathcal{M}\setminus\mathcal{A}_{\gamma}^\sigma :-\beta\in\mathcal{A}_{\gamma-0.5}^\sigma\}\cup\{\beta\in\mathcal{M}\setminus\mathcal{A}_{\gamma}^\sigma:-\beta\in\mathcal{B}_{1}^\sigma\}\cup\{\beta\in\mathcal{M}\setminus\mathcal{A}_{\gamma}^\sigma:-\beta\in\mathcal{B}_{2}^\sigma\}\cup\{\beta\in\mathcal{M}\setminus\mathcal{A}_{\gamma}^\sigma:-\beta\in\mathcal{B}_{3}^\sigma\}\cup\{\beta\in\mathcal{M}\setminus\mathcal{A}_{\gamma}^\sigma:-\beta\in\mathcal{B}_5^\sigma\}=\mathcal{B}(\mathcal{A}_{\gamma}^\sigma)$.
\end{itemize} 
Similarly, we obtain
\begin{itemize}
\item[] $\mathcal{A}_\gamma^b=\mathcal{A}_{\gamma-0.5}^b\cup\mathcal{B}_1^b\cup\mathcal{B}_2^b\cup\mathcal{B}_3^b\cup\mathcal{B}_5^b$, 
\item[] $\tilde{\mathcal{A}}_{\gamma}^b=\tilde{\mathcal{A}}_{\gamma-0.5}^b\cup\mathcal{B}_2^b\cup\mathcal{B}_3^b\cup\mathcal{B}_5^b$,
\item[] $\mathcal{B}(\mathcal{A}_\gamma^b)\setminus \tilde{\mathcal{B}}_1^b=\mathcal{B}_4^b\cup\tilde{\mathcal{B}}_2^b\cup(\tilde{\mathcal{B}}_3^b\setminus\mathcal{B}_5^b)\cup\tilde{\mathcal{B}}_5^b$.
\end{itemize} 
Thus, preceding equation can be written as
\begin{align}\label{eq:gen. scheme derivation}
&X^k(t_{n+1})=X^k(t_{n})+\sum_{\beta\in\mathcal{A}_{\gamma}^b\setminus\tilde{\mathcal{A}}_{\gamma}^b}\int_{t_n}^{t_{n+1}}I_{\beta}[b^k(X(t_n),\alpha(t_n))]_{t_n,s}ds\notag
\\
&+\sum_{\beta\in\tilde{\mathcal{A}}_{\gamma}^b}\int_{t_n}^{t_{n+1}}I_{\beta}[b^k(X(t_n),\alpha(\cdot))]_{t_n,s}ds\notag+\sum_{\beta\in\mathcal{B}(\mathcal{A}_\gamma^b)\setminus \tilde{\mathcal{B}}_1^b}\int_{t_n}^{t_{n+1}}I_{\beta}[b^k(X(\cdot),\alpha(\cdot))]_{t_n,s}ds\notag
\\
&+\sum_{\beta\in\mathcal{B}_1^b}\int_{t_n}^{t_{n+1}}I_{\beta}[b^k(X(\cdot),\alpha(\cdot))-b^k(X(t_n),\alpha(t_n))]_{t_n,s}ds\notag
\\
&+\sum_{\beta\in\mathcal{A}_{\gamma}^\sigma\setminus\tilde{\mathcal{A}}_{\gamma}^\sigma}\sum_{j=1}^m\int_{t_n}^{t_{n+1}}I_{\beta}[\sigma^{(k,j)}(X(t_n),\alpha(t_n))]_{t_n,s}dW^j(s)\notag
\\
&+\sum_{\beta\in\tilde{\mathcal{A}}_{\gamma}^\sigma}\sum_{j=1}^m\int_{t_n}^{t_{n+1}}I_{\beta}[\sigma^{(k,j)}(X(t_n),\alpha(\cdot))]_{t_n,s}dW^j(s)\notag
\\
&+\sum_{\beta\in\mathcal{B}(\mathcal{A}_\gamma^\sigma)\setminus \tilde{\mathcal{B}}_1^\sigma}\sum_{j=1}^m\int_{t_n}^{t_{n+1}}I_{\beta}[\sigma^{(k,j)}(X(\cdot),\alpha(\cdot))]_{t_n,s}dW^j(s)\notag
\\
&+\sum_{\beta\in\mathcal{B}_1^\sigma}\sum_{j=1}^m\int_{t_n}^{t_{n+1}}I_{\beta}[\sigma^{(k,j)}(X(\cdot),\alpha(\cdot))-\sigma^{(k,j)}(X(t_n),\alpha(t_n))]_{t_n,s}dW^j(s) 
\end{align}
almost surely for all $n\in\{0,1,\ldots,n_T-1\}$ and $k\in\{1,\ldots,d\}$.
Hence, the explicit numerical scheme \eqref{eq:gen.scheme} of $\gamma\in \{n/2:n\in\mathbb{N}\}$-order  for SDEwMS \eqref{eq:sdems} is obtained by ignoring the fourth, fifth, eighth and ninth terms on the right side of \eqref{eq:gen. scheme derivation}.
\subsection{Moment bound}
In this subsection, we shall show the moment stability of the numerical scheme \eqref{eq:gen.scheme} of $\gamma\in \{n/2:n\in\mathbb{N}\}$-order for SDEwMS \eqref{eq:sdems}. 
The result on the moment bound of the scheme \eqref{eq:gen.scheme} is stated in Lemma \ref{lem:schemeMoment}. 
We first establish some  preliminary lemmas which are required for proving Lemma~\ref{lem:schemeMoment}.

For the following lemmas, let us recall the definitions of $\eta(\beta)$ and $l(\beta)$ from Subsection \ref{sub:Multi-indices} and $J^{\beta}_{i_0}$ from Subsection \ref{sub:Operators}. 
It is assumed that all integrals appearing in Lemmas \ref{lem:multiple estimate 0,1} and \ref{lem:multiple estimate N_mu} are well-defined.  Also, one requires the standard stopping time arguments in the proof which we avoid to  write for the sake of notational simplicity. 
\begin{lemma}
\label{lem:multiple estimate 0,1}
Let $f:\mathbb{R}^d\times\mathcal{S}\mapsto\mathbb{R}$ and $\beta\in\{(j_1,\ldots,j_l)\in\mathcal{M}:j_i\in\{0,1,\ldots,m\}\forall i\in\{1,\ldots,l\}\}$.
Then,
\begin{align}
E\Big(&\sup_{u\in[s,t]}| I_{\beta}[f(X(\cdot),\alpha(\cdot))]_{s,u}|^2\Big|\mathcal{F}_T^{\alpha} \Big) \notag
\\
& \leq C (t-s)^{\eta(\beta)-l(\beta)} \int_{s}^{t}\int_{s}^{t_{1}}\cdots\int_{s}^{t_{l(\beta)-1}}E\Big(|J^{\beta}_{\alpha(t_{l(\beta)})} f(X(t_{l(\beta)}),\alpha(t_{l(\beta)}))|^2\Big|\mathcal{F}_T^{\alpha} \Big)  dt_{l(\beta)}\cdots dt_2dt_1\notag
\end{align}
almost surely for all $s<t\in[0,T]$.

In addition, if  $|J^{\beta}_{i_0}f(x,i_0)|\leq L (1+|x|)$ for all $x\in\mathbb{R}^d$ and $i_0\in\mathcal{S}$ where $L>0$ is a constant, then
\begin{align*}
E\Big(\sup_{u\in[s,t]}|I_{\beta}[f(X(\cdot),\alpha(\cdot))]_{s,u}|^2&\Big|\mathcal{F}_T^{\alpha} \Big)\leq C (t-s)^{\eta(\beta)} E\big(\sup_{u\in[s,t]}(1+|X(u)|)^2\big|\mathcal{F}_T^{\alpha} \big).
\end{align*}

Moreover, for $\bar{\beta}\in\mathcal{M}$, 
\begin{align*}
E\Big(& \sup_{u\in[s,t]}|I_{\bar{\beta}*\beta}[f(X(\cdot),\alpha(\cdot))]_{s,u}|^2\Big|\mathcal{F}_T^{\alpha} \Big)  \notag
\\
&\leq C (t-s)^{\eta(\beta)-l(\beta)} \int_{s}^{t}\int_{s}^{t_{1}}\cdots\int_{s}^{t_{l(\beta)-1}} E\Big( |I_{\bar{\beta}}[J^\beta_{\alpha(\cdot)} f(X(\cdot),\alpha(\cdot))]_{s,t_{l(\beta
)}}|^2\Big|\mathcal{F}_T^{\alpha} \Big)  dt_{l(\beta)}\cdots dt_2dt_1\notag
\end{align*}
almost surely for all $s<t\in[0,T]$. 
\end{lemma}
\begin{proof}
The first inequality of the lemma is proved by using induction on $l(\beta)$.
Let $l(\beta)=1$ and $\beta=(0)$ which gives $\eta(\beta)=2$. 
On the application of H\"older's inequality, we have
\begin{align*}
E\Big(\sup_{u\in[s,t]}|I_{\beta}[f(X(\cdot),\alpha(\cdot))]_{s,u}|^2&\Big|\mathcal{F}_T^{\alpha} \Big)=E\Big(\sup_{u\in[s,t]}\Big|\int^u_{s}L^0_{\alpha(t_1)}f(X(t_1),\alpha(t_1))dt_1 \Big|^2\Big|\mathcal{F}_T^{\alpha} \Big)
\\
\leq & C (t-s)\int_{s}^{t}E\Big(|L^0_{\alpha(t_1)}f(X(t_1),\alpha(t_1))|^2\Big|\mathcal{F}_T^{\alpha} \Big)dt_1
\\
=&C (t-s)^{\eta(\beta)-l(\beta)}\int_{s}^{t}E\Big(|J^{\beta}_{\alpha(t_1)}f(X(t_1),\alpha(t_1))|^2\Big|\mathcal{F}_T^{\alpha} \Big)dt_1
\end{align*}
almost surely for all $s<t\in[0,T]$ where $J^{\beta}_{i_0}=L_{i_0}^0$, $i_0\in \mathcal{S}$.
Further, if $l(\beta)=1$ and $\beta=(j)$ for all $j\in\{1,2,\ldots,m\}$ which yields $\eta(\beta)=1$. 
Then,  Burkholder-Davis-Gundy inequality gives,
\begin{align*}
E\Big(\sup_{u\in[s,t]}|I_{\beta}[f(X(\cdot),\alpha(\cdot))]_{s,u}|^2& \Big|\mathcal{F}_T^{\alpha} \Big)=E\Big(\sup_{u\in[s,t]}\Big|\int^u_{s}L^j_{\alpha(t_1)}f(X(t_1),\alpha(t_1))dW^j(t_1) \Big|^2\Big|\mathcal{F}_T^{\alpha} \Big)
\\
\leq & C \int_{s}^{t}E\Big(|L^j_{\alpha(t_1)}f(X(t_1),\alpha(t_1))|^2\Big|\mathcal{F}_T^{\alpha} \Big)dt_1
\\
=&C (t-s)^{\eta(\beta)-l(\beta)}\int_{s}^{t}E\Big(|J^{\beta}_{\alpha(t_1)}f(X(t_1),\alpha(t_1))|^2\Big|\mathcal{F}_T^{\alpha} \Big)dt_1
\end{align*}
almost surely for all $s<t\in[0,T]$ where $J^{\beta}_{i_0}=L_{i_0}^j$, $i_0\in \mathcal{S}$ and $j\in \{1,\ldots,m\}$.
Thus, the first inequality of the lemma holds for $l(\beta)=1$. 
For inductive arguments, we assume that it holds for $l(\beta)=k$ for a fixed $k\in \mathbb{N}$. 
Now, consider $\beta=(j_1,\ldots,j_{k+1})$ and  $j_{k+1}=0$. 
On using H\"older's inequality and the inductive hypothesis, we have, 
\begin{align*}
E\Big(&\sup_{u\in[s,t]}|I_{\beta}[f(X(\cdot),\alpha(\cdot))]_{s,u}|^2\Big|\mathcal{F}_T^{\alpha} \Big)=E\Big(\sup_{u\in[s,t]}|\int^{u}_{s}I_{\beta-}[L^{j_{k+1}}_{\alpha(\cdot)}f(X(\cdot),\alpha(\cdot))]_{s,t_1}dt_1|^2\Big|\mathcal{F}_T^{\alpha} \Big)
\\
\leq & C(t-s) \int^{t}_{s}E\Big(\sup_{u\in[s,t_1]}|I_{\beta-}[L^{j_{k+1}}_{\alpha(\cdot)}f(X(\cdot),\alpha(\cdot))]_{s,u}|^2\Big|\mathcal{F}_T^{\alpha} \Big)dt_1
\\
\leq& C (t-s)^{\eta(\beta)-l(\beta)} \int_{s}^{t}\int_{s}^{t_{1}}\cdots\int_{s}^{t_{l(\beta)-1}} E\Big(|J^{\beta}_{\alpha(t_{l(\beta)})} f(X(t_{l(\beta)}),\alpha(t_{l(\beta)}))|^2\Big|\mathcal{F}_T^{\alpha} \Big)  dt_{l(\beta)}\cdots dt_2dt_1
\end{align*}
almost surely for all $s<t\in[0,T]$.
Moreover, if $j_{k+1}\in\{1,2,\ldots,m\}$, then by using Burkholder-Davis-Gundy inequality and inductive hypothesis, we obtain
\begin{align*}
E\Big(&\sup_{u\in[s,t]}|I_{\beta}[f(X(\cdot),\alpha(\cdot))]_{s,u}|^2\Big|\mathcal{F}_T^{\alpha} \Big)=E\Big(\sup_{u\in[s,t]}|\int^{u}_{s}I_{\beta-}[L^{j_{k+1}}_{\alpha(\cdot)}f(X(\cdot),\alpha(\cdot))]_{s,t_1}dW^j(t_1)|^2\Big|\mathcal{F}_T^{\alpha} \Big)
\\
\leq & C \int^{t}_{s}E\Big(\sup_{u\in[s,t_1]}|I_{\beta-}[L^{j_{k+1}}_{\alpha(\cdot)}f(X(\cdot),\alpha(\cdot))]_{s,u}|^2\Big|\mathcal{F}_T^{\alpha} \Big)dt_1
\\
\leq& C (t-s)^{\eta(\beta)-l(\beta)} \int_{s}^{t}\int_{s}^{t_{1}}\cdots\int_{s}^{t_{l(\beta)-1}} E\Big(|J^{\beta}_{\alpha(t_{l(\beta)})} f(X(t_{l(\beta)}),\alpha(t_{l(\beta)}))|^2\Big|\mathcal{F}_T^{\alpha} \Big)  dt_{l(\beta)}\cdots dt_2dt_1
\end{align*} 
almost surely for all $s<t\in[0,T]$.
This completes the proof of first inequality of the lemma.
The second inequality of the lemma follows immediately from the first inequality on using  $|J^{\beta}_{i_0}f(x,i_0)|\leq L (1+|x|)$ for all $x\in\mathbb{R}^d$ and $i_0\in\mathcal{S}$. 
By adapting the similar arguments, the third inequality of the lemma also follows. 
\end{proof}
For following lemma,  let us define $l_{k}:=\displaystyle\sum_{i=1}^{k}l(\beta_i)$ for $\beta_i\in\mathcal{M}$,  $i\in\{1,\ldots,k\}$ and $k\in\mathbb{N}$ and  recall the definitions of operators  $J^{\beta}_{i_0}$, $J^{\beta}_{i_0k_0}$ from Subsection \ref{sub:Operators}. 

\begin{lemma}
\label{lem:multiple estimate N_mu}
Let $f:\mathbb{R}^d\times\mathcal{S}\mapsto\mathbb{R}$. Then,  for all $ \beta\in \mathcal{M}$ such that $\beta=\beta_{r+1}\star(N_{\mu_{r}})\star \beta_{r}\star \cdots\star(N_{\mu_2})\star \beta_2\star(N_{\mu_1})\star \beta_1$ where $ \mu_1, \ldots,\mu_r\in\{1,\ldots,\mu\}$, $\beta_1,\ldots,\beta_{r+1} \in \mathcal{M}_1$ and   $r=[n](\beta)\in \mathbb{N}$,
\begin{align*}
 E\Big(|I_{\beta}[f(X(\cdot),&\alpha(\cdot))]_{s,t_0}|^2\Big|\mathcal{F}_T^{\alpha} \Big)\leq C\prod_{i=1}^r\mu_i (t_0-s)^{\eta(\beta_1)+\cdots +\eta(\beta_{r+1})-l_{r+1}}
\\
&\int_s^{t_0}\int_{s}^{t_{1}}\cdots\int_{s}^{t_{l_1-1}}\sum_{i_1\neq k_1}\int_{s}^{t_{l_1}}\mathbbm{1}\{N^{(s,t_{l_1}]}=\mu_1\}
\\
&\int_s^{t_{l_1+1}}\int_{s}^{t_{l_1+2}}\cdots\int_{s}^{t_{l_2}}\sum_{i_2\neq k_2}\int_{s}^{t_{l_2+1}}\mathbbm{1}\{N^{(s,t_{l_2+1}]}=\mu_2\} \cdots
\\
&\int_s^{t_{l_{r-1}+r-1}}\int_{s}^{t_{l_{r-1}+r}}\cdots\int_{s}^{t_{l_r+r-2}}\sum_{i_r\neq k_r}\int_{s}^{t_{l_r+r-1}}\mathbbm{1}\{N^{(s,t_{l_r+r-1}]}=\mu_r\}
\\
&\int_s^{t_{l_{r}+r}}\int_{s}^{t_{l_{r}+r+1}}\cdots\int_{s}^{t_{l_{r+1}+r-1}}
\\
&E\Big(|J^{\beta_{r+1}}_{\alpha(t_{l_{r+1}+r})}J^{\beta_r}_{i_rk_r}\cdots J^{\beta_2}_{i_2k_2}(J^{\beta_1}_{k_1}f(X(t_{l_{r+1}+r}),k_1)-J^{\beta_1}_{i_1}f(X(t_{l_{r+1}+r}),i_1))|^2\Big|\mathcal{F}_T^{\alpha} \Big)
\\
&dt_{l_{r+1}+r}\cdots dt_{l_{r}+r+2}dt_{l_{r}+r+1}d[M_{i_rk_r}](t_{l_r+r})dt_{l_r+r-1}\cdots dt_{l_{r-1}+r+1}dt_{l_{r-1}+r} \cdots
\\
&d[M_{i_2k_2}](t_{l_2+2})dt_{l_2+1}\cdots dt_{l_1+3}dt_{l_1+2}d[M_{i_1k_1}](t_{l_1+1})dt_{l_1}\cdots dt_2dt_1
\end{align*}
almost surely for all $s<t_0\in[0,T]$ provided $\mu_1\geq \mu_2 \geq \cdots\geq \mu_r$, otherwise, $I_{\beta}[f(X(\cdot),\alpha(\cdot))]_{s,t_0}=~0$.
\end{lemma}
\begin{proof}
We shall prove the lemma by induction on $r=[n](\beta)$.
Let $[n](\beta)=1$.
Then, $\beta=\beta_2\star(N_{\mu_1})\star\beta_1$. 
If $\beta_1\neq\nu$, on using  Lemma \ref{lem:multiple estimate 0,1}, we obtain
\begin{align*}
E\Big(& |I_{\beta}[f(X(\cdot),\alpha(\cdot))]_{s,t_0}|^2\Big|\mathcal{F}_T^{\alpha} \Big)=E\Big(|I_{\beta_2\star(N_{\mu_1})\star\beta_1}[f(X(\cdot),\alpha(\cdot))]_{s,t_0}|^2\Big|\mathcal{F}_T^{\alpha} \Big)
\\
&\leq C(t_0-s)^{\eta(\beta_1)-l_1} \int_s^{t_0}\int_{s}^{t_{1}}\cdots\int_{s}^{t_{l_1-1}}E\Big(|I_{\beta_2\star(N_{\mu_1})}[J^{\beta_1}_{\alpha(\cdot)}f(X(\cdot),\alpha(\cdot))]_{s,t_{l_1}}|^2\Big|\mathcal{F}_T^{\alpha} \Big)  dt_{l_1}\cdots dt_2dt_1
\end{align*}
almost surely for any $s<t_0\in [0,T]$. 
Notice that if $\beta_1=\nu$, then $J^{\beta_1}_{i_0}$ is the identity operator for all $i_0 \in \mathcal{S}$,  $\eta(\beta_1)=l_1=0$ and the iterated integrals$ \int_s^{t_0}\int_{s}^{t_{1}}\cdots\int_{s}^{t_{l_1-1}}$ on the right side of the above expression and in the  subsequent calculations disappear. 
Also, the above inequality becomes an equality. 
Further,  the definition of multiple integrals from Subsection \ref{sub:multiple integral} yields
\begin{align}
E\Big(&|I_{\beta}[f(X(\cdot),\alpha(\cdot))]_{s,t_0}|^2\Big|\mathcal{F}_T^{\alpha} \Big)\leq C(t_0-s)^{\eta(\beta_1)-l_1} \int_s^{t_0}\int_{s}^{t_{1}}\cdots\int_{s}^{t_{l_1-1}}\notag
\\
& E\Big(\big|\sum_{i_1\neq k_1}\int_{s}^{t_{l_1}}\mathbbm{1}\{N^{(s,t_{l_1}]}=\mu_1\}I_{\beta_2}[J^{\beta_1}_{k_1}f(X(\cdot),k_1)-J^{\beta_1}_{i_1}f(X(\cdot),i_1)]_{s,t_{l_1+1}}d[M_{i_1k_1}](t_{l_1+1})\big|^2\Big|\mathcal{F}_T^{\alpha} \Big)  \notag
\\
&dt_{l_1}\cdots dt_2dt_1\notag
\\
\leq&  C\mu_1(t_0-s)^{\eta(\beta_1)-l_1} \int_s^{t_0}\int_{s}^{t_{1}}\cdots\int_{s}^{t_{l_1-1}}\sum_{i_1\neq k_1}\int_{s}^{t_{l_1}}\mathbbm{1}\{N^{(s,t_{l_1}]}=\mu_1\}\notag
\\
& E\Big(\big|I_{\beta_2}[J^{\beta_1}_{k_1}f(X(\cdot),k_1)-J^{\beta_1}_{i_1}f(X(\cdot),i_1)]_{s,t_{l_1+1}}\big|^2\Big|\mathcal{F}_T^{\alpha} \Big)d[M_{i_1k_1}](t_{l_1+1}) dt_{l_1}\cdots dt_2dt_1\notag
\end{align}
almost surely for all $s<t_0\in[0,T]$ where the last inequality is obtained by using, 
\begin{align}
\Big|\sum_{i_1\neq k_1}\int_{s}^{t_{l_1}}I_{\beta_2}&[J^{\beta_1}_{k_1}f(X(\cdot),k_1)-J^{\beta_1}_{i_1}f(X(\cdot),i_1)]_{s,t_{l_1+1}}d[M_{i_1k_1}](t_{l_1+1})\Big|^2 \notag
\\
=&\Big|\sum_{i=1}^{\displaystyle N^{(s,t_{l_1}]}}I_{\beta_2}[J^{\beta_1}_{\alpha(\tau_i)}f(X(\cdot),\alpha(\tau_i))-J^{\beta_1}_{\alpha(\tau_{i-1})}f(X(\cdot),\alpha(\tau_{i-1}))]_{s,\tau_i} \Big|^2 \notag
\\
\leq & N^{(s,t_{l_1}]} \sum_{i=1}^{\displaystyle N^{(s,t_{l_1}]}}\Big|I_{\beta_2}[J^{\beta_1}_{\alpha(\tau_i)}f(X(\cdot),\alpha(\tau_i))-J^{\beta_1}_{\alpha(\tau_{i-1})}f(X(\cdot),\alpha(\tau_{i-1}))]_{s,\tau_i} \Big|^2, \label{eq:[M]estimate}
\end{align}
$\tau_1\ldots, \tau_{\displaystyle N^{(s,t_{l_1}]}}$ are jump times of the Markov chain $\alpha$ in the interval $(s,t_{l_1}]$.
Moreover, if $\beta_2\neq\nu$, then the application of Lemma \ref{lem:multiple estimate 0,1} gives 
\begin{align*}
E\Big(&|I_{\beta}[f(X(\cdot),\alpha(\cdot))]_{s,t_0}|^2\Big|\mathcal{F}_T^{\alpha} \Big)
\\
\leq&  C\mu_1(t_0-s)^{\eta(\beta_1)+\eta(\beta_2)-l_2} \int_s^{t_0}\int_{s}^{t_{1}}\cdots\int_{s}^{t_{l_1-1}}\sum_{i_1\neq k_1}\int_{s}^{t_{l_1}}\mathbbm{1}\{N^{(s,t_{l_1}]}=\mu_1\}\int_s^{t_{l_1+1}}\int_{s}^{t_{l_1+2}}\cdots\int_{s}^{t_{l_2}}\notag
\\
& E\Big(\big|J^{\beta_2}_{\alpha(t_{l_2+1})}(J_{k_1}^{\beta_1}f(X(t_{l_2+1}),k_1)-J_{i_1}^{\beta_1}f(X(t_{l_2+1}),i_1))\big|^2\Big|\mathcal{F}_T^{\alpha} \Big) 
\\
& dt_{l_2+1}\cdots dt_{l_1+3}dt_{l_1+2}d[M_{i_1k_1}](t_{l_1+1}) dt_{l_1}\cdots dt_2dt_1
\end{align*}
almost surely for all $s<t_0\in[0,T]$. 
As before, if  $\beta_2=\nu$, then  $J^{\beta_2}_{i_0}$ is the identity operator for $i_0 \in \mathcal{S}$, $\eta(\beta_2)=l(\beta_2)=0$,  $l_2=l_1$ and the iterated integrals $\int_s^{t_{l_1+1}}\int_{s}^{t_{l_1+2}}\cdots\int_{s}^{t_{l_2}}$ on the right side of the above inequality does not appear. 
Hence, lemma holds for $[n](\beta)=1$.

We make the inductive hypothesis that the lemma holds for $r-1=[n](\beta)-1\in \mathbb{N}$.

Now,  let us consider $r=[n](\beta)\in \{2,3,\ldots\}$ and recall $\beta=\beta_{r+1}\star(N_{\mu_{r}})\star \beta_{r}\star \cdots\star(N_{\mu_2})\star \beta_2\star(N_{\mu_1})\star \beta_1=\bar{\beta}\star(N_{\mu_1})\star \beta_1$ where $\bar{\beta}=\beta_{r+1}\star(N_{\mu_{r}})\star \beta_{r}\star \cdots\star(N_{\mu_2})\star \beta_2$.
Clearly, $[n](\bar{\beta)}=r-1$. 

For $\beta_{1}\neq\nu$,  Lemma \ref{lem:multiple estimate 0,1} yields, 
\begin{align*}
E\Big(&|I_{\beta}[f(X(\cdot),\alpha(\cdot))]_{s,t_0}|^2\Big|\mathcal{F}_T^{\alpha} \Big)=E\Big(|I_{\bar{\beta} \star(N_{\mu_1})\star \beta_1}[f(X(\cdot),\alpha(\cdot))]_{s,t_0}|^2\Big|\mathcal{F}_T^{\alpha} \Big)
\\
\leq&C(t_0-s)^{\eta(\beta_1)-l_1} \int_s^{t_0}\int_{s}^{t_{1}}\cdots\int_{s}^{t_{l_1-1}}E\Big(|I_{\bar{\beta}\star(N_{\mu_1})}[J^{\beta_1}_{\alpha(\cdot)}f(X(\cdot),\alpha(\cdot))]_{s,t_{l_1}}|^2\Big|\mathcal{F}_T^{\alpha} \Big)  dt_{l_1}\cdots dt_2dt_1
\end{align*}
almost surely for any $s<t_0\in[0,T]$. 
As before, if $\beta_1=\nu$, then $J^{\beta_1}_{i_0}$ is the identity operator for $i_0 \in \mathcal{S}$,  $\eta(\beta_1)=l_1=0$ and the iterated integrals $ \int_s^{t_0}\int_{s}^{t_{1}}\cdots\int_{s}^{t_{l_1-1}}$ does not appear on the right side of the above expression and in the corresponding  estimates that follow.
Also, the above inequality becomes  an equality. 
Further, the definition of multiple integrals from Subsection \ref{sub:multiple integral} yields, 
\begin{align}
E\Big(&|I_{\beta}[f(X(\cdot),\alpha(\cdot))]_{s,t_0}|^2\Big|\mathcal{F}_T^{\alpha} \Big)\leq C(t_0-s)^{\eta(\beta_1)-l_1} \int_s^{t_0}\int_{s}^{t_{1}}\cdots\int_{s}^{t_{l_1-1}} E\Big(\big|\sum_{i_1\neq k_1}\int_{s}^{t_{l_1}}\notag
\\
&\mathbbm{1}\{N^{(s,t_{l_1}]}=\mu_1\} I_{\bar{\beta}}[J^{\beta_1}_{k_1}f(X(\cdot),k_1)-J^{\beta_1}_{i_1}f(X(\cdot),i_1)]_{s,t_{l_1+1}}d[M_{i_1k_1}](t_{l_1+1})\big|^2\Big|\mathcal{F}_T^{\alpha} \Big)dt_{l_1}\cdots dt_2dt_1 \notag
\end{align}
which on using an estimate similar to the one obtained in  \eqref{eq:[M]estimate} gives

\begin{align}
E\Big(&|I_{\beta}[f(X(\cdot),\alpha(\cdot))]_{s,t_0}|^2\Big|\mathcal{F}_T^{\alpha} \Big)\leq C\mu_1(t_0-s)^{\eta(\beta_1)-l_1} \int_s^{t_0}\int_{s}^{t_{1}}\cdots\int_{s}^{t_{l_1-1}} \sum_{i_1\neq k_1}\int_{s}^{t_{l_1}} \mathbbm{1}\{N^{(s,t_{l_1}]}=\mu_1\}\notag
\\
& E\Big(\big|I_{\bar{\beta}}[J^{\beta_1}_{k_1}f(X(\cdot),k_1)-J^{\beta_1}_{i_1}f(X(\cdot),i_1)]_{s,t_{l_1+1}}\big|^2\Big|\mathcal{F}_T^{\alpha} \Big)d[M_{i_1k_1}](t_{l_1+1})dt_{l_1}\cdots dt_2dt_1 \notag
\end{align}
almost surely for any $s<t_0\in [0,T]$. 
The proof is completed by using the following inductive hypothesis, 
 \begin{align*}
E&\Big(\big|I_{\bar{\beta}}[J^{\beta_1}_{k_1}f(X(\cdot),k_1)-J^{\beta_1}_{i_1}f(X(\cdot),i_1)]_{s,t_{l_1+1}}\big|^2\Big|\mathcal{F}_T^{\alpha} \Big)\leq C\prod_{i=2}^r\mu_i (t_0-s)^{\eta(\beta_2)+\cdots +\eta(\beta_{r+1})-(l_{r+1}-l_1)}
 \\
&  \int_s^{t_{l_1+1}}\int_{s}^{t_{l_1+2}}\cdots\int_{s}^{t_{l_2}}\sum_{i_2\neq k_2}\int_{s}^{t_{l_2+1}}\mathbbm{1}\{N^{(s,t_{l_2+1}]}=\mu_2\} 
\\
&\cdots \int_s^{t_{l_{r-1}+r-1}}\int_{s}^{t_{l_{r-1}+r}}\cdots\int_{s}^{t_{l_r+r-2}}\sum_{i_r\neq k_r}\int_{s}^{t_{l_r+r-1}}\mathbbm{1}\{N^{(s,t_{l_r+r-1}]}=\mu_r\}
\\
&\int_s^{t_{l_{r}+r}}\int_{s}^{t_{l_{r}+r+1}}\cdots\int_{s}^{t_{l_{r+1}+r-1}}E\Big(|J^{\beta_{r+1}}_{\alpha(t_{l_{r+1}+r})}J^{\beta_r}_{i_rk_r}\cdots J^{\beta_3}_{i_3k_3}(J^{\beta_2}_{k_2}(J^{\beta_1}_{k_1}f(X(\cdot),k_1)-J^{\beta_1}_{i_1}f(X(\cdot),i_1))
\\
&-J^{\beta_2}_{i_2}(J^{\beta_1}_{k_1}f(X(\cdot),k_1)-J^{\beta_1}_{i_1}f(X(\cdot),i_1))|^2\Big|\mathcal{F}_T^{\alpha} \Big)dt_{l_{r+1}+r}\cdots dt_{l_{r}+r+2}dt_{l_{r}+r+1}
\\
&d[M_{i_rk_r}](t_{l_r+r})dt_{l_r+r-1}\cdots dt_{l_{r-1}+r+1}dt_{l_{r-1}+r} \cdots d[M_{i_1k_1}](t_{l_2+2})dt_{l_2+1}\cdots dt_{l_1+3}dt_{l_1+2}
\end{align*}
  almost surely for all $s<t_0\in[0,T]$.   
\end{proof}
The first inequality in the following corollary follows from Lemma \ref{lem:multiple estimate 0,1} and Remark \ref{rem:A b sigma ds dw growth} and the second inequality  from Lemma \ref{lem:multiple estimate N_mu} and Remark \ref{rem:A b N growth}.
\begin{corollary}
\label{cor:A_gamma moment}
Let Assumptions \ref{ass:initial data}, \ref{ass: b sigma lipschitz} and \ref{ass:A b sigma ds dw lipschitz} to \ref{ass:A sigma N lipschitz} hold. Then,
\begin{align*}
\sum_{\beta\in\mathcal{A}_{\gamma}^b\setminus\tilde{\mathcal{A}}_{\gamma}^b}E(|I_{\beta}[b^k(X(s),\alpha(s))]_{s,t}|^2) + \sum_{\beta\in\mathcal{A}_{\gamma}^\sigma\setminus\tilde{\mathcal{A}}_{\gamma}^\sigma}E(|I_{\beta}[\sigma^{(k,j)}(X(s),\alpha(s))]_{s,t}|^2) \leq& C E(1+|X(s)|)^2,
\\
\sum_{\beta\in\tilde{\mathcal{A}}_{\gamma}^b}E(|I_{\beta}[b^k(X(s),\alpha(\cdot))]_{s,t}|^2)+\sum_{\beta\in\tilde{\mathcal{A}}_{\gamma}^\sigma}E(|I_{\beta}[\sigma^{(k,j)}(X(s),\alpha(\cdot))]_{s,t}|^2) \leq& C E(1+|X(s)|)^2,
\end{align*}
for all $k\in\{1,\ldots,d\}$, $j\in\{1,\ldots,m\}$ and $s<t\in[0,T]$.
\end{corollary}
\begin{lemma}
\label{lem:schemeMoment}
Let Assumptions \ref{ass:initial data}, \ref{ass: b sigma lipschitz} and \ref{ass:A b sigma ds dw lipschitz} to \ref{ass:A sigma N lipschitz} be satisfied. Then, the $\gamma$-order explicit numerical scheme \eqref{eq:gen.scheme} of SDEwMS \eqref{eq:sdems} satisfies the following 
\begin{align*}
E\Big(\sup_{n\in \{0,1,\ldots, n_T\}}|Y(t_n)|^2\Big) \leq C
\end{align*}
where the  constant $C>0$ does not depend on $h$. 
\end{lemma}
\begin{proof}
From \eqref{eq:gen.scheme}, we have, 
\begin{align*} 
Y^k&(t_{n})=Y^k_0+\sum^{n-1}_{i=0}\int_{t_{i}}^{t_{i+1}}\Big(\sum_{\beta\in\mathcal{A}_{\gamma}^b\setminus\tilde{\mathcal{A}}_{\gamma}^b}I_{\beta}[b^k(Y(t_{i}),\alpha(t_{i})) ]_{t_{i},s}+\sum_{\beta\in\tilde{\mathcal{A}}_{\gamma}^b}I_{\beta}[b^k(Y(t_{i}),\alpha(\cdot)) ]_{t_{i},s}\Big)ds\notag
\\
&+\sum_{j=1}^m\sum^{n-1}_{i=0}\int_{t_{i}}^{t_{i+1}}\Big(\sum_{\beta\in\mathcal{A}_{\gamma}^\sigma\setminus\tilde{\mathcal{A}}_{\gamma}^\sigma}I_{\beta}[\sigma^{(k,j)}(Y(t_{i}),\alpha(t_{i})) ]_{t_{i},s}+\sum_{\beta\in\tilde{\mathcal{A}}_{\gamma}^\sigma}I_{\beta}[\sigma^{(k,j)}(Y(t_{i}),\alpha(\cdot)) ]_{t_{i},s}\Big)dW^j(s)
\end{align*}
 almost surely for all $n\in\{0,1,\ldots,n_T\}$ and $k\in\{1,\ldots,d\}$.
The H\"older's inequality and Burkholder-Davis-Gundy inequality yield,
\begin{align*}
E&\Big( \sup_{n\in \{0,1,\ldots, n'\}}|Y^k(t_n)|^2\Big)\leq CE|Y^k_0|^2
\\
+CE&\Big(\sup_{n\in \{0,1,\ldots, n'\}}\big|\sum^{n-1}_{i=0}\int_{t_{i}}^{t_{i+1}}\Big(\sum_{\beta\in\mathcal{A}_{\gamma}^b\setminus\tilde{\mathcal{A}}_{\gamma}^b}I_{\beta}[b^k(Y(t_{i}),\alpha(t_{i})) ]_{t_{i},s}+\sum_{\beta\in\tilde{\mathcal{A}}_{\gamma}^b}I_{\beta}[b^k(Y(t_{i}),\alpha(\cdot)) ]_{t_{i},s}\Big)ds\big|^2\Big)
\\
+CE&\Big(\sup_{n\in \{0,1,\ldots, n'\}}\big|\sum_{j=1}^m\sum^{n-1}_{i=0}\int_{t_{i}}^{t_{i+1}}\Big(\sum_{\beta\in\mathcal{A}_{\gamma}^\sigma\setminus\tilde{\mathcal{A}}_{\gamma}^\sigma}I_{\beta}[\sigma^{(k,j)}(Y(t_{i}),\alpha(t_{i})) ]_{t_{i},s}
\\
&+\sum_{\beta\in\tilde{\mathcal{A}}_{\gamma}^\sigma}I_{\beta}[\sigma^{(k,j)}(Y(t_{i}),\alpha(\cdot)) ]_{t_{i},s}\Big)dW^j(s)\big|^2\Big)
\\
\leq C&E|Y^k_0|^2+C\sum^{n'-1}_{i=0}\sum_{\beta\in\mathcal{A}_{\gamma}^b\setminus\tilde{\mathcal{A}}_{\gamma}^b}\int_{t_{i}}^{t_{i+1}}E\big|I_{\beta}[b^k(Y(t_{i}),\alpha(t_{i})) ]_{t_{i},s}\big|^2ds
\\
&+C\sum^{n'-1}_{i=0}\sum_{\beta\in\tilde{\mathcal{A}}_{\gamma}^b}\int_{t_{i}}^{t_{i+1}}E\big|I_{\beta}[b^k(Y(t_{i}),\alpha(\cdot)) ]_{t_{i},s}\big|^2ds
\\
&+C\sum_{j=1}^m\sum^{n'-1}_{i=0}\sum_{\beta\in\mathcal{A}_{\gamma}^\sigma\setminus\tilde{\mathcal{A}}_{\gamma}^\sigma}\int_{t_{i}}^{t_{i+1}}E\big|I_{\beta}[\sigma^{(k,j)}(Y(t_{i}),\alpha(t_{i})) ]_{t_{i},s}\big|^2ds
\\
&+C\sum_{j=1}^m\sum^{n'-1}_{i=0}\sum_{\beta\in\tilde{\mathcal{A}}_{\gamma}^\sigma}\int_{t_{i}}^{t_{i+1}}E\big|I_{\beta}[\sigma^{(k,j)}(Y(t_{i}),\alpha(\cdot)) ]_{t_{i},s}\big|^2ds
\end{align*}
for any $n'\in\{1,\ldots,n_T\}$ and $k\in\{1,\ldots,d\}$. 
Further, on using Corollary \ref{cor:A_gamma moment}, we obtain
\begin{align*}
E\Big( \sup_{n\in \{0,1,\ldots, n'\}}|Y^k(t_n)|^2\Big)\leq CE|Y^k_0|^2+C+Ch\sum^{n'-1}_{i=0}E\Big(\sup_{n\in\{0,1,\ldots,i\}}|Y^k(t_n)|^2\Big)
\end{align*}
for any $n'\in\{1,\ldots,n_T\}$ and $k\in\{1,\ldots,d\}$. 
The application of Gronwall's lemma completes the proof.
\end{proof}

\subsection{Strong rate of convergence}
\label{Rate of scheme}
In this subsection, we provide the proof of Theorem \ref{thm:main}. For this, we first  establish the following lemmas.  
\begin{lemma}
\label{lem:multiple estimate remainder bar_N}
Let Assumptions \ref{ass:initial data}, \ref{ass: b sigma lipschitz} and \ref{ass:A b sigma ds dw lipschitz} be satisfied.
   Then, for all $\beta\in\mathcal{B}(\mathcal{A}_\gamma^b)\setminus\tilde{\mathcal{B}}^b_1$, $\bar{\beta}\in\mathcal{B}(\mathcal{A}_\gamma^\sigma)\setminus\tilde{\mathcal{B}}^\sigma_1$ such that their first component is $\bar{N}_{2\gamma}$, the following holds, 
\begin{align*}
E\Big(|I_{\beta}[b^k(X(\cdot),\alpha(\cdot))]_{s,t}|^2 \Big)\leq& C (t-s)^{\eta(\beta)}
\\
E\Big(|I_{\bar{\beta}}[\sigma^{(k,j)}(X(\cdot),\alpha(\cdot))]_{s,t}|^2 \Big)\leq& C (t-s)^{\eta(\bar{\beta})}
\end{align*}
for all  $k\in\{1,\ldots,d\}$, $j\in\{1,\ldots,m\}$, $s<t\in[0,T]$ provided $0<(t-s)<1/(2q)$.
\end{lemma}
\begin{proof}
Let $l(\beta)=1$, \textit{i.e.}, $\beta=(\bar{N}_{2\gamma})$. 
On using the definition of multiple integrals from Subsection \ref{sub:multiple integral} and Remark \ref{rem:A b sigma ds dw growth}, we obtain
\begin{align*}
E\Big(|I_{\beta}[b^k(X(\cdot),\alpha(\cdot))]_{s,t}|^2 &\Big)=E\Big(|I_{(\bar{N}_{2\gamma})}[b^k(X(\cdot),\alpha(\cdot))]_{s,t}|^2 \Big)
\\
=&E\Big(\big|\sum_{i_0\neq k_0}\int_s^t\mathbbm{1}\{N^{(s,t]}>{2\gamma}\}(b^k(X(u),k_0)-b^k(X(u),i_0))d[M_{i_0k_0}](u)\big|^2\Big)
\\
\leq& CE\Big(\mathbbm{1}\{N^{(s,t]}>{2\gamma}\}(N^{(s,t]})^2E\Big(\sup_{u\in[s,t]}(1+|X(u)|)^2\Big|\mathcal{F}^T_{\alpha}\Big)\Big)
\end{align*}
which further implies due to Theorem \ref{thm:true moment} and Lemma \ref{lem:rateMS},
\begin{align*}
E\Big(|I_{\beta}[b^k(X(\cdot),\alpha(\cdot))]_{s,t}|^2 \Big)\leq&C \sum_{N={2\gamma}+1} N^2P\{N^{(s,t]}=N\}\leq C \sum_{N={2\gamma}+1}N^2((t-s)q)^N
\\
\leq& C (t-s)^{{2\gamma}+1}\sum_{N=0}(N+{2\gamma}+1)^2((t-s)q)^N \leq C (t-s)^{\eta(\beta)}
\end{align*}
for all $k\in\{1,\ldots,d\}$ and $s<t\in[0,T]$ where the series $\displaystyle\sum_{N=0}(N+{2\gamma}+1)^2(q(t-s))^N$ is convergent for $0<(t-s)<1/(2q)$.

Notice that for $l(\beta)>1$, if $j_2,\ldots,j_l\in\{N_1,\ldots,N_{2\gamma}\}$, then $I_{\beta}[b^k(X(\cdot),\alpha(\cdot))]_{s,t}=0$ and the first inequality  holds trivially.

Thus, Lemma \ref{lem:multiple estimate 0,1} gives
\begin{align*}
E\Big(&|I_{\beta}[b^k(X(\cdot),\alpha(\cdot))]_{s,t}|^2 \Big)=E\Big(|I_{(\bar{N}_{2\gamma})\star-\beta}[b^k(X(\cdot),\alpha(\cdot))]_{s,t}|^2 \Big)
\\
\leq& C(t-s)^{\eta(-\beta)-l(-\beta)} \int_{s}^{t}\int_{s}^{t_{1}}\cdots \int_{s}^{t_{l(-\beta)-1}} E\big( \big|I_{(\bar{N}_{2\gamma})}[J^{-\beta}_{\alpha(\cdot)}b^k(X(\cdot),\alpha(\cdot))]_{s,t_{l(-\beta
)}}\big|^2\big)  dt_{l(-\beta)}\cdots dt_2dt_1
\end{align*}
which by the definition of multiple integrals from Subsection \ref{sub:multiple integral} yields,
\begin{align*}
E\Big(&|I_{\beta}[b^k(X(\cdot),\alpha(\cdot))]_{s,t}|^2 \Big) \leq C(t-s)^{\eta(-\beta)-l(-\beta)} \int_{s}^{t}\int_{s}^{t_{1}}\cdots \int_{s}^{t_{l(-\beta)-1}}
\\
& E\Big(\big|\sum_{i_0\neq k_0}\int_{s}^{t_{l(-\beta)}}\mathbbm{1}\{N^{(s,t_{l(-\beta)}]}>{2\gamma} \}  (J^{-\beta}_{k_0}b^k(X(u),k_0)-J^{-\beta}_{i_0}b^k(X(u),i_0))d[M_{i_0k_0}](u)\big|^2\Big)  
\\
&dt_{l(-\beta)}\cdots dt_2dt_1
\end{align*}
 for all $s<t\in[0,T]$ and $k\in\{1,\ldots,d\}$.
 Notice that, $-\beta\in\mathcal{A}^b_\gamma\setminus\tilde{\mathcal{A}}^b_\gamma$.
Thus, by the application of Remark \ref{rem:A b sigma ds dw growth} and Theorem \ref{thm:true moment}, we have
\begin{align*}
E\Big(|I_{\beta}[&b^k(X(\cdot),\alpha(\cdot))]_{s,t}|^2 \Big)\leq C(t-s)^{\eta(-\beta)-l(-\beta)} \int_{s}^{t}\int_{s}^{t_{1}}\cdots \int_{s}^{t_{l(-\beta)-1}}
\\
&E\Big(\big(N^{(s,t_{l(-\beta)}]}\big)^2\mathbbm{1}\{N^{(s,t_{l(-\beta)}]}>{2\gamma} \}  E\Big(\sup_{u\in[s,t]}(1+|X(u)|)^2\Big|\mathcal{F}_T^{\alpha} \Big)\Big)dt_{l(-\beta)}\cdots dt_2dt_1
\\
\leq & C(t-s)^{\eta(-\beta)-l(-\beta)} \int_{s}^{t}\int_{s}^{t_{1}}\cdots\int_{s}^{t_{l(-\beta)-1}}\sum_{N={2\gamma}+1} N^2P\{N^{(s,t_{l(-\beta)}]}=N\}dt_{l(-\beta)}\cdots dt_2dt_1
\end{align*}
and  due to Lemma \ref{lem:rateMS}, 
\begin{align*}
E\Big(|I_{\beta}&[b^k(X(\cdot),\alpha(\cdot))]_{s,t}|^2 \Big)\leq C(t-s)^{\eta(-\beta)-l(-\beta)} \int_{s}^{t}\int_{s}^{t_{1}}\cdots\int_{s}^{t_{l(-\beta)-1}}\sum_{N={2\gamma}+1} N^2((t_{l(-\beta)}-s)q)^N
\\
&\qquad dt_{l(-\beta)}\cdots dt_2dt_1
\\
\leq& C(t-s)^{\eta(-\beta)-l(-\beta)} \int_{s}^{t}\int_{s}^{t_{1}}\cdots\int_{s}^{t_{l(-\beta)-1}}(t_{l(-\beta)}-s)^{{2\gamma}+1}\sum_{N=0}(N+{2\gamma}+1)^2(q(t_{l(-\beta)}-s))^N
\\
&\qquad dt_{l(-\beta)}\cdots dt_2dt_1\leq C(t-s)^{\eta(\beta)}
\end{align*}
for all $k\in\{1,\ldots,d\}$ and $s<t\in[0,T]$ where the the last inequality is obtained by using the finiteness of the series $\displaystyle\sum_{N=0}(N+{2\gamma}+1)^2(q(t_{l(-\beta)}-s))^N$
 for $0<(t-s)<1/(2q)$.
This completes the proof of the first inequality of the lemma.
One can prove the second inequality of the lemma by adapting similar arguments.
\end{proof}
\begin{lemma}\label{lem:multiple estimate remainder N_mu}
Let Assumptions \ref{ass:initial data}, \ref{ass: b sigma lipschitz}, \ref{ass:reminder b N growth} and \ref{ass:reminder sigma N growth} be satisfied.
Then, for all $\beta\in\mathcal{B}(\mathcal{A}_\gamma^b)\setminus\tilde{\mathcal{B}}^b_1$, $\bar{\beta}\in\mathcal{B}(\mathcal{A}_\gamma^\sigma)\setminus\tilde{\mathcal{B}}^\sigma_1$ such that none of the components of $\beta,\bar{\beta}$ is equal to $\bar{N}_{2\gamma}$ and atleast one of their component  is equal to $N_1,\ldots,N_{2\gamma}$,  we have, 
\begin{align*}
E\Big(|I_{\beta}[b^k(X(\cdot),\alpha(\cdot))]_{s,t}|^2 \Big)\leq& C (t-s)^{\eta(\beta)}
\\
E\Big(|I_{\bar{\beta}}[\sigma^{(k,j)}(X(\cdot),\alpha(\cdot))]_{s,t}|^2 \Big)\leq& C (t-s)^{\eta(\bar{\beta})}
\end{align*}
 for all $k\in\{1,\ldots,d\}$, $j\in\{1,\ldots,m\}$ and $s<t\in[0,T]$ provided $0<(t-s)<1/(2q)$.
\end{lemma}
\begin{proof}
Clearly, we can write $\beta=\beta_{r+1}\star(N_{\mu_{r}})\star \beta_{r}\star \cdots\star(N_{\mu_2})\star \beta_2\star(N_{\mu_1})\star \beta_1$ where $\mu_1,\ldots,\mu_r\in\{1,\ldots,2\gamma\}$, $\beta_1,\ldots,\beta_{r+1}\in\mathcal{M}_1$ and $r\in\{1,\ldots, 2\gamma-2\}$.
Further, if $\mu_1> \mu_2>\cdots>\mu_r$ is not satisfied, then $I_\beta[b^k(X(\cdot),\alpha(\cdot))]_{s,t}=0$.
On using Lemma \ref{lem:multiple estimate N_mu} and Remark \ref{rem:reminder b N growth}, we obtain
\begin{align*}
E\Big(|I_{\beta}&[b^k(X(\cdot),\alpha(\cdot))]_{s,t}|^2 \Big)\leq CE\Big(\prod_{i=1}^r\mu_i (t-s)^{\eta(\beta_1)+\cdots +\eta(\beta_{r+1})-l_{r+1}}
\\
&\int_s^{t}\int_{s}^{t_{1}}\cdots\int_{s}^{t_{l_1-1}}\sum_{i_1\neq k_1}\int_{s}^{t_{l_1}}\mathbbm{1}\{N^{(s,t_{l_1}]}=\mu_1\}
\\
&\int_s^{t_{l_1+1}}\int_{s}^{t_{l_1+2}}\cdots\int_{s}^{t_{l_2}}\sum_{i_2\neq k_2}\int_{s}^{t_{l_2+1}}\mathbbm{1}\{N^{(s,t_{l_2+1}]}=\mu_2\} \cdots
\\
&\int_s^{t_{l_{r-1}+r-1}}\int_{s}^{t_{l_{r-1}+r}}\cdots\int_{s}^{t_{l_r+r-2}}\sum_{i_r\neq k_r}\int_{s}^{t_{l_r+r-1}}\mathbbm{1}\{N^{(s,t_{l_r+r-1}]}=\mu_r\}
\\
&\int_s^{t_{l_{r}+r}}\int_{s}^{t_{l_{r}+r+1}}\cdots\int_{s}^{t_{l_{r+1}+r-1}}
\\
&E\Big(|J^{\beta_{r+1}}_{\alpha(t_{l_{r+1}+r})}J^{\beta_r}_{i_rk_r}\cdots J^{\beta_2}_{i_2k_2}(J^{\beta_1}_{k_1}b^k(X(t_{l_{r+1}+r}),k_1)-J^{\beta_1}_{i_1}b^k(X(t_{l_{r+1}+r}),i_1))|^2\Big|\mathcal{F}_T^{\alpha} \Big)
\\
&dt_{l_{r+1}+r}\cdots dt_{l_{r}+r+2}dt_{l_{r}+r+1}d[M_{i_1k_1}](t_{l_r+r})dt_{l_r+r-1}\cdots dt_{l_{r-1}+r+1}dt_{l_{r-1}+r} 
\\
&d[M_{i_1k_1}](t_{l_2+2})dt_{l_2+1}\cdots dt_{l_1+3}dt_{l_1+2}d[M_{i_1k_1}](t_{l_1+1})dt_{l_1}\cdots dt_2dt_1\Big)
\\
\leq& C(t-s)^{n(\beta)+2\bar{n}(\beta)-l(\beta_1)}\int_s^{t}\int_{s}^{t_{1}}\cdots\int_{s}^{t_{l_1-1}}E\Big(\mathbbm{1}\{N^{(s,t_{l_1}]}=\mu_1\}E\Big(\sup_{u\in[s,t]}(1+|X(u)|)^2\Big|\mathcal{F}_T^{\alpha} \Big)\Big)
\\
&dt_{l_1}\cdots dt_2dt_1
\end{align*}
which by the application of Theorem \ref{thm:true moment} and Lemma \ref{lem:rateMS} yield,
\begin{align*}
E\Big(|I_{\beta}[b^k(X(\cdot),\alpha(\cdot))]_{s,t}|^2 \Big)
\leq& C(t-s)^{n(\beta)+2\bar{n}(\beta)-l(\beta_1)}\int_s^{t}\int_{s}^{t_{1}}\cdots\int_{s}^{t_{l_1-1}}
\\
&\qquad P\{N^{(s,t_{l_1}]}=\mu_1\}dt_{l_1}\cdots dt_2dt_1
\leq C(t-s)^{\eta(\beta)}
\end{align*}
for all $k\in\{1,\ldots,d\}$ and $s<t\in[0,T]$ where $0<(t-s)<1/2q$.
By adopting similar arguments, one gets the second inequality of the lemma.
\end{proof}
Using Lemma \ref{lem:multiple estimate 0,1}, Assumption \ref{ass:reminder b sigma ds dw growth} and Theorem \ref{thm:true moment}, the following corollary is obtained.

\begin{corollary}\label{cor:multiple estimate remainder dsdW}
Let Assumptions \ref{ass:initial data}, \ref{ass: b sigma lipschitz} and \ref{ass:reminder b sigma ds dw growth} hold. Then, for all  $\beta\in(\mathcal{B}(\mathcal{A}_\gamma^b)\setminus\tilde{\mathcal{B}}^b_1)\cap\mathcal{M}_1$, $\bar{\beta}\in(\mathcal{B}(\mathcal{A}_\gamma^\sigma)\setminus\tilde{\mathcal{B}}^\sigma_1)\cap\mathcal{M}_1$, $k\in\{1,\ldots,d\}$, $j\in\{1,\ldots,m\}$ and $s<t\in[0,T]$, 
\begin{align*}
E\Big(|I_{\beta}[b^k(X(\cdot),\alpha(\cdot))]_{s,t}|^2 \Big)+E\Big(|I_{\bar{\beta}}[\sigma^{(k,j)}(X(\cdot),\alpha(\cdot))]_{s,t}|^2 \Big)\leq& C (t-s)^{\eta(\beta)}.
\end{align*}
\end{corollary}
\begin{lemma}\label{lem:B_A_b_N_mu}
Let Assumptions \ref{ass:initial data}, \ref{ass: b sigma lipschitz} hold and Assumption \ref{ass:reminder b N growth} be satisfied for all $\beta\in \mathcal{B}(\mathcal{A}_\gamma^b)\setminus\tilde{\mathcal{B}}^b_1$ such that $\eta(\beta)=2\gamma-1$, none of the component of $\beta$ is $1,\ldots,m$ and at least one of its component is $N_1,\ldots,N_{2\gamma-1}$. 
Then, for all $ n_T\in\mathbb{N}$ and $k\in\{1,\ldots,d\}$,
\begin{align*}
E\Big(&\sup_{n\in\{0,1,\ldots, n_T\}}\big|\sum^{n-1}_{i=0}\int_{t_{i}}^{t_{i+1}}I_{\beta}[b^k(X(\cdot),\alpha(\cdot)) ]_{t_{i},s}ds  \big|^2\Big)\leq Ch^{2\gamma}.
\end{align*}
\end{lemma}
\begin{proof}
Clearly, we can write $\beta=\tilde{\beta}\star(N_{\mu_1})\star\bar{\beta}$ where $\tilde{\beta}\in\mathcal{M}$, $\bar{\beta}\in\mathcal{M}_1$ and $\mu_1\in\{1,\ldots,2\gamma-1\}$.
Let $\bar{\beta}\neq\nu$.
Then, on using the definition of multiple integrals from Subsection \ref{sub:multiple integral} and noticing that $d[M_{i_0k_0}](u)$ is a positive measure along with  $dM_{i_0k_0}(u)=d[M_{i_0k_0}](u)-d\langle M_{i_0k_0}\rangle(u)$, one obtains, 
\begin{align*}
E\Big(&\sup_{n\in\{0,1,\ldots, n_T\}}\big|\sum^{n-1}_{i=0}\int_{t_{i}}^{t_{i+1}}I_{\beta}[b^k(X(\cdot),\alpha(\cdot)) ]_{t_{i},s}ds  \big|^2\Big)
\\
=&E\Big(\sup_{n\in\{0,1,\ldots, n_T\}}\big|\sum^{n-1}_{i=0}\int_{t_{i}}^{t_{i+1}}\int_{t_{i}}^s\int_{t_{i}}^{s_1}\cdots\int_{t_{i}}^{s_{l(\bar{\beta})-1}}\sum_{i_0\neq k_0}\int_{t_{i}}^{s_{l(\bar{\beta})}}\mathbbm{1}\{N^{(t_{i},s_{l(\bar{\beta})}]}=\mu_1\}
\\
&\qquad I_{\tilde{\beta}}[J^{\bar{\beta}}_{k_0}b^k(X(\cdot),k_0)-J^{\bar{\beta}}_{i_0}b^k(X(\cdot),i_0)]_{t_{i},u}d[M_{i_0k_0}](u)ds_{l(\bar{\beta})}\cdots ds_2 ds_1 ds \big|^2\Big)
\\
\leq&CE\Big(\sup_{n\in\{0,1,\ldots, n_T\}}\big|\sum^{n-1}_{i=0}\int_{t_{i}}^{t_{i+1}}\int_{t_{i}}^s\int_{t_{i}}^{s_1}\cdots\int_{t_{i}}^{s_{l(\bar{\beta})-1}}\sum_{i_0\neq k_0}\int_{t_{i}}^{s_{l(\bar{\beta})}}\mathbbm{1}\{N^{(t_{i},s_{l(\bar{\beta})}]}=\mu_1\}
\\
&\qquad I_{\tilde{\beta}}[J^{\bar{\beta}}_{k_0}b^k(X(\cdot),k_0)-J^{\bar{\beta}}_{i_0}b^k(X(\cdot),i_0)]_{t_{i},u}dM_{i_0k_0}(u)ds_{l(\bar{\beta})}\cdots ds_2 ds_1ds \big|^2\Big)
\\
&+CE\Big(\sup_{n\in\{0,1,\ldots, n_T\}}\big|\sum^{n-1}_{i=0}\int_{t_{i}}^{t_{i+1}}\int_{t_{i}}^s\int_{t_{i}}^{s_1}\cdots\int_{t_{i}}^{s_{l(\bar{\beta})-1}}\sum_{i_0\neq k_0}\int_{t_{i}}^{s_{l(\bar{\beta})}}\mathbbm{1}\{N^{(t_{i},s_{l(\bar{\beta})}]}=\mu_1\}
\\
&\qquad I_{\tilde{\beta}}[J^{\bar{\beta}}_{k_0}b^k(X(\cdot),k_0)-J^{\bar{\beta}}_{i_0}b^k(X(\cdot),i_0)]_{t_{i},u}d\langle M_{i_0k_0}\rangle(u)ds_{l(\bar{\beta})}\cdots ds_2 ds_1ds \big|^2\Big)
\end{align*}
for all $ n_T\in\mathbb{N}$ and $k\in\{1,\ldots,d\}$.
If $\bar{\beta}=\nu$, then $J^{\bar{\beta}}_{i_0}$ is the identity operator for $i_0 \in \mathcal{S}$,  $s_0=s$ and the iterated integrals $ \int_{t_{i}}^s\int_{t_{i}}^{s_1}\cdots\int_{t_{i}}^{s_{l(\bar{\beta})-1}}$ does not appear on the right side of the above equation and in the forthcoming  estimates.
Notice that the first term on the right side of the above inequality is a martingale. 
On using Doob's martingale inequality, the definition of multiple integrals from Subsection \ref{sub:multiple integral} and H\"older's inequality, we obtain
\begin{align}
E&\Big(\sup_{n\in\{0,1,\ldots, n_T\}}\big|\sum^{n-1}_{i=0}\int_{t_{i}}^{t_{i+1}}I_{\beta}[b^k(X(\cdot),\alpha(\cdot)) ]_{t_{i},s}ds  \big|^2\Big)\notag
\\
\leq&C \sum^{n_T-1}_{i=0}E\big|\int_{t_{i}}^{t_{i+1}}\int_{t_{i}}^s\int_{t_{i}}^{s_1}\cdots\int_{t_{i}}^{s_{l(\bar{\beta})-1}}\sum_{i_0\neq k_0}\int_{t_{i}}^{s_{l(\bar{\beta})}}\mathbbm{1}\{N^{(t_{i},s_{l(\bar{\beta})}]}=\mu_1\}\notag
\\
&\qquad I_{\tilde{\beta}}[J^{\bar{\beta}}_{k_0}b^k(X(\cdot),k_0)-J^{\bar{\beta}}_{i_0}b^k(X(\cdot),i_0)]_{t_{i},u}dM_{i_0k_0}(u)ds_{l(\bar{\beta})}\cdots ds_2 ds_1 ds \big|^2\notag
\\
+&Cn_T \sum^{n_T-1}_{i=0}E\big|\int_{t_{i}}^{t_{i+1}}\int_{t_{i}}^s\int_{t_{i}}^{s_1}\cdots\int_{t_{i}}^{s_{l(\bar{\beta})-1}}\sum_{i_0\neq k_0}\int_{t_{i}}^{s_{l(\bar{\beta})}}\mathbbm{1}\{N^{(t_{i},s_{l(\bar{\beta})}]}=\mu_1\}\notag
\\
&\qquad I_{\tilde{\beta}}[J^{\bar{\beta}}_{k_0}b^k(X(\cdot),k_0)-J^{\bar{\beta}}_{i_0}b^k(X(\cdot),i_0)]_{t_{i},u}d\langle M_{i_0k_0}\rangle(u)ds_{l(\bar{\beta})}\cdots ds_2 ds_1ds \big|^2\notag
\\
\leq&C \sum^{n_T-1}_{i=0}E\big|\int_{t_{i}}^{t_{i+1}}\int_{t_{i}}^s\int_{t_{i}}^{s_1}\cdots\int_{t_{i}}^{s_{l(\bar{\beta})-1}}\sum_{i_0\neq k_0}\int_{t_{i}}^{s_{l(\bar{\beta})}}\mathbbm{1}\{N^{(t_{i},s_{l(\bar{\beta})}]}=\mu_1\}\notag
\\
&\qquad I_{\tilde{\beta}}[J^{\bar{\beta}}_{k_0}b^k(X(\cdot),k_0)-J^{\bar{\beta}}_{i_0}b^k(X(\cdot),i_0)]_{t_{i},u}d[M_{i_0k_0}](u)ds_{l(\bar{\beta})}\cdots ds_2 ds_1 ds \big|^2\notag
\\
+&Cn_T \sum^{n_T-1}_{i=0}E\big|\int_{t_{i}}^{t_{i+1}}\int_{t_{i}}^s\int_{t_{i}}^{s_1}\cdots\int_{t_{i}}^{s_{l(\bar{\beta})-1}}\mathbbm{1}\{N^{(t_{i},s_{l(\bar{\beta})}]}=\mu_1\}\sum_{i_0\neq k_0}\int_{t_{i}}^{s_{l(\bar{\beta})}}q_{i_0k_0}\mathbbm{1}\{\alpha(u-)=i_0\}\notag
\\
&\qquad I_{\tilde{\beta}}[J^{\bar{\beta}}_{k_0}b^k(X(\cdot),k_0)-J^{\bar{\beta}}_{i_0}b^k(X(\cdot),i_0)]_{t_{i},u}duds_{l(\bar{\beta})}\cdots ds_2 ds_1ds \big|^2\notag
\\
\leq&Ch\sum^{n_T-1}_{i=0}\int_{t_{i}}^{t_{i+1}}E|I_{\beta}[b^k(X(\cdot),\alpha(\cdot)) ]_{t_{i},s}|^2ds\notag
\\
&+Ch^{l(\bar{\beta})+1} \sum^{n_T-1}_{i=0}\int_{t_{i}}^{t_{i+1}}\int_{t_{i}}^s\int_{t_{i}}^{s_1}\cdots\int_{t_{i}}^{s_{l(\bar{\beta})-1}}\sum_{i_0\neq k_0}\int_{t_{i}}^{s_{l(\bar{\beta})}}E\Big(\mathbbm{1}\{N^{(t_{i},s_{l(\bar{\beta})}]}=\mu_1\}\notag
\\
&\qquad E\Big(|I_{\tilde{\beta}}[J^{\bar{\beta}}_{k_0}b^k(X(\cdot),k_0)-J^{\bar{\beta}}_{i_0}b^k(X(\cdot),i_0)]_{t_{i},u}|^2\Big|\mathcal{F}_T^{\alpha}\Big)\Big)duds_{l(\bar{\beta})}\cdots ds_2 ds_1ds\label{eq:angle M} 
\end{align}
for all $ n_T\in\mathbb{N}$ and $k\in\{1,\ldots,d\}$.
Now, for different possibilities of $\tilde{\beta}\in\mathcal{M}$, we estimate $E\Big(|I_{\tilde{\beta}}[J^{\bar{\beta}}_{k_0}b^k(X(\cdot),k_0)-J^{\bar{\beta}}_{i_0}b^k(X(\cdot),i_0)]_{t_{i},u}|^2\Big|\mathcal{F}_T^{\alpha}\Big)$ for all $u\in[t_i,t_{i+1}]$, $i\in\{0,\ldots,n_T-1\}$, $n_T\in\mathbb{N}$ and $k\in\{1,\ldots,d\}$.
If $\tilde{\beta}=\nu$, then by Remark \ref{rem:reminder b N growth} and Theorem \ref{thm:true moment}, we have
\begin{align*}
E\Big(|I_{\tilde{\beta}}[J^{\bar{\beta}}_{k_0}b^k(X(\cdot),k_0)-J^{\bar{\beta}}_{i_0}b^k(X(\cdot),i_0)]_{t_{i},u}|^2\Big|\mathcal{F}_T^{\alpha}\Big)\leq& CE\Big((1+|X(u)|)^2\Big|\mathcal{F}_T^{\alpha}\Big)\leq C
\end{align*} 
for all $u\in[t_i,t_{i+1}]$, $i\in\{0,\ldots,n_T-1\}$, and $k\in\{1,\ldots,d\}$.
Now, if $\tilde{\beta}\in\mathcal{M}_1\setminus\{\nu\}$, then Remark \ref{rem:reminder b N growth}, Lemma \ref{lem:multiple estimate 0,1} and Theorem \ref{thm:true moment} yield,
\begin{align*}
E\Big(|I_{\tilde{\beta}}[J^{\bar{\beta}}_{k_0}b^k(X(\cdot),k_0)-&J^{\bar{\beta}}_{i_0}b^k(X(\cdot),i_0)]_{t_{i},u}|^2\Big|\mathcal{F}_T^{\alpha}\Big)
\\
\leq& C(u-t_i)^{n(\tilde{\beta})+2\bar{n}(\tilde{\beta})}E\Big(\sup_{v\in[0,T]}(1+|X(v)|)^2\Big|\mathcal{F}_T^{\alpha}\Big)
\leq C(u-t_i)^{n(\tilde{\beta})+2\bar{n}(\tilde{\beta})}
\end{align*}
for all $u\in[t_i,t_{i+1}]$, $i\in\{0,\ldots,n_T-1\}$, $n_T\in\mathbb{N}$ and $k\in\{1,\ldots,d\}$.
Further, if $\tilde{\beta}\in\mathcal{M}_2\cup\mathcal{M}_3$, then on using Remark \ref{rem:reminder b N growth}, Lemma \ref{lem:multiple estimate N_mu} and Theorem \ref{thm:true moment}, we obtain
\begin{align*}
E\Big(|I_{\tilde{\beta}}[J^{\bar{\beta}}_{k_0}b^k(X(\cdot),k_0)&-J^{\bar{\beta}}_{i_0}b^k(X(\cdot),i_0)]_{t_{i},u}|^2\Big|\mathcal{F}_T^{\alpha}\Big)
\\
\leq& C(u-t_i)^{n(\tilde{\beta})+2\bar{n}(\tilde{\beta})}E\Big(\sup_{v\in[0,T]}(1+|X(v)|)^2\Big|\mathcal{F}_T^{\alpha}\Big)\leq C(u-t_i)^{n(\tilde{\beta})+2\bar{n}(\tilde{\beta})}
\end{align*}
for all $u\in[t_i,t_{i+1}]$, $i\in\{0,\ldots,n_T-1\}$, $n_T\in\mathbb{N}$ and $k\in\{1,\ldots,d\}$.
Hence, \eqref{eq:angle M} becomes
\begin{align*}
E&\Big(\sup_{n\in\{0,1,\ldots, n_T\}}\big|\sum^{n-1}_{i=0}\int_{t_{i}}^{t_{i+1}}I_{\beta}[b^k(X(\cdot),\alpha(\cdot)) ]_{t_{i},s}ds  \big|^2\Big)\leq Ch\sum^{n_T-1}_{i=0}\int_{t_{i}}^{t_{i+1}}E|I_{\beta}[b^k(X(\cdot),\alpha(\cdot)) ]_{t_{i},s}|^2ds\notag
\\
&+Ch^{l(\bar{\beta})+1} \sum^{n_T-1}_{i=0}\int_{t_{i}}^{t_{i+1}}\int_{t_{i}}^s\int_{t_{i}}^{s_1}\cdots\int_{t_{i}}^{s_{l(\bar{\beta})-1}}\int_{t_{i}}^{s_{l(\bar{\beta})}}E(\mathbbm{1}\{N^{(t_{i},s_{l(\bar{\beta})}]}=\mu_1\})(u-t_i)^{n(\tilde{\beta})+2\bar{n}(\tilde{\beta})} 
\\
&\qquad duds_{l(\bar{\beta})}\cdots ds_2 ds_1ds
\end{align*}
for all $ n_T\in\mathbb{N}$ and $k\in\{1,\ldots,d\}$.
By the application of Lemmas \ref{lem:multiple estimate remainder N_mu} and \ref{lem:rateMS}, we have
\begin{align*}
E&\Big(\sup_{n\in\{0,1,\ldots, n_T\}}\big|\sum^{n-1}_{i=0}\int_{t_{i}}^{t_{i+1}}I_{\beta}[b^k(X(\cdot),\alpha(\cdot)) ]_{t_{i},s}ds  \big|^2\Big)\leq Ch^{2\gamma}
\\
&+Ch^{l(\bar{\beta})+1} \sum^{n_T-1}_{i=0}\int_{t_{i}}^{t_{i+1}}\int_{t_{i}}^s\int_{t_{i}}^{s_1}\cdots\int_{t_{i}}^{s_{l(\bar{\beta})-1}}\int_{t_{i}}^{s_{l(\bar{\beta})}}P\{N^{(t_{i},s_{l(\bar{\beta})}]}=\mu_1\}(u-t_i)^{n(\tilde{\beta})+2\bar{n}(\tilde{\beta})} 
\\
&\qquad duds_{l(\bar{\beta})}\cdots ds_2 ds_1ds
\\
\leq& Ch^{2\gamma}
\end{align*}
for all $ n_T\in\mathbb{N}$ and $k\in\{1,\ldots,d\}$.
\end{proof}
\begin{lemma}\label{lem:B_A_b estimate}
Let Assumptions \ref{ass:initial data}, \ref{ass: b sigma lipschitz}, \ref{ass:A b sigma ds dw lipschitz}, \ref{ass:reminder b sigma ds dw growth} and \ref{ass:reminder b N growth} hold.
 Then,
\begin{align*}
\sum_{\beta\in \mathcal{B}(\mathcal{A}_\gamma^b)\setminus\tilde{\mathcal{B}}^b_1 }E\Big(&\sup_{n\in\{0,1,\ldots, n_T\}}\big|\sum^{n-1}_{i=0}\int_{t_{i}}^{t_{i+1}}I_{\beta}[b^k(X(\cdot),\alpha(\cdot)) ]_{t_{i},s}ds  \big|^2\Big)\leq Ch^{2\gamma}, 
\end{align*}
for all $ n_T\in\mathbb{N}$ and $k\in\{1,\ldots,d\}$.
\end{lemma}
\begin{proof}
We write $\mathcal{B}(\mathcal{A}_\gamma^b)\setminus\tilde{\mathcal{B}}^b_1=\mathcal{C}_1\cup\mathcal{C}_{21}\cup\mathcal{C}_{22}\cup\mathcal{C}_{31}\cup\mathcal{C}_{32}\cup\mathcal{C}_{33}$ where
\begin{itemize}
\item[] $\mathcal{C}_1:=\{\beta=(j_1,\ldots,j_l)\in \mathcal{B}(\mathcal{A}_\gamma^b)\setminus\tilde{\mathcal{B}}^b_1:\eta(\beta)=2\gamma-1,j_i\notin\{1,\ldots,m\} \forall i\in\{1,\ldots,l\},j_k\in\{N_1,\ldots,N_{2\gamma}\}\mbox{ for any }k\in\{1,\ldots,l\}\}$,
\item[] $\mathcal{C}_{21}:=\{\beta=(j_1,\ldots,j_l)\in \mathcal{B}(\mathcal{A}_\gamma^b)\setminus\tilde{\mathcal{B}}^b_1:\eta(\beta)=2\gamma-1,j_i\in\{1,\ldots,m\}\mbox{ for any }i\in\{1,\ldots,l\},j_i\notin\{N_1,\ldots,N_{2\gamma}\} \forall i\in\{1,\ldots,l\}\}$,
\item[]  $\mathcal{C}_{22}:=\{\beta=(j_1,\ldots,j_l)\in \mathcal{B}(\mathcal{A}_\gamma^b)\setminus\tilde{\mathcal{B}}^b_1:\eta(\beta)=2\gamma-1,j_i\in\{1,\ldots,m\}\mbox{ for any }i\in\{1,\ldots,l\},j_k\in\{N_1,\ldots,N_{2\gamma}\} \mbox{ for any } k\in\{1,\ldots,l\}\}$,
\item[] $\mathcal{C}_{31}:=\{\beta=(j_1,\ldots,j_l)\in \mathcal{B}(\mathcal{A}_\gamma^b)\setminus\tilde{\mathcal{B}}^b_1:\eta(\beta)\geq 2\gamma,j_1=\bar{N}_{2\gamma}\}$,
\item[]  $\mathcal{C}_{32}:=\{\beta=(j_1,\ldots,j_l)\in \mathcal{B}(\mathcal{A}_\gamma^b)\setminus\tilde{\mathcal{B}}^b_1:\eta(\beta)\geq 2\gamma,j_i\in\{0,1,\ldots,m\}\forall i\in\{1,\ldots,l\}\}$,
\item[]  $\mathcal{C}_{33}:=\{\beta=(j_1,\ldots,j_l)\in \mathcal{B}(\mathcal{A}_\gamma^b)\setminus\tilde{\mathcal{B}}^b_1:\eta(\beta)\geq 2\gamma,J_1\neq\bar{N} _{2\gamma},j_i\in\{N_1,\ldots,N_{2\gamma}\} \mbox{ for any } i\in\{1,\ldots,l\}\}$.
\end{itemize}
We can write
\begin{align*}
 \sum_{\beta\in \mathcal{B}(\mathcal{A}_\gamma^b)\setminus\tilde{\mathcal{B}}^b_1 }E\Big(&\sup_{n\in\{0,1,\ldots, n_T\}}\big|\sum^{n-1}_{i=0}\int_{t_{i}}^{t_{i+1}}I_{\beta}[b^k(X(\cdot),\alpha(\cdot)) ]_{t_{i},s}ds  \big|^2\Big)
\\
=&\sum_{\beta\in\mathcal{C}_{1}  }E\Big(\sup_{n\in\{0,1,\ldots, n_T\}}\big|\sum^{n-1}_{i=0}\int_{t_{i}}^{t_{i+1}}I_{\beta}[b^k(X(\cdot),\alpha(\cdot)) ]_{t_{i},s}ds  \big|^2\Big)
\\
&+\sum_{\beta\in\mathcal{C}_{21}  }E\Big(\sup_{n\in\{0,1,\ldots, n_T\}}\big|\sum^{n-1}_{i=0}\int_{t_{i}}^{t_{i+1}}I_{\beta}[b^k(X(\cdot),\alpha(\cdot)) ]_{t_{i},s}ds  \big|^2\Big)
\\
&+\sum_{\beta\in\mathcal{C}_{22}  }E\Big(\sup_{n\in\{0,1,\ldots, n_T\}}\big|\sum^{n-1}_{i=0}\int_{t_{i}}^{t_{i+1}}I_{\beta}[b^k(X(\cdot),\alpha(\cdot)) ]_{t_{i},s}ds  \big|^2\Big)
\\
&+ \sum_{\beta\in\mathcal{C}_{31}  }E\Big(\sup_{n\in\{0,1,\ldots, n_T\}}\big|\sum^{n-1}_{i=0}\int_{t_{i}}^{t_{i+1}}I_{\beta}[b^k(X(\cdot),\alpha(\cdot)) ]_{t_{i},s}ds  \big|^2\Big)
\\
&+\sum_{\beta\in\mathcal{C}_{32}  }E\Big(\sup_{n\in\{0,1,\ldots, n_T\}}\big|\sum^{n-1}_{i=0}\int_{t_{i}}^{t_{i+1}}I_{\beta}[b^k(X(\cdot),\alpha(\cdot)) ]_{t_{i},s}ds  \big|^2\Big)
\\
&+ \sum_{\beta\in\mathcal{C}_{33}  }E\Big(\sup_{n\in\{0,1,\ldots, n_T\}}\big|\sum^{n-1}_{i=0}\int_{t_{i}}^{t_{i+1}}I_{\beta}[b^k(X(\cdot),\alpha(\cdot)) ]_{t_{i},s}ds  \big|^2\Big)
\end{align*}
for all $ n_T\in\mathbb{N}$ and $k\in\{1,\ldots,d\}$.
Notice that $I_{\beta}[b^k(X(\cdot),\alpha(\cdot)) ]_{t_{i},s}$ is a martingale for every $\beta\in\mathcal{C}_{21}$ and $\beta\in\mathcal{C}_{22}$. 
Then, Doob's martingale inequality, H\"older's inequality and Lemma \ref{lem:B_A_b_N_mu} yield,
\begin{align*}
&\sum_{\beta\in \mathcal{B}(\mathcal{A}_\gamma^b)\setminus\tilde{\mathcal{B}}^b_1 }E\Big(\sup_{n\in\{0,1,\ldots, n_T\}}\big|\sum^{n-1}_{i=0}\int_{t_{i}}^{t_{i+1}}I_{\beta}[b^k(X(\cdot),\alpha(\cdot)) ]_{t_{i},s}ds  \big|^2\Big)\leq Ch^{2\gamma}
\\
&+C h\sum^{n_T-1}_{i=0}\Big(\sum_{\beta\in\mathcal{C}_{21}  }\int_{t_{i}}^{t_{i+1}}E|I_{\beta}[b^k(X(\cdot),\alpha(\cdot)) ]_{t_{i},s}|^2ds+\sum_{\beta\in\mathcal{C}_{22}  }\int_{t_{i}}^{t_{i+1}}E|I_{\beta}[b^k(X(\cdot),\alpha(\cdot)) ]_{t_{i},s}|^2ds\Big)
\\
&+C \sum^{n_T-1}_{i=0}\Big(\sum_{\beta\in\mathcal{C}_{31}  }\int_{t_{i}}^{t_{i+1}}E|I_{\beta}[b^k(X(\cdot),\alpha(\cdot)) ]_{t_{i},s}|^2ds+ \sum_{\beta\in\mathcal{C}_{32}  }\int_{t_{i}}^{t_{i+1}}E|I_{\beta}[b^k(X(\cdot),\alpha(\cdot)) ]_{t_{i},s}|^2ds\Big)
\\
&+C \sum_{\beta\in\mathcal{C}_{33}  }\sum^{n_T-1}_{i=0}\int_{t_{i}}^{t_{i+1}}E|I_{\beta}[b^k(X(\cdot),\alpha(\cdot)) ]_{t_{i},s}|^2ds
\end{align*} 
for all $ n_T\in\mathbb{N}$ and $k\in\{1,\ldots,d\}$.
Hence on using Lemmas \ref{lem:multiple estimate remainder bar_N}, \ref{lem:multiple estimate remainder N_mu} and Corollary \ref{cor:multiple estimate remainder dsdW}, we obtain
\begin{align*}
\sum_{\beta\in \mathcal{B}(\mathcal{A}_\gamma^b)\setminus\tilde{\mathcal{B}}^b_1 }&E\Big(\sup_{n\in\{0,1,\ldots, n_T\}}\big|\sum^{n-1}_{i=0}\int_{t_{i}}^{t_{i+1}}I_{\beta}[b^k(X(\cdot),\alpha(\cdot)) ]_{t_{i},s}ds  \big|^2\Big)\leq Ch^{2\gamma}
\\
&+Ch\sum^{n_T-1}_{i=0}\int_{t_{i}}^{t_{i+1}}\Big(\sum_{\beta\in\mathcal{C}_{21}  }(s-t_{i})^{\eta(\beta)}ds+\sum_{\beta\in\mathcal{C}_{22}  }(s-t_{i})^{\eta(\beta)}ds\Big)
\\
&+C\sum^{n_T-1}_{i=0}\int_{t_{i}}^{t_{i+1}}\Big(\sum_{\beta\in\mathcal{C}_{31}  }(s-t_{i})^{\eta(\beta)}ds+\sum_{\beta\in\mathcal{C}_{32}  }(s-t_{i})^{\eta(\beta)}ds+\sum_{\beta\in\mathcal{C}_{33}  }(s-t_{i})^{\eta(\beta)}ds\Big)
\\
\leq& Ch^{2\gamma}
\end{align*}
for all $ n_T\in\mathbb{N}$ and $k\in\{1,\ldots,d\}$.
\end{proof}
\begin{lemma}\label{lem:B_A_sigma estimate}
Let  Assumptions \ref{ass:initial data}, \ref{ass: b sigma lipschitz}, \ref{ass:A b sigma ds dw lipschitz}, \ref{ass:reminder b sigma ds dw growth} and \ref{ass:reminder sigma N growth} be satisfied. Then
\begin{align*}
\sum_{\beta\in \mathcal{B}(\mathcal{A}_\gamma^\sigma)\setminus\tilde{\mathcal{B}}^\sigma_1}E\Big(&\sup_{n\in\{0,1,\ldots, n_T\}}\big|\sum^{n-1}_{i=0}\int_{t_{i}}^{t_{i+1}}I_{\beta}[\sigma^{(k,j)}(X(\cdot),\alpha(\cdot)) ]_{t_{i},s}dW^j(s)  \big|^2\Big)\leq Ch^{2\gamma}
\end{align*}
for all $ n_T\in\mathbb{N}$, $k\in\{1,\ldots,d\}$ and $j\in\{1,\ldots,m\}$. 
\end{lemma}
\begin{proof}
We write $\mathcal{B}(\mathcal{A}_\gamma^\sigma)\setminus\tilde{\mathcal{B}}^\sigma_1=\mathcal{D}_{1}\cup\mathcal{D}_{2}\cup\mathcal{D}_{3}$ where
\begin{itemize}
\item[] $\mathcal{D}_{1}:=\{(j_1,\ldots,j_l)\in \mathcal{B}(\mathcal{A}_\gamma^\sigma)\setminus\tilde{\mathcal{B}}^\sigma_1:j_1=\bar{N}_{2\gamma}\}$,
\item[]  $\mathcal{D}_{2}:=\{(j_1,\ldots,j_l)\in \mathcal{B}(\mathcal{A}_\gamma^\sigma)\setminus\tilde{\mathcal{B}}^\sigma_1:j_i\in\{0,1,\ldots,m\}\forall i\in\{1,\ldots,l\}\}$,
\item[]  $\mathcal{D}_{3}:=\{(j_1,\ldots,j_l)\in \mathcal{B}(\mathcal{A}_\gamma^\sigma)\setminus\tilde{\mathcal{B}}^\sigma_1:j_1\neq\bar{N}_{2\gamma},j_i\in\{N_1,\ldots,N_{2\gamma}\} \mbox{ for any } i\in\{1,\ldots,l\}\}$.
\end{itemize}
On using Doob's martingale inequality and Burkholder-Davis-Gundy inequality, we have
\begin{align*}
 &\sum_{\beta\in \mathcal{B}(\mathcal{A}_\gamma^\sigma)\setminus\tilde{\mathcal{B}}^\sigma_1}E\Big(\sup_{n\in\{0,1,\ldots, n_T\}}\big|\sum^{n-1}_{i=0}\int_{t_{i}}^{t_{i+1}}I_{\beta}[\sigma^{(k,j)}(X(\cdot),\alpha(\cdot)) ]_{t_{i},s}dW^j(s)  \big|^2\Big)
\\
\leq& C\sum^{n_T-1}_{i=0}\Big(\sum_{\beta\in\mathcal{D}_{1}}\int_{t_{i}}^{t_{i+1}}E|I_{\beta}[\sigma^{(k,j)}(X(\cdot),\alpha(\cdot)) ]_{t_{i},s}|^2ds+\sum_{\beta\in\mathcal{D}_{2}}\int_{t_{i}}^{t_{i+1}}E|I_{\beta}[\sigma^{(k,j)}(X(\cdot),\alpha(\cdot)) ]_{t_{i},s}|^2ds
\\
&+\sum_{\beta\in\mathcal{D}_{3}}\int_{t_{i}}^{t_{i+1}}E|I_{\beta}[\sigma^{(k,j)}(X(\cdot),\alpha(\cdot)) ]_{t_{i},s}|^2ds\Big)
\end{align*}
for all $ n_T\in\mathbb{N}$, $k\in\{1,\ldots,d\}$ and $j\in\{1,\ldots,m\}$. 
The application of Lemmas \ref{lem:multiple estimate remainder bar_N}, \ref{lem:multiple estimate remainder N_mu} and Corollary \ref{cor:multiple estimate remainder dsdW} yield,
\begin{align*}
\sum_{\beta\in \mathcal{B}(\mathcal{A}_\gamma^\sigma)\setminus\tilde{\mathcal{B}}^\sigma_1}&E\Big(\sup_{n\in\{0,1,\ldots, n_T\}}\big|\sum^{n-1}_{i=0}\int_{t_{i}}^{t_{i+1}}I_{\beta}[\sigma^{(k,j)}(X(\cdot),\alpha(\cdot)) ]_{t_{i},s}dW^j(s)  \big|^2\Big)
\\
\leq C&\sum^{n_T-1}_{i=0}\Big(\sum_{\beta\in\mathcal{D}_{1}}\int_{t_{i}}^{t_{i+1}}(s-t_{i})^{\eta(\beta)}ds+ \sum_{\beta\in\mathcal{D}_{2}}\int_{t_{i}}^{t_{i+1}}(s-t_{i})^{\eta(\beta)}ds+ \sum_{\beta\in\mathcal{D}_{3}}\int_{t_{i}}^{t_{i+1}}(s-t_{i})^{\eta(\beta)}ds\Big)
\\
\leq C&h^{2\gamma}
\end{align*}
for all $n_T\in\mathbb{N}$, $k\in\{1,\ldots,d\}$ and $j\in\{1,\ldots,m\}$.
\end{proof}
\begin{lemma}\label{lem:convergence B_1 estimate}
Let Assumptions \ref{ass:initial data}, \ref{ass: b sigma lipschitz} and \ref{ass:A b sigma ds dw lipschitz} hold. Then,
\begin{align*}
\sum_{\beta\in\mathcal{B}_1^b}E\Big(\sup_{n\in\{1,\ldots,n_T\}}\Big|\sum_{i=0}^{n-1}\int_{t_i}^{t_{i+1}}I_{\beta}[b^k(X(\cdot),\alpha(\cdot))-b^k(X(t_i),\alpha(t_i))]_{t_i,s}ds\Big|^2\Big)\leq&Ch^{2\gamma},
\\
\sum_{\beta\in\mathcal{B}_1^\sigma}E\Big(\sup_{n\in\{1,\ldots,n_T\}}\Big|\sum_{i=0}^{n-1}\int_{t_i}^{t_{i+1}}I_{\beta}[\sigma^{(k,j)}(X(\cdot),\alpha(\cdot))-\sigma^{(k,j)}(X(t_i),\alpha(t_i))]_{t_i,s}dW^j(s)\Big|^2\Big)\leq& Ch^{2\gamma}
\end{align*}
for all $k\in\{1,\ldots,d\}$, $j\in\{1,\ldots,m\}$ and $n_T\in\mathbb{N}$.
\end{lemma}
\begin{proof}
We write $\mathcal{B}_1^b=\mathcal{B}_{11}^b\cup\mathcal{B}_{12}^b$ where
\begin{itemize}
\item[] $\mathcal{B}_{11}^b:=\{\beta\in\mathcal{B}_1^b:$ the components of $\beta$ are not equal to $1,\ldots,m\} $ and
\item[] $\mathcal{B}_{12}^b:=\{\beta\in\mathcal{B}_1^b:$ at least one of the components of $\beta$ is equal to $1,\ldots,m\} $.
\end{itemize}
Notice that $I_{\beta}[b^k(X(\cdot),\alpha(\cdot)) ]_{t_{i},s}$ is a martingale for every $\beta\in\mathcal{B}_{12}^b$.
Thus, Doob's martingale inequality and H\"older's inequality yield,
\begin{align*}
\sum_{\beta\in\mathcal{B}_1^b}E\Big(\sup_{n\in\{1,\ldots,n_T\}}\Big|&\sum_{i=0}^{n-1}\int_{t_i}^{t_{i+1}}I_{\beta}[b^k(X(\cdot),\alpha(\cdot))-b^k(X(t_i),\alpha(t_i))]_{t_i,s}ds\Big|^2\Big)
\\
=&\sum_{\beta\in\mathcal{B}_{11}^b}E\Big(\sup_{n\in\{1,\ldots,n_T\}}\Big|\sum_{i=0}^{n-1}\int_{t_i}^{t_{i+1}}I_{\beta}[b^k(X(\cdot),\alpha(\cdot))-b^k(X(t_i),\alpha(t_i))]_{t_i,s}ds\Big|^2\Big)
\\
&+\sum_{\beta\in\mathcal{B}_{12}^b}E\Big(\sup_{n\in\{1,\ldots,n_T\}}\Big|\sum_{i=0}^{n-1}\int_{t_i}^{t_{i+1}}I_{\beta}[b^k(X(\cdot),\alpha(\cdot))-b^k(X(t_i),\alpha(t_i))]_{t_i,s}ds\Big|^2\Big)
\\
\leq&C \sum_{\beta\in\mathcal{B}_{11}^b}\sum_{i=0}^{n_T-1}\int_{t_i}^{t_{i+1}}E|I_{\beta}[b^k(X(\cdot),\alpha(\cdot))-b^k(X(t_i),\alpha(t_i))]_{t_i,s}|^2ds
\\
&+Ch\sum_{\beta\in\mathcal{B}_{12}^b}\sum_{i=0}^{n_T-1}\int_{t_i}^{t_{i+1}}E|I_{\beta}[b^k(X(\cdot),\alpha(\cdot))-b^k(X(t_i),\alpha(t_i))]_{t_i,s}|^2ds
\end{align*}
which further implies, 
\begin{align*}
\sum_{\beta\in\mathcal{B}_1^b}E\Big(&\sup_{n\in\{1,\ldots,n_T\}}\Big|\sum_{i=0}^{n-1}\int_{t_i}^{t_{i+1}}I_{\beta}[b^k(X(\cdot),\alpha(\cdot))-b^k(X(t_i),\alpha(t_i))]_{t_i,s}ds\Big|^2\Big)
\\
\leq&C \sum_{\beta\in\mathcal{B}_{11}^b}\sum_{i=0}^{n_T-1}\int_{t_i}^{t_{i+1}}E|I_{\beta}[b^k(X(\cdot),\alpha(\cdot))-b^k(X(t_i),\alpha(\cdot))]_{t_i,s}|^2ds
\\
&+C \sum_{\beta\in\mathcal{B}_{11}^b}\sum_{i=0}^{n_T-1}\int_{t_i}^{t_{i+1}}E\Big(\mathbbm{1}\{N^{(t_i,s]}\geq 1\}|I_{\beta}[b^k(X(t_i),\alpha(\cdot))-b^k(X(t_i),\alpha(t_i))]_{t_i,s}|^2\Big)ds
\\
&+hC \sum_{\beta\in\mathcal{B}_{12}^b}\sum_{i=0}^{n_T-1}\int_{t_i}^{t_{i+1}}E|I_{\beta}[b^k(X(\cdot),\alpha(\cdot))-b^k(X(t_i),\alpha(\cdot))]_{t_i,s}|^2ds
\\
&+hC \sum_{\beta\in\mathcal{B}_{12}^b}\sum_{i=0}^{n_T-1}\int_{t_i}^{t_{i+1}}E\Big(\mathbbm{1}\{N^{(t_i,s]}\geq 1\}|I_{\beta}[b^k(X(t_i),\alpha(\cdot))-b^k(X(t_i),\alpha(t_i))]_{t_i,s}|^2\Big)ds
\\
\leq&C \sum_{\beta\in\mathcal{B}_{11}^b}\sum_{i=0}^{n_T-1}\int_{t_i}^{t_{i+1}}E|I_{\beta}[b^k(X(\cdot),\alpha(\cdot))-b^k(X(t_i),\alpha(\cdot))]_{t_i,s}|^2ds
\\
&+C \sum_{\beta\in\mathcal{B}_{11}^b}\sum_{i=0}^{n_T-1}\int_{t_i}^{t_{i+1}}E\Big(\mathbbm{1}\{N^{(t_i,s]}\geq 1\}E\Big(|I_{\beta}[b^k(X(t_i),\alpha(\cdot))]_{t_i,s}|^2\Big|\mathcal{F}_T^{\alpha}\Big)\Big)ds
\\
&+C \sum_{\beta\in\mathcal{B}_{11}^b}\sum_{i=0}^{n_T-1}\int_{t_i}^{t_{i+1}}E\Big(\mathbbm{1}\{N^{(t_i,s]}\geq 1\}E\Big(|I_{\beta}[b^k(X(t_i),\alpha(t_i))]_{t_i,s}|^2\Big|\mathcal{F}_T^{\alpha}\Big)\Big)ds
\\
&+Ch \sum_{\beta\in\mathcal{B}_{12}^b}\sum_{i=0}^{n_T-1}\int_{t_i}^{t_{i+1}}E|I_{\beta}[b^k(X(\cdot),\alpha(\cdot))-b^k(X(t_i),\alpha(\cdot))]_{t_i,s}|^2ds
\\
&+Ch \sum_{\beta\in\mathcal{B}_{12}^b}\sum_{i=0}^{n_T-1}\int_{t_i}^{t_{i+1}}E\Big(\mathbbm{1}\{N^{(t_i,s]}\geq 1\}E\Big(|I_{\beta}[b^k(X(t_i),\alpha(\cdot))]_{t_i,s}|^2\Big|\mathcal{F}_T^{\alpha}\Big)\Big)ds
\\
&+Ch \sum_{\beta\in\mathcal{B}_{12}^b}\sum_{i=0}^{n_T-1}\int_{t_i}^{t_{i+1}}E\Big(\mathbbm{1}\{N^{(t_i,s]}\geq 1\}E\Big(|I_{\beta}[b^k(X(t_i),\alpha(t_i))]_{t_i,s}|^2\Big|\mathcal{F}_T^{\alpha}\Big)\Big)ds
\end{align*}
for all $k\in\{1,\ldots,d\}$ and $n_T\in\mathbb{N}$.
Moreover, Assumptions \ref{ass: b sigma lipschitz}, \ref{ass:A b sigma ds dw lipschitz}, Remark \ref{rem:A b sigma ds dw growth} and Lemma \ref{lem:multiple estimate 0,1} yield,
\begin{align*}
\sum_{\beta\in\mathcal{B}_1^b}E\Big(&\sup_{n\in\{1,\ldots,n_T\}}\Big|\sum_{i=0}^{n-1}\int_{t_i}^{t_{i+1}}I_{\beta}[b^k(X(\cdot),\alpha(\cdot))-b^k(X(t_i),\alpha(t_i))]_{t_i,s}ds\Big|^2\Big)
\\
\leq& C\sum_{i=0}^{n_T-1}\sum_{\beta\in\mathcal{B}_{11}^b}\int_{t_i}^{t_{i+1}}(s-t_i)^{2\gamma-1}E\Big(\sup_{u\in[t_i,s]}|X(u)-X(t_i)|^2\Big)ds
\\
&+C\sum_{i=0}^{n_T-1}\sum_{\beta\in\mathcal{B}_{11}^b}\int_{t_i}^{t_{i+1}}(s-t_i)^{2\gamma-1}E\Big(\mathbbm{1}\{N^{(t_i,s]}\geq 1\}E\Big(\sup_{u\in[t_i,s]}(1+|X(u)|)^2\Big|\mathcal{F}_T^{\alpha}\Big)\Big)ds
\\
&+Ch\sum_{i=0}^{n_T-1}\sum_{\beta\in\mathcal{B}_{12}^b}\int_{t_i}^{t_{i+1}}(s-t_i)^{2\gamma-2}E\Big(\sup_{u\in[t_i,s]}|X(u)-X(t_i)|^2\Big)ds
\\
&+Ch\sum_{i=0}^{n_T-1}\sum_{\beta\in\mathcal{B}_{12}^b}\int_{t_i}^{t_{i+1}}(s-t_i)^{2\gamma-2}E\Big(\mathbbm{1}\{N^{(t_i,s]}\geq 1\}E\Big(\sup_{u\in[t_i,s]}(1+|X(u)|)^2\Big|\mathcal{F}_T^{\alpha}\Big)\Big)ds
\end{align*}
for all $k\in\{1,\ldots,d\}$ and $n_T\in\mathbb{N}$.
On using Theorem \ref{thm:true moment} and Lemma \ref{lem:rateMS}, we obtain
\begin{align*}
\sum_{\beta\in\mathcal{B}_1^b}E\Big(&\sup_{n\in\{1,\ldots,n_T\}}\Big|\sum_{i=0}^{n-1}\int_{t_i}^{t_{i+1}}I_{\beta}[b^k(X(\cdot),\alpha(\cdot))-b^k(X(t_i),\alpha(t_i))]_{t_i,s}ds\Big|^2\Big)
\\
\leq& Ch^{2\gamma}+Ch^{2\gamma-1}\sum_{i=0}^{n_T-1}\int_{t_i}^{t_{i+1}}P\{N^{(t_i,s]}\geq 1\}ds\leq Ch^{2\gamma}
\end{align*}
for all $k\in\{1,\ldots,d\}$ and $n_T\in\mathbb{N}$.
Now, we prove the second inequality of the lemma.
Due to Doob's martingale inequality and It\^o's isometry, we have
\begin{align*}
\sum_{\beta\in\mathcal{B}_1^\sigma}E&\Big(\sup_{n\in\{1,\ldots,n_T\}}\Big|\sum_{i=0}^{n-1}\int_{t_i}^{t_{i+1}}I_{\beta}[\sigma^{(k,j)}(X(\cdot),\alpha(\cdot))-\sigma^{(k,j)}(X(t_i),\alpha(t_i))]_{t_i,s}dW^j(s)\Big|^2\Big)
\\
\leq&C\sum_{\beta\in\mathcal{B}_1^\sigma}\sum_{i=0}^{n_T-1}\int_{t_i}^{t_{i+1}}E|I_{\beta}[\sigma^{(k,j)}(X(\cdot),\alpha(\cdot))-\sigma^{(k,j)}(X(t_i),\alpha(t_i))]_{t_i,s}|^2ds
\\
\leq&C\sum_{\beta\in\mathcal{B}_1^\sigma}\sum_{i=0}^{n_T-1}\int_{t_i}^{t_{i+1}}E|I_{\beta}[\sigma^{(k,j)}(X(\cdot),\alpha(\cdot))-\sigma^{(k,j)}(X(t_i),\alpha(\cdot))]_{t_i,s}|^2ds
\\
&+C\sum_{\beta\in\mathcal{B}_1^\sigma}\sum_{i=0}^{n_T-1}\int_{t_i}^{t_{i+1}}E\Big(\mathbbm{1}\{N^{(t_i,s]}\geq 1\}|I_{\beta}[\sigma^{(k,j)}(X(t_i),\alpha(\cdot))-\sigma^{(k,j)}(X(t_i),\alpha(t_i))]_{t_i,s}|^2\Big)ds
\\
\leq&C\sum_{\beta\in\mathcal{B}_1^\sigma}\sum_{i=0}^{n_T-1}\int_{t_i}^{t_{i+1}}E|I_{\beta}[\sigma^{(k,j)}(X(\cdot),\alpha(\cdot))-\sigma^{(k,j)}(X(t_i),\alpha(\cdot))]_{t_i,s}|^2ds
\\
&+C\sum_{\beta\in\mathcal{B}_1^\sigma}\sum_{i=0}^{n_T-1}\int_{t_i}^{t_{i+1}}E\Big(\mathbbm{1}\{N^{(t_i,s]}\geq 1\}E\Big(|I_{\beta}[\sigma^{(k,j)}(X(t_i),\alpha(\cdot))]_{t_i,s}|^2\Big|\mathcal{F}_T^{\alpha}\Big)\Big)ds
\\
&+C\sum_{\beta\in\mathcal{B}_1^\sigma}\sum_{i=0}^{n_T-1}\int_{t_i}^{t_{i+1}}E\Big(\mathbbm{1}\{N^{(t_i,s]}\geq 1\}E\Big(|I_{\beta}[\sigma^{(k,j)}(X(t_i),\alpha(t_i))]_{t_i,s}|^2\Big|\mathcal{F}_T^{\alpha}\Big)\Big)ds
\end{align*}
which by using Assumptions \ref{ass: b sigma lipschitz}, \ref{ass:A b sigma ds dw lipschitz}, Remark \ref{rem:A b sigma ds dw growth} and Lemma \ref{lem:multiple estimate 0,1} yield,
\begin{align*}
\sum_{\beta\in\mathcal{B}_1^\sigma}E&\Big(\sup_{n\in\{1,\ldots,n_T\}}\Big|\sum_{i=0}^{n-1}\int_{t_i}^{t_{i+1}}I_{\beta}[\sigma^{(k,j)}(X(\cdot),\alpha(\cdot))-\sigma^{(k,j)}(X(t_i),\alpha(t_i))]_{t_i,s}dW^j(s)\Big|^2\Big)
\\
\leq&C\sum_{i=0}^{n_T-1}\sum_{\beta\in\mathcal{B}_{1}^\sigma}\int_{t_i}^{t_{i+1}}(s-t_i)^{2\gamma-1}E\Big(\sup_{u\in[t_i,s]}|X(u)-X(t_i)|^2\Big)ds
\\
&+C\sum_{i=0}^{n_T-1}\sum_{\beta\in\mathcal{B}_{1}^\sigma}\int_{t_i}^{t_{i+1}}(s-t_i)^{2\gamma-1}E\Big(\mathbbm{1}\{N^{(t_i,s]}\geq 1\}E\Big(\sup_{u\in[t_i,s]}(1+|X(u)|)^2\Big|\mathcal{F}_T^{\alpha}\Big)\Big)ds
\end{align*}
for all $k\in\{1,\ldots,d\}$ and $j\in\{1,\ldots,m\}$ and $n_T\in\mathbb{N}$.
Hence, Theorem \ref{thm:true moment} and Lemma \ref{lem:rateMS}
give
\begin{align*}
\sum_{\beta\in\mathcal{B}_1^\sigma}E&\Big(\sup_{n\in\{1,\ldots,n_T\}}\Big|\sum_{i=0}^{n-1}\int_{t_i}^{t_{i+1}}I_{\beta}[\sigma^{(k,j)}(X(\cdot),\alpha(\cdot))-\sigma^{(k,j)}(X(t_i),\alpha(t_i))]_{t_i,s}dW^j(s)\Big|^2\Big)\leq Ch^{2\gamma}
\end{align*}
for all $k\in\{1,\ldots,d\}$ and $j\in\{1,\ldots,m\}$ and $n_T\in\mathbb{N}$.
\end{proof}

\begin{proof}[\textbf{Proof of theorem \ref{thm:main}}]
Due to  \eqref{eq:gen. scheme derivation} and \eqref{eq:gen.scheme}, 
\begin{align*}
&X^k(t_{n})-Y^k(t_{n})=X^k_0-Y^k_0+\sum_{i=0}^{n-1}\sum_{\beta\in\mathcal{A}_{\gamma}^b\setminus\tilde{\mathcal{A}}_{\gamma}^b}\int_{t_i}^{t_{i+1}}I_{\beta}[b^k(X(t_i),\alpha(t_i))-b^k(Y(t_i),\alpha(t_i))]_{t_i,s}ds\notag
\\
&+\sum_{i=0}^{n-1}\sum_{\beta\in\tilde{\mathcal{A}}_{\gamma}^b}\int_{t_i}^{t_{i+1}}I_{\beta}[b^k(X(t_i),\alpha(\cdot))-b^k(Y(t_i),\alpha(\cdot))]_{t_i,s}ds \notag
\\
&+\sum_{i=0}^{n-1}\sum_{\beta\in\mathcal{A}_{\gamma}^\sigma\setminus\tilde{\mathcal{A}}_{\gamma}^\sigma}\sum_{j=1}^m\int_{t_i}^{t_{i+1}}I_{\beta}[\sigma^{(k,j)}(X(t_i),\alpha(t_i))-\sigma^{(k,j)}(Y(t_i),\alpha(t_i))]_{t_i,s}dW^j(s)\notag
\\
&+\sum_{i=0}^{n-1}\sum_{\beta\in\tilde{\mathcal{A}}_{\gamma}^\sigma}\sum_{j=1}^m\int_{t_i}^{t_{i+1}}I_{\beta}[\sigma^{(k,j)}(X(t_i),\alpha(\cdot))-\sigma^{(k,j)}(Y(t_i),\alpha(\cdot))]_{t_i,s}dW^j(s)\notag
\\
&+\sum_{i=0}^{n-1}\sum_{\beta\in\mathcal{B}(\mathcal{A}_\gamma^b)\setminus \tilde{\mathcal{B}}_1^b}\int_{t_i}^{t_{i+1}}I_{\beta}[b^k(X(\cdot),\alpha(\cdot))]_{t_i,s}ds\notag
\\
&+\sum_{i=0}^{n-1}\sum_{\beta\in\mathcal{B}_1^b}\int_{t_i}^{t_{i+1}}I_{\beta}[b^k(X(\cdot),\alpha(\cdot))-b^k(X(t_i),\alpha(t_i))]_{t_i,s}ds\notag
\\
&+\sum_{i=0}^{n-1}\sum_{\beta\in\mathcal{B}(\mathcal{A}_\gamma^\sigma)\setminus \tilde{\mathcal{B}}_1^\sigma}\sum_{j=1}^m\int_{t_i}^{t_{i+1}}I_{\beta}[\sigma^{(k,j)}(X(\cdot),\alpha(\cdot))]_{t_i,s}dW^j(s)\notag
\\
&+\sum_{i=0}^{n-1}\sum_{\beta\in\mathcal{B}_1^\sigma}\sum_{j=1}^m\int_{t_i}^{t_{i+1}}I_{\beta}[\sigma^{(k,j)}(X(\cdot),\alpha(\cdot))-\sigma^{(k,j)}(X(t_i),\alpha(t_i))]_{t_i,s}dW^j(s) 
\end{align*}
which implies,  for all $k\in\{1,\ldots,d\}$ and $n\in\{1,\ldots,n_T\}$,
\begin{align}
E&\Big(\sup_{n\in\{1,\ldots,n'\}}|X^k(t_n)-Y^k(t_n)|^2\Big)\leq CE|X^k_0-Y^k_0|^2\notag
\\
+&CE\Big(\sup_{n\in\{1,\ldots,n'\}}\Big|\sum_{i=0}^{n-1}\sum_{\beta\in\mathcal{A}_{\gamma}^b\setminus\tilde{\mathcal{A}}_{\gamma}^b}\int_{t_i}^{t_{i+1}}I_{\beta}[b^k(X(t_i),\alpha(t_i))-b^k(Y(t_i),\alpha(t_i))]_{t_i,s}ds\Big|^2\Big)\notag
\\
+&CE\Big(\sup_{n\in\{1,\ldots,n'\}}\Big|\sum_{i=0}^{n-1}\sum_{\beta\in\tilde{\mathcal{A}}_{\gamma}^b}\int_{t_i}^{t_{i+1}}I_{\beta}[b^k(X(t_i),\alpha(\cdot))-b^k(Y(t_i),\alpha(\cdot))]_{t_i,s}ds\Big|^2\Big)\notag
\\
+&CE\Big(\sup_{n\in\{1,\ldots,n'\}}\Big|\sum_{i=0}^{n-1}\sum_{\beta\in\mathcal{A}_{\gamma}^\sigma\setminus\tilde{\mathcal{A}}_{\gamma}^\sigma}\sum_{j=1}^m\int_{t_i}^{t_{i+1}}I_{\beta}[\sigma^{(k,j)}(X(t_i),\alpha(t_i))-\sigma^{(k,j)}(Y(t_i),\alpha(t_i))]_{t_i,s}dW^j(s)\Big|^2\Big)\nonumber
\\
+&CE\Big(\sup_{n\in\{1,\ldots,n'\}}\Big|\sum_{i=0}^{n-1}\sum_{\beta\in\tilde{\mathcal{A}}_{\gamma}^\sigma}\sum_{j=1}^m\int_{t_i}^{t_{i+1}}I_{\beta}[\sigma^{(k,j)}(X(t_i),\alpha(\cdot))-\sigma^{(k,j)}(Y(t_i),\alpha(\cdot))]_{t_i,s}dW^j(s)\Big|^2\Big) \notag
\\
+&CE\Big(\sup_{n\in\{1,\ldots,n'\}}\Big| \sum_{i=0}^{n-1}\sum_{\beta\in\mathcal{B}(\mathcal{A}_\gamma^b)\setminus \tilde{\mathcal{B}}_1^b}\int_{t_i}^{t_{i+1}}I_{\beta}[b^k(X(\cdot),\alpha(\cdot))]_{t_i,s}ds \Big|^2\Big)\notag
\\
+&CE\Big(\sup_{n\in\{1,\ldots,n'\}}\Big|\sum_{i=0}^{n-1}\sum_{\beta\in\mathcal{B}_1^b}\int_{t_i}^{t_{i+1}}I_{\beta}[b^k(X(\cdot),\alpha(\cdot))-b^k(X(t_i),\alpha(t_i))]_{t_i,s}ds\Big|^2\Big)\notag
\\
+&CE\Big(\sup_{n\in\{1,\ldots,n'\}}\Big| \sum_{i=0}^{n-1}\sum_{\beta\in\mathcal{B}(\mathcal{A}_\gamma^\sigma)\setminus \tilde{\mathcal{B}}_1^\sigma}\sum_{j=1}^m\int_{t_i}^{t_{i+1}}I_{\beta}[\sigma^{(k,j)}(X(\cdot),\alpha(\cdot))]_{t_i,s}dW^j(s) \Big|^2\Big)\notag
\\
+&CE\Big(\sup_{n\in\{1,\ldots,n'\}}\Big|\sum_{i=0}^{n-1}\sum_{\beta\in\mathcal{B}_1^\sigma}\sum_{j=1}^m\int_{t_i}^{t_{i+1}}I_{\beta}[\sigma^{(k,j)}(X(\cdot),\alpha(\cdot))-\sigma^{(k,j)}(X(t_i),\alpha(t_i))]_{t_i,s}dW^j(s)\Big|^2\Big)\notag 
\end{align}
for all $k\in\{1,\ldots,d\}$ and $n'\in\{1,\ldots,n_T\}$. 
On using Doob's martingale inequality, Burkholder-Davis-Gundy inequality and H\"older's inequality, Lemmas \ref{lem:B_A_b estimate}, \ref{lem:B_A_sigma estimate}, \ref{lem:convergence B_1 estimate} and Assumption \ref{ass:initial data}, we obtain
\begin{align*}
E\Big(\sup_{n\in\{1,\ldots,n'\}}&|X^k(t_n)-Y^k(t_n)|^2\Big)\leq Ch^{2\gamma}
\\
&+C\sum_{i=0}^{n'-1}\sum_{\beta\in\mathcal{A}_{\gamma}^b\setminus\tilde{\mathcal{A}}_{\gamma}^b}\int_{t_i}^{t_{i+1}}E|I_{\beta}[b^k(X(t_i),\alpha(t_i))-b^k(Y(t_i),\alpha(t_i))]_{t_i,s}|^2ds
\\
&+C\sum_{i=0}^{n-1}\sum_{\beta\in\tilde{\mathcal{A}}_{\gamma}^b}\int_{t_i}^{t_{i+1}}E|I_{\beta}[b^k(X(t_i),\alpha(\cdot))-b^k(Y(t_i),\alpha(\cdot))]_{t_i,s}|^2ds
\\
&+C\sum_{i=0}^{n'-1}\sum_{\beta\in\mathcal{A}_{\gamma}^\sigma\setminus\tilde{\mathcal{A}}_{\gamma}^\sigma}\sum_{j=1}^m\int_{t_i}^{t_{i+1}}E|I_{\beta}[\sigma^{(k,j)}(X(t_i),\alpha(t_i))-\sigma^{(k,j)}(Y(t_i),\alpha(t_i))]_{t_i,s}|^2ds
\\
&+C\sum_{i=0}^{n-1}\sum_{\beta\in\tilde{\mathcal{A}}_{\gamma}^\sigma}\sum_{j=1}^m\int_{t_i}^{t_{i+1}}E|I_{\beta}[\sigma^{(k,j)}(X(t_i),\alpha(\cdot))-\sigma^{(k,j)}(Y(t_i),\alpha(\cdot))]_{t_i,s}|^2ds
\end{align*}
for all $k\in\{1,\ldots,d\}$ and $n'\in\{1,\ldots,n_T\}$.
By the application of Assumptions \ref{ass: b sigma lipschitz}, \ref{ass:A b sigma ds dw lipschitz}, Remark \ref{rem:A b sigma N lipschitz} and Lemmas \ref{lem:multiple estimate 0,1}, \ref{lem:multiple estimate N_mu}, we have 
\begin{align*}
E\Big(\sup_{n\in\{1,\ldots,n'\}}|X^k(t_n)&-Y^k(t_n)|^2\Big)\leq Ch^{2\gamma}
\\
&+ C\sum_{i=0}^{n'-1}\sum_{\beta\in\mathcal{A}_{\gamma}^b\cup\mathcal{A}_{\gamma}^\sigma}\int_{t_i}^{t_{i+1}}(s-t_i)^{\eta(\beta)}E\Big(\sup_{n\in\{0,1,\ldots,i\}}|X^k(t_n)-Y^k(t_n)|^2\Big)ds
\\
\leq&Ch^{2\gamma}+ Ch\sum_{i=0}^{n'-1}E\Big(\sup_{n\in\{0,1,\ldots,i\}}|X^k(t_n)-Y^k(t_n)|^2\Big)
\end{align*} 
for all $k\in\{1,\ldots,d\}$ and $n'\in\{1,\ldots,n_T\}$.
 The Gronwall’s lemma completes the proof.
\end{proof}

\bibliographystyle{amsplain}

\begin{thebibliography}{20}
\bibitem{bao2016permanence}
J. Bao and J. Shao (2016).  Permanence and extinction of regime-switching predator-prey models, \emph{SIAM J. Math. Anal.}, 48, 725-73. 

\bibitem{dareiotis2016tamed}
K. Dareiotis, C. Kumar, and S. Sabanis (2016).  On tamed Euler approximations of SDEs driven by L\'evy noise with applications to delay equations, \emph{SIAM J. Numer. Anal.},  54-3, 1840-1872.

\bibitem{kloeden1992numerical}
P. E. Kloeden and E. Platen (1992).
\emph{ Numerical Solution of Stochastic Differential Equations}, Springer Berlin Heidelberg.

\bibitem{kumar2020tamedmilstein}
C. Kumar and T. Kumar (2020). On explicit tamed milstein-type scheme for stochastic differential equation with markovian switching.
\emph{J. Comput. Appl. Math.}, 377, 112917.

\bibitem{kumar2021milstein}
C. Kumar and T. Kumar (2021). A note on explicit milstein-type scheme for stochastic differential equation with markovian switching, \emph{J. Comput. Appl. Math.}, 395, 113594.

\bibitem{mao2006stochastic}
X. Mao and C. Yuan (2006).  \emph{Stochastic Differential Equations with Markovian Switching}, Imperial College Press, London. 

\bibitem{nguyen2017milstein}
S. L. Nguyen, T. A. Hoang, D. T. Nguyen and G. Yin (2017). Milstein-type procedures for numerical solutions of stochastic differential equations with Markovian switching, \emph{SIAM J. Numer. Anal.},  55 (2),  953-979.

\bibitem{nguyen2018tamed}
D. T. Nguyen, S. L. Nguyen,  T. A. Hoang and G. Yin (2018). Tamed-Euler method for hybrid stochastic differential equations with Markovian switching, \emph{Nonlinear Anal. Hybrid Syst.}, 30, 14-30.

\bibitem{nguyen2017pathwise}
S. L. Nguyen and G. Yin (2012).  Pathwise convergence rates for numerical solutions of Markovian switching stochastic differential equations, \emph{Nonlinear Anal. Real World Appl.}, 13(3), 1170-1185.

\bibitem{mao2006hierarchical}
S. P. Sethi and Q. Zhang (1994).  \emph{Hierarchical Decision Making in Stochastic Manufacturing Systems}, Birkh\"auser, Boston.

\bibitem{yin1994hybrid}
G. Yin and C. Zhu (2010).  \emph{Hybrid Switching Diffusions, Properties and Applications}, Springer, New York.

\bibitem{yuan2004convergence}
C. Yuan and X. Mao (2004). Convergence of the euler--maruyama method for stochastic differential equations with markovian switching, \emph{Math. Comput. Simulation}, 64(2), 223-235.

\bibitem{zhang1998nonlinear}
Q. Zhang (1998). Nonlinear filtering and control of a switching diffusion with small observation noise, \emph{SIAM J. Control Optim.}, 36, 1638-1668.

\bibitem{zhang2001stock}
Q. Zhang (2001). Stock trading: An optimal selling rule, \emph{SIAM J. Control Optim.}, 40, 64-87.
\end{thebibliography}

\end{document}